\documentclass[11pt]{amsart}

\usepackage{amsmath}
\usepackage{amssymb,amsthm,amsfonts}
\usepackage{amsrefs}
\usepackage[mathscr]{euscript}
\usepackage[margin=1.5cm]{geometry}
\usepackage{graphicx}
\usepackage{wrapfig}
\usepackage[usenames,dvipsnames,svgnames,table]{xcolor}

\renewcommand{\mathcal}{\mathscr}

\def\R {\mathbb{R}}

\def\S {\mathbb{S}}

\def\S {\mathbb{S}}

\renewcommand{\epsilon}{\varepsilon}
\newcommand{\eps}{\varepsilon}
\newcommand{\e}{\varepsilon}

\renewcommand{\leq}{\leqslant}
\renewcommand{\le}{\leqslant}
\renewcommand{\geq}{\geqslant}
\renewcommand{\ge}{\geqslant}

\newcommand{\Per}{\mathrm{Per}}

\newtheorem{proposition}{Proposition}[section]
\newtheorem{theorem}[proposition]{Theorem}
\newtheorem{corollary}[proposition]{Corollary}
\newtheorem{lemma}[proposition]{Lemma}
\theoremstyle{definition}
\newtheorem{definition}[proposition]{Definition}

\newtheorem{remark}[proposition]{Remark}
\numberwithin{equation}{section}

\begin{document}

\title[Nonlocal curvature flows]{Fattening and nonfattening phenomena\\
for planar nonlocal curvature flows}

\thanks{
{\em Annalisa Cesaroni}:
Dipartimento di Scienze Statistiche,
Universit\`a di Padova, Via Battisti 241/243, 35121 Padova, Italy. {\tt annalisa.cesaroni@unipd.it}
\\
{\em Serena Dipierro}:
Department of Mathematics
and Statistics,
University of Western Australia,
35 Stirling Hwy, Crawley WA 6009, Australia,
Dipartimento di Matematica, Universit\`a di Milano,
Via Saldini 50, 20133 Milan, Italy.
{\tt serydipierro@yahoo.it}\\
{\em Matteo Novaga}: Dipartimento di Matematica,
Universit\`a di Pisa, 
Largo Pontecorvo 5, 56127 Pisa,
Italy. {\tt matteo.novaga@unipi.it}\\
{\em Enrico Valdinoci}:
Department of Mathematics
and Statistics,
University of Western Australia,
35 Stirling Hwy, Crawley WA 6009, Australia,
School of Mathematics
and Statistics,
University of Melbourne, 813 Swanston St, Parkville VIC 3010, Australia,
Dipartimento di Matematica, Universit\`a di Milano,
Via Saldini 50, 20133 Milan, Italy. {\tt enrico@math.utexas.edu}}

\author[A. Cesaroni]{Annalisa Cesaroni}
\author[S. Dipierro]{Serena Dipierro}
\author[M. Novaga]{Matteo Novaga}
\author[E. Valdinoci]{Enrico Valdinoci}



\maketitle

\begin{abstract}
We discuss fattening phenomenon for the evolution of sets according to their nonlocal curvature. More precisely,
we consider  a class of generalized curvatures which correspond to the first variation of suitable nonlocal perimeter functionals, defined  in terms of an interaction kernel $K$,
which is  symmetric, nonnegative,
possibly singular at the origin, and satisfies appropriate integrability conditions. 

We prove a general result  about uniqueness of the geometric evolutions
starting from regular sets with positive $K$-curvature in $\R^n$ 
and  we discuss the fattening phenomenon in $\R^2$ for the evolution
starting from the cross, showing  that this phenomenon is very sensitive
to the strength of the interactions. 
As a matter of fact, we show that the fattening of the cross occurs
for kernels with sufficiently large mass near the origin,
while for kernels that are sufficiently weak near the origin
such a fattening phenomenon does not occur.

We also provide some further results in the case of the fractional
mean curvature flow, showing
that strictly starshaped sets in $\R^n$ have a unique geometric evolution.

Moreover, we exhibit two illustrative examples in $\R^2$ of closed nonregular curves,
the first with a Lipschitz-type singularity and the second with a cusp-type
singularity, given by two tangent circles of equal radius,
whose evolution develops fattening in the first case,
and is uniquely defined in the second, thus remarking
the high sensitivity of the fattening phenomenon
in terms of the regularity of the initial datum.
The latter example is in 
striking contrast to the classical case of the (local) curvature flow,
where two tangent circles always develop fattening.

As a byproduct of our analysis, we provide
also a simple proof of the fact that the cross in $\R^2$ is not
a $K$-minimal set for the nonlocal perimeter functional associated to $K$.
\end{abstract}

\tableofcontents

\section{Introduction}
In this paper we are interested in the analysis of 
the fattening phenomenon for evolutions of sets   according to nonlocal curvature flows. 
Fattening is a particular kind of singularity which arises in the evolution of boundaries by their (local or nonlocal) curvatures  
and more generally in geometric evolution of manifolds and is related to nonuniqueness of geometric solutions to the flow. 
Fattening phenomenon has been studied for mean curvature flow since long time
and a complete characterization of 
initial data which develop fattening is still missing. In the case of the plane, it is known that  smooth compact level curves   never develop an interior, due 
to a result by Grayson on the evolution of regular compact curves. This result is no more valid for fractional mean curvature flow in the plane, as proved recently in
\cite{MR3778164}. We recall
that examples of fattening of nonregular or noncompact curves in the plane 
for the mean curvature flow have been given in~\cites{MR1100206, MR1189906, bell},
where,
in particular, the fattening of the evolution starting from the cross is proved.  
Finally nonfattening for strictly starshaped initial data is proved in \cite{soner}, whereas
nonfattening of 
convex and mean convex initial data is proved in \cite{bss}, see also \cite{MR2208291} and  \cite{bcl}.
\medskip

In this paper we start the analysis of the fattening phenomenon (mostly in the plane)
for general nonlocal curvature flows. This problem has not yet been
considered in the literature apart from the result in~\cite{MR3713894} about
nonfattening for convex initial data under 
fractional mean curvature evolution in any space dimension.  

Here we will show that some results which are true for the mean curvature flow
are still valid, such as nonfattening for
regular initial data with positive curvature or
strictly starshaped initial data.

Nevertheless, in 
general,  some different behaviors with respect to the mean curvature flow arise, 
due to the fact that the fattening phenomenon is very sensitive to 
the strength of the nonlocal interactions. We discuss in particular
the evolution starting from the cross in the plane, which develops fattening only 
if the interactions are sufficiently strong. Moreover,
we show an example of a closed  curve with positive
curvature which fattens, and an example of a 
closed curve whose evolution by fractional mean curvature
flow does not present fattening, differently from the case of the
evolution by mean curvature flow. 
\medskip 

We now introduce the mathematical setting in which we work.
Given an initial set~$E_0\subset\R^n$, we 
define its evolution~$E_t$ for~$t>0$ according to
a nonlocal curvature flow as follows:
the velocity at a point~$x\in \partial E_t$ is given by
\begin{equation}\label{kflow} 
\partial_t x\cdot \nu=-H^K_{E_t}(x)
\end{equation}  
where $\nu$ is the outer normal at $\partial E_t$ in $x$.
The quantity~$H_E^K(x)$ is the
$K$-curvature of~$E$ at $x$, which is defined in the forthcoming formula~\eqref{CURVA}. 
More precisely, we take a function~$K:\R^n\setminus\{0\}\to[0,+\infty)$
which is a rotationally invariant kernel, 
namely
\begin{equation}\label{ROTAZION}
K(x)=K_0(|x|),
\end{equation}
for some~$K_0:(0,+\infty) \to[0,+\infty)$. We assume that
\begin{equation}\label{ROTAZION2}
\min\{1, |x|\}\, K(x) \in L^1(\R^n),\qquad \text{ i.e. }\int_0^1 \rho^{n}\,K_0(\rho)\,d\rho+\int_1^{+\infty} \rho^{n-1}\,K_0(\rho)\,d\rho
<+\infty.\end{equation}
Given~$E\subset\R^n$ and~$x\in\partial E$ we define  the $K$-curvature of~$E$
at~$x$, defined by
\begin{equation}\label{CURVA} H^K_E(x):=\lim_{\e\searrow0}
\int_{\R^n\setminus B_\e(x)}\Big(
\chi_{\R^n\setminus E}(y)-\chi_E(y)\Big)\,K(x-y)\,dy,\end{equation}
where, as usual,
$$ \chi_E(y):=\left\{
\begin{matrix}
1 & {\mbox{ if }}y\in E,
\\ 0 & {\mbox{ if }}y\not\in E.
\end{matrix}
\right. $$
We point out that~\eqref{ROTAZION} is a very mild integrability assumption,
compatible with the structure of nonlocal minimal surfaces (see e.g. condition~(1.5)
in \cite{csv}) and which fits the requirements in ~\cites{MR2487027, MR3401008}
in order to have existence and uniqueness for the
level set flow associated to~\eqref{kflow} (see Appendix~\ref{viscosec} for the details
about this matter).

Furthermore, when $K(x)= \frac{1}{|x|^{n+s}}$ for some $s\in (0,1)$, we will denote the 
$K$-curvature of a set~$E$ at a point~$x$ as~$H^s_E(x)$, and we indicate it as
the fractional mean curvature of the set~$E$ at~$x$.

While the setting in~\eqref{CURVA} makes clear sense for sets
with~$C^{1,1}$-boundaries,
as customary we also use the notion of $K$-curvatures for sets which
are locally the graphs of continuous functions: in this case, the $K$-curvature
may be also infinite and the definition is in the sense of viscosity
(see \cites{MR2487027, MR3401008} and Section~5 in~\cite{MR2675483}).  

We observe that 
the curvature defined in \eqref{CURVA}  is the the first variation of the following nonlocal perimeter functional, see \cites{cn, m}, 
\begin{equation}\label{per} \Per_K(E):=  \int_{E}\int_{\R^n\setminus E} K(x-y)\,dx\,dy, 
\end{equation} 
and so the geometric evolution law in~\eqref{kflow} can be interpreted as the $L^2$
gradient flow of this perimeter functional, 
as proved in~\cite{MR3401008}. 

The existence and uniqueness
of solutions for the $K$-curvature flow in~\eqref{kflow} in the
viscosity sense have been investigated 
in~\cite{MR2487027} by
introducing  the level set formulation of the geometric evolution problem~\eqref{kflow} and 
a proper notion of viscosity solution.
We refer to \cite{MR3401008}
for a general framework for the analysis via the level set formulation of a 
wide class of local and nonlocal translation-invariant geometric flows.

The level set flow associated to~\eqref{kflow} can be defined as follows. 
Given an initial set $E\subset \R^n$ and $C:=\partial E$,
we choose a bounded Lipschitz continuous function 
$u_E:\R^n\to \R$ such that 
\begin{eqnarray*}&&
C=\{x\in\R^n  \text{ s.t. } u_E(x)=0\}=\partial\{x\in\R^n\text{ s.t. }  u_E(x)\geq 0\}\\
{\mbox{and }} && E=\{x\in\R^n\text{ s.t. } u_E(x)\geq 0\}. 
\end{eqnarray*}
Let also~$u_E(x,t)$ be the viscosity solution of the following nonlocal parabolic problem
\begin{equation}\label{levelset}
\begin{cases}
\partial_t u(x,t)+|Du(x,t)| H^K_{\{y| u(y,t)\geq u(x,t)\}}(x)=0,\\
u(x,0)= u_E(x).
\end{cases} 
\end{equation} 
Then the level set flow of $C$ is given by 
\begin{equation}\label{sigmaet} \Sigma_E(t):=\{x\in\R^n\text{ s.t. }  u_E(x,t)=0\}. \end{equation} 
We associate to this level set the outer and inner flows defined as follows:
\begin{equation}
\label{outin} E^+(t):= \{x\in\R^n\text{ s.t. }  u_E(x,t)\geq 0\} \qquad
{\mbox{and}} \qquad E^-(t):= \{x\in\R^n\text{ s.t. }  u_E(x,t)>0\}.
\end{equation} 
We observe that the equation in \eqref{levelset} is geometric, so  
if we replace the initial condition with any function $u_0$ with the same level sets $\{u_0 \geq 0\}$ and $\{ u_0 > 0 \}$, the evolutions
$E^+(t)$ and $E^-(t)$ remain the same.
For more details, we refer to Appendix~\ref{viscosec}. 
\medskip

The $K$-curvature flow has been recently studied from different perspectives,
in particular the case fractional mean curvature flow,
taking into account geometric features such as conservation of the positivity of
the fractional mean curvature,
conservation of convexity and formation of neckpinch singularities,
see~\cites{SAEZ, MR3713894, MR3778164}. 

In this paper,  we analyze
the possible lack of uniqueness for the geometric evolution, i.e.
the situation in which~$\partial E^+(t)\neq \partial E^-(t)$, 
in terms of the fattening properties of the zero level set of the viscosity solutions.
To this end, we give the following definition:

\begin{definition}\label{fatdef} We say that
fattening occurs at time~$t>0$ if the set~$\Sigma_E(t)$, defined in~\eqref{sigmaet},
has nonempty interior, i.e. 
\[{\rm{int}}(E^+(t)\setminus E^-(t))\neq \varnothing.\]
\end{definition}   

We point out that in \cite[Section 6]{MR3713894},
in the case of fractional (anisotropic) mean  curvature flow in any dimension,
it has been proved that 
if the initial set~$E\subseteq \R^n$ is convex, then
the evolution remains convex for all~$t>0$ and~$E^+(t)=\overline{E^{-}(t)}$, so
fattening never occurs.

We start with a result
about nonfattening of bounded regular sets with positive $K$-curvature
(for the classical case of 
the mean curvature flow, see \cites{bss, MR2208291, bcl}).

\begin{theorem}\label{positive} Let ~\eqref{ROTAZION} and~\eqref{ROTAZION2} hold. 
Let $E\subset\R^n$ be a compact set of class $C^{1,1}$ and we assume that
there exists $\delta>0$ such that
\begin{equation}\label{g7g7fugvasfegdu18iqey20}
{\mbox{$H_E^K(x)\geq \delta$ for every $x\in \partial E$.}}\end{equation}
Then $\Sigma_{ E}(t)$ has empty interior for every $t$. 
\end{theorem} 

We point out that, to get the result in Theorem~\ref{positive},
the assumption on the regularity of the sets cannot be completely dropped:  
indeed in the forthcoming Theorem~\ref{FAT3} we will provide an
example of bounded set in the plane, with a ``Lipschitz-type'' singularity and
with positive $K$-curvature, which develops fattening.

\subsection{Evolution of the cross} \label{EVOCRO:S}
We consider now the cross in  $\R^2$, i.e. 
\begin{equation}\label{CROSS}
{\mathcal{C}}:=\big\{ x=(x_1,x_2)\in\R^2 {\mbox{ s.t. }} |x_1|\geq |x_2|\big\}.\end{equation} 
It is well known, see \cite{MR1100206}, that the evolution of the cross according to the curvature flow immediately 
develops fattening for $t>0$. So, an interesting question is if the same phenomenon appears also for general nonlocal curvature flows as \eqref{kflow}, 
for kernels which satisfy ~\eqref{ROTAZION} and~\eqref{ROTAZION2}. 
We show that  actually
the fattening feature in nonlocal curvature flows is very sensible to the specific properties of the kernel
since it depends on the strength of the interactions: we identify in particular two classes of kernels, 
giving fattening of the cross  in the first class, i.e. for kernels which
satisfy \eqref{KERSI}, \eqref{UNIMPIU} below, and  nonfattening of the cross in
the second class, i.e. for kernels which satisfy \eqref{NOFHY} below.  
 
\begin{remark}\upshape\label{remcroce} 
Recalling the notation in \eqref{sigmaet}, we observe that 
\begin{equation}\label{9iwkdJSJJ293Si}
\big\{ x=(x_1,x_2)\in\R^2 {\mbox{ s.t. }} |x_1|= |x_2|\big\}\subseteq \Sigma_{\mathcal{C}}(t)\qquad
{\mbox{ for all }} t>0.\end{equation}
Indeed, up to a rotation of coordinate system, we write~$\mathcal{C}=\{(y_1,y_2)\in\R^2\text{ s.t. }y_1y_2\geq 0\}$. Define a bounded  Lipschitz function $u_0$ such that
$u_0(y_1,y_2)=u_0(-y_1,-y_2)=-u_0(-y_1,y_2)=-u_0(y_1,-y_2)$, 
and such that $\mathcal{C}= \{(y_1,y_2)\in\R^2\text{ s.t. }u_0(y_1,y_2)\geq 0\}$.
Then the solution to \eqref{levelset} with initial condition $u_0$ satisfies
$$u(y_1,y_2, t)=u(-y_1,-y_2,t)=-u(-y_1,y_2,t)=-u(y_1,-y_2,t),$$
see Appendix \ref{viscosec}.
In particular this implies that
$\big\{ (y_1,y_2)\in\R^2 {\mbox{ s.t. }} y_1y_2=0\big\}\subseteq \big\{ (y_1,y_2)\in\R^2 {\mbox{ s.t. }}u(y_1,y_2, t)=0\big\}=
 \Sigma_{\mathcal{C}}(t)$, that is~\eqref{9iwkdJSJJ293Si} once we rotate back. 
\end{remark}

We introduce the function 
\begin{equation}\label{defpsi} \Psi(r):=\int_{B_{r/4}(7r/4,0)} K(x)\,dx.\end{equation}
In our framework, the function~$\Psi(r)$ plays a crucial role in quantitative
$K$-curvature estimates. 
Notice that when~$K(x)=\frac{1}{|x|^{2+s}}$ with~$s\in(0,1)$, the function~$\Psi(r)$
reduces, up to multiplicative constants, to $\frac1{r^s}$.

We define, for any~$r>0$, the ``perturbed cross''
\begin{equation}\label{PCROSS}
{\mathcal{C}}_r := [-r,r]^2 \cup {\mathcal{C}}\subseteq \R^2.
\end{equation}
Then, we have:

\begin{proposition}\label{PROP}
Assume that~\eqref{ROTAZION} and~\eqref{ROTAZION2} hold true in $\R^2$.
Then,
we have that  
\begin{equation}\label{PRO1} H^K_{ {\mathcal{C}}_r }(p)\leq 0 \end{equation}
for any~$p\in\partial{ {\mathcal{C}}_r }$.
Also, for any~$t\in[-r,r]$,
\begin{equation}\label{PRO2} H^K_{ {\mathcal{C}}_r }(t,r)\le-2\Psi(r).\end{equation}
\end{proposition}

Proposition~\ref{PROP} provides the cornerstone to detect the fattening
phenomenon of the $K$-curvature flow emanating from the cross, when the kernel $K$ satisfies 
\begin{equation}\label{KERSI}
\int_{0}^1 \frac{d\rho}{\Psi(\rho)}<+\infty. 
\end{equation} We will need also the following technical assumption: there exists $r_0>0$ such that  for all $r\in (0,r_0)$, 
\begin{equation}\label{UNIMPIU}
\inf_{p\in B_{3\sqrt{2}\,r}}\int_{B_{r/4}(3r/4,0)-p} K(x)\,dx >0.
\end{equation} This assumption is trivially satisfied if $K>0$ in $B_{(3\sqrt{2}+1)r_0}$. 

Indeed in this case, 
we have that, for short times,
the set~$\Sigma_{ {\mathcal{C}} }(t)$ contains a ball centered at the origin
(see\footnote{The pictures of this paper have just a
qualitative and exemplifying purpose, to favor the intuition and
make the reading simpler. They are sketchy,
not quantitatively accurate
and they are not the outcome of any rigorous simulation.}
Figure~\ref{PPP}), according to
the following result:

\begin{figure}
    \centering
    \includegraphics[width=8cm]{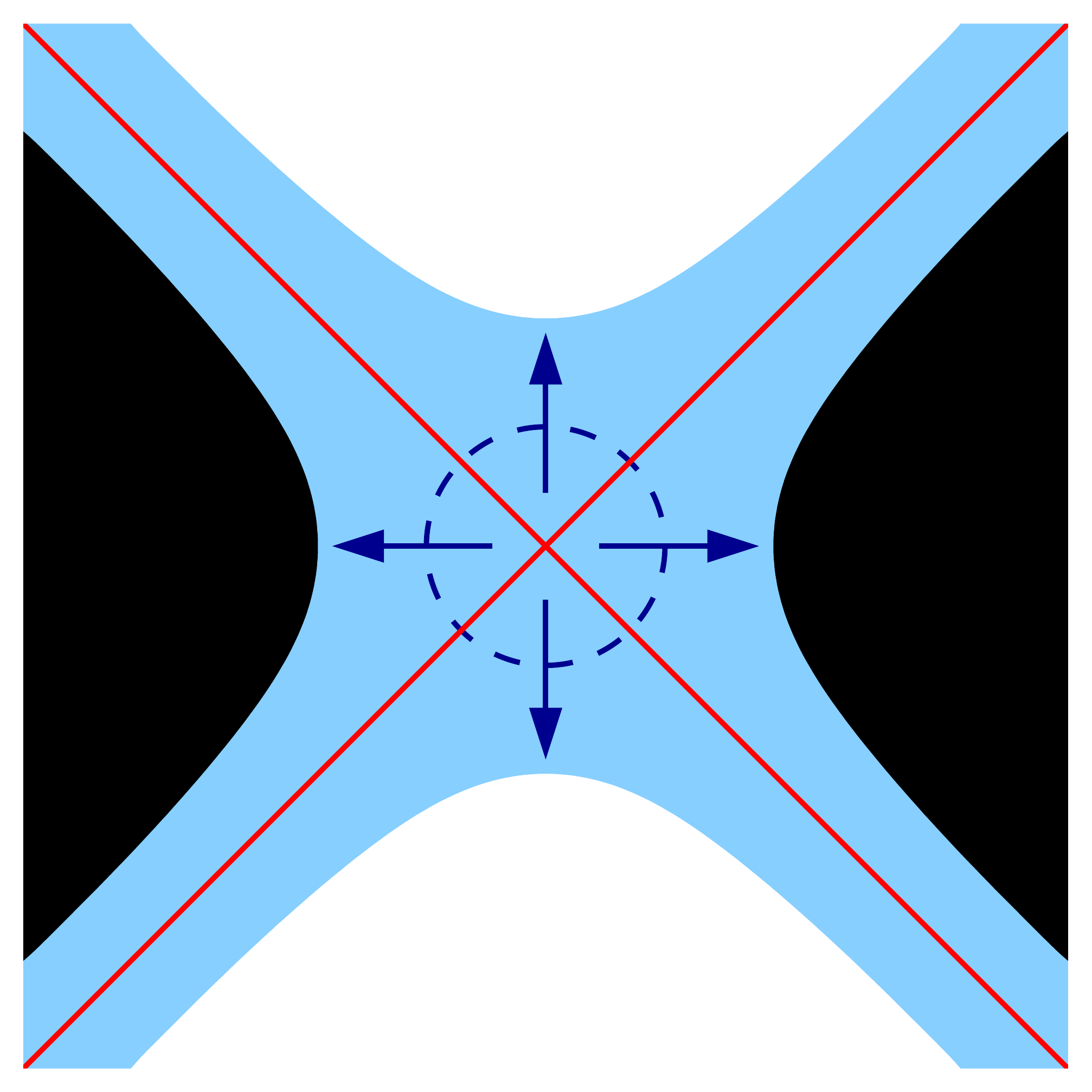}
    \caption{\em {{The fattening phenomenon described in Theorem~\ref{FAT1}.}}}
    \label{PPP}
\end{figure}

\begin{theorem}\label{FAT1}
Assume that~\eqref{ROTAZION}, ~\eqref{ROTAZION2},  \eqref{KERSI}  and \eqref{UNIMPIU} 
hold true. 
For~$r\in(0,1)$,   we define 
\begin{equation}\label{LAMBDA DEF} \Lambda(r):=\int_0^r \frac{d\rho}{\Psi(\rho)}.\end{equation}
Then, there exists $T>0$ such that 
\begin{equation} \label{9238475y6uedjidftgeri}
B_{r(t)}\subset \Sigma_{ {\mathcal{C}} }(t) \end{equation}
for any~$t\in(0,T)$, where~$r(t)$ is defined implicitly by
\begin{equation}\label{rdit}
\Lambda(r(t))=t.\end{equation}
\end{theorem}

We notice that the setting in~\eqref{LAMBDA DEF} is well defined
in view of the structural assumption in~\eqref{KERSI}
and~$\Lambda(r)$, as defined in~\eqref{LAMBDA DEF}, is strictly increasing,
which makes the implicit definition in~\eqref{rdit} well posed.

\begin{remark}\upshape 
We  point out that the structural assumptions in~\eqref{ROTAZION2}
and~\eqref{KERSI}
are satisfied by kernels of the form~$K(x)=
\frac{1}{|x|^{2+s}}$ for some $s\in (0,1)$, 
or more generally by kernels such that 
\begin{equation}\label{0r438HA340IASJdfjggtt0454} K\in L^1(\R^2\setminus B_1)\qquad \text{ and }\qquad
 \frac{1}{C\,|x|^\alpha}\leq K(x)\leq \frac{C}{|x|^\beta},\quad
{\mbox{ with~$\alpha>1$, $\beta<3$, $C\ge1$, for any $x\in B_1$}}.
\end{equation}
Indeed, the upper bound for~$K$ in~\eqref{0r438HA340IASJdfjggtt0454}
plainly implies \eqref{ROTAZION2}. 
Moreover,
the lower bound for~$K$ in~\eqref{0r438HA340IASJdfjggtt0454}
 implies that
\[ \Psi(r)=\int_{B_{r/4}(7r/4,0)} K(x)\,dx\geq  \int_{B_{r/4}(7r/4,0)}
\frac{1}{|x|^\alpha}\,dx\geq \frac{1}{(2r)^{\alpha}}|B_{r/4}|=C_0\, r^{2-\alpha}\]
where $C_0>0$ is independent of $r$, and this yields~\eqref{KERSI}. 
Finally as for \eqref{UNIMPIU}, we observe that it is trivially satisfied.

Note that $r(t)$ defined in  \eqref{rdit} satisfies $r(t)\geq C_0 t^{\frac{1}{\alpha-1}}$, in particular, 
in the case $K(x)=\frac{1}{|x|^{2+s}}$,   $r(t)$ is proportional
to~$t^{\frac1{1+s}}$. \end{remark} 
\medskip

As a counterpart of Theorem~\ref{FAT1}, we show that the fattening phenomenon does
not occur in straight crosses when the interaction kernel has sufficiently
strong integrability properties. Namely, we have that:

\begin{theorem}\label{NOFAT}
Assume~\eqref{ROTAZION} and~\eqref{ROTAZION2}. Suppose also that
\begin{equation}\label{NOFHY}
\begin{split}&
{\mbox{$K_0\le K_1$, with~$K_1$ nonincreasing
and}}
\\&
\Phi(r):=\int_{ [-r,r]\times\R} K_1(|x|)\,dx<+\infty,
\end{split}
\end{equation}
for any~$r>0$, and that 
\begin{equation}\label{HIG}
\lim_{\delta\searrow0} \int_{\delta}^{1} \frac{d\tau}{\Phi(\tau)}=+\infty.
\end{equation}
Then
\begin{equation}\label{8G}
{\mbox{the evolution of~${\mathcal{C}}$ under the $K$-curvature flow
coincides with~${\mathcal{C}}$ itself.}}
\end{equation}
\end{theorem}

\begin{remark} \upshape We notice that conditions~ \eqref{ROTAZION2}, \eqref{NOFHY}
and~\eqref{HIG} are satisfied by kernels $K$ such that  $K_0$
is nonincreasing, and which satisfy  
\begin{equation}\label{23456780r438HA340IASJdfjggtt0454}
K\in L^1(\R^2\setminus B_1)\qquad \text{ and }\qquad K(x)\leq \frac{C}{|x|^\alpha},\quad
{\mbox{ with~$\alpha\in(0,1]$, $C>0$, for any $x\in B_1$}}.
\end{equation}
Indeed,
we observe first that  in this case  \eqref{ROTAZION2} is automatically satisfied. 
Moreover,  from \eqref{23456780r438HA340IASJdfjggtt0454},
we can take~$K_1:=K_0$ in \eqref{NOFHY} and have that
\begin{eqnarray*}
\Phi(r)&=&\int_{ [-r,r]\times \R} K_0(|x|)\,dx
\\ &\le&
\int_{ B_r} \frac{C}{|x|^\alpha}\,dx+ \int_{[-r,r]^2\setminus B_r} K_0(|x|)\,dx
+\int_{ [-r,r]\times ((-\infty,-r]\cup[r,+\infty))} K_0(|x|)\,dx
\\&
\leq& C r^{2-\alpha} +4 r\int_r^{+\infty} K_0(x_2)\,dx_2 
\\& \leq& C r^{2-\alpha} +C r\left(\int_r^1 \frac{dx_2}{x_2^\alpha} +1\right)\\
&\le& Cr|\log r|,
\end{eqnarray*} 
up to renaming~$C>0$, and
so~\eqref{HIG} is satisfied. 

We also observe that
condition~\eqref{23456780r438HA340IASJdfjggtt0454} is somewhat
complementary to~\eqref{0r438HA340IASJdfjggtt0454}.
\end{remark} 
\medskip 

\subsection{A remark on $K$-minimal cones} 
As a byproduct of the results that we discussed in Subsection~\ref{EVOCRO:S},
we observe that actually the cross is not a $K$-minimal set for the
$K$-perimeter in $\R^2$, obtaining an alternative (and more general)
proof of a result 
discussed in Proposition~5.2.3
of~\cite{MR3469920} for the fractional perimeter
(see~\cite{MR3090533} for a full regularity theory
of fractional minimal cones in the plane).

For this, we define
\begin{equation}\label{locper}\Per_K(E, B_R)
:=  \int_{E\cap B_R}\int_{\R^2\setminus E} K(x-y)\,dx\,dy +
\int_{E\setminus B_R }\int_{B_R\setminus E} K(x-y)\,dx\,dy.
 \end{equation} 
Then,
we 
say that $E$ is a minimizer for $\Per_K$ in the ball $B_R$ if
\[\Per_K(E, B_R)\leq \Per_K(F,B_R)\] for every measurable set~$F$ such 
that $E\setminus B_R=F\setminus B_R$.

Also, a measurable set $E\subset\R^2$ is said to be
$K$-minimal for the $K$-perimeter if it is a  minimizer for $\Per_K$
in every ball $B_R$.  Then, we have:

\begin{proposition}\label{mincroce}
Let~\eqref{ROTAZION} and
\eqref{ROTAZION2}
hold, and
assume that $K$ is not identically zero.
Then $\mathcal{C}\subseteq \R^2$, as defined in \eqref{CROSS}, is not $K$-minimal for the $K$-perimeter.
\end{proposition} 

\medskip
 \subsection{Fractional curvature evolution of starshaped sets} 
 
Now we   restrict ourselves to the case of homogeneous kernels $K$, i.e. we consider the case  (up to multiplicative constants) in which
\begin{equation}\label{NUCLEONORMA}
K_0(r)=\frac{1}{r^{n+s}}, \qquad{\mbox{with }}s\in(0,1).
\end{equation}

We start by observing that strictly starshaped sets never fattens, similarly as for the (local) curvature flow (see \cite{soner}).
A similar result has also been observed in \cite[Remark 6.4]{MR3713894}. 

\begin{proposition}\label{strictstar} Assume \eqref{NUCLEONORMA}. 
Let $\S^{n-1}=\{\omega\in \R^n{\mbox{ s.t. }} |\omega|=1\}$, $
f:\S^{n-1}\to (0, +\infty)$
be a continuous positive function and  $E\subset \R^n$ be such that  
\begin{equation}\label{strictstar1}
 E=\{0\}\cup \left\{x\in\R^n\text{ s.t. }   x\neq 0,  |x|\leq f\left(\frac{x}{|x|}\right)\right\}.
\end{equation}
Then, the set~$\Sigma_{E}(t)$ has empty interior for all $t>0$.
\end{proposition}

Now we restrict ourselves to the case of the plane, so $n=2$.  We show that in general,  for  starshaped sets $E$ which do not satisfy \eqref{strictstar1}, 
we can expect either fattening or nonfattening. We provide 
 two  different examples of such sets in $\R^2$, which are particularly interesting in our opinion, since 
 they  model two different type of singularities 
 that can arise in the geometric evolution of closed curves in $\R^2$,
that is the ``Lipschitz-type'' singularity,  and the``cusp singularity''.  
 The first example  is 
the ``double droplet'' in Figure~\ref{Goev0}, namely
\begin{equation}\label{DDR} {\mathcal{G}}:= {\mathcal{G}}_+\cup {\mathcal{G}}_-\subseteq\R^2,\end{equation}
where ${\mathcal{G}}_+$ is the convex hull of~$B_1(-1,1)$ with the origin,
and~${\mathcal{G}}_-$ the convex hull of~$B_1(1,-1)$ with the origin. 
The second example is given by two tangent balls 
\begin{equation}\label{1.26bis}
{\mathcal{O}}:= B_1(-1,0)\cup B_1(1,0)\subseteq \R^2.\end{equation}
We prove that fattening phenomenon occurs in the first case, whereas it does not occur in
 the second.  It is also  interesting  to observe that the evolution of  $\mathcal{O}$  by curvature flow
 immediately develops fattening, see \cite{bell}.
 
\begin{figure}
    \centering
    \includegraphics[width=8cm]{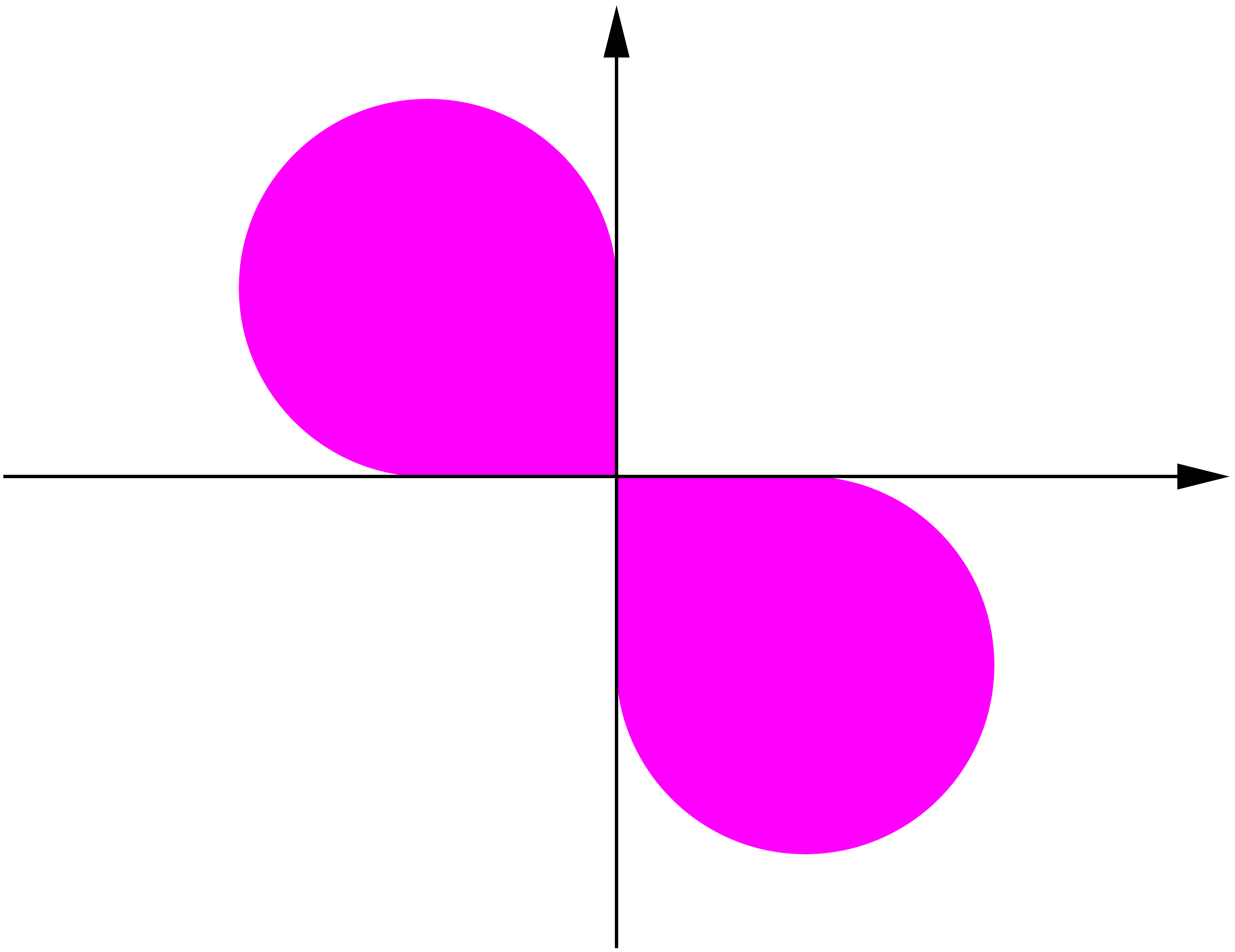}
    \caption{\em {{The double droplet~${\mathcal{G}}$.}}}
    \label{Goev0}
\end{figure}

We start by considering the evolution of the set~${\mathcal{G}}$
defined in~\eqref{DDR}. Note that this provides an example of bounded set with positive $K$-curvature
(being contained in a cross with zero $K$-curvature), whose evolution develops
fattening near the origin, as sketched in
Figure~\ref{Goev} and detailed in the following statement.

\begin{theorem}\label{FAT3}  Assume \eqref{NUCLEONORMA} with $n=2$. 
Then there exist~$\hat c $, $T>0$ such that 
\begin{equation} \label{9238475y6uedjidftgeri:T3}
B_{r(t)}\subset \Sigma_{ {\mathcal{G}} }(t) \end{equation}
for any~$t\in(0,T)$, where
\begin{equation}\label{rdit:T3}
r(t):=\hat c t^{1/(1+s)}.\end{equation}
\end{theorem}
\begin{remark}\upshape
The same result as in Theorem \ref{FAT3} holds more generally for kernels $K_0$ which satisfy ~\eqref{ROTAZION}, ~\eqref{ROTAZION2}, \eqref{KERSI} and 
\begin{equation}\label{NOME}
\frac{\underline{a}}{r^{2+s}}\le K_0(r)\le\frac{\overline{a}}{r^{2+s}}\qquad {\mbox{ for all }} r>0
\end{equation}
for some suitable~${\overline{a}}\ge {\underline{a}}>0$.
\end{remark} 
\begin{figure}
    \centering
    \includegraphics[width=8cm]{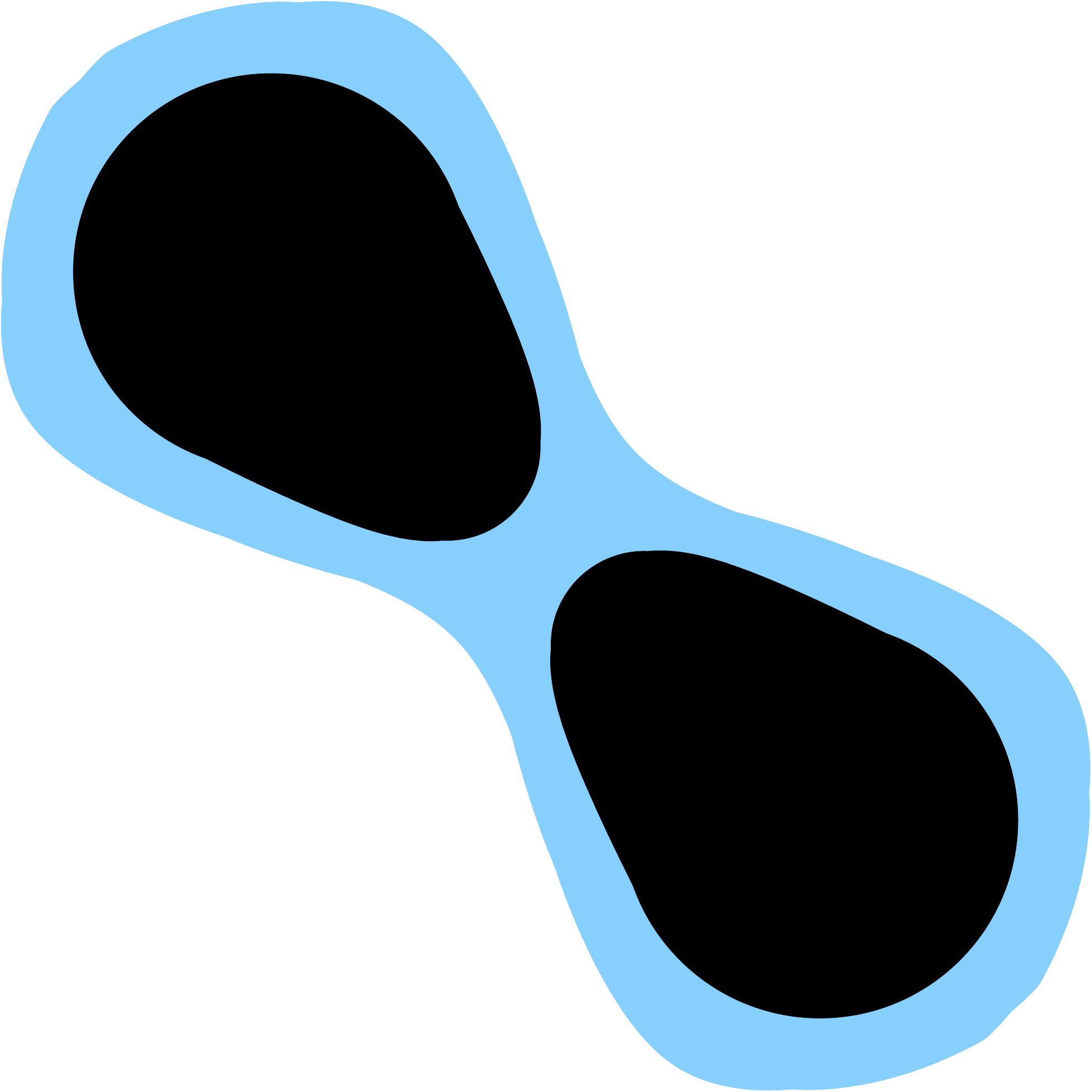}
    \caption{\em {{The fattening phenomenon described in Theorem~\ref{FAT3}.}}}
    \label{Goev}
\end{figure}

We now consider the case of two tangent balls as in~\eqref{1.26bis},
and we show that~${\mathcal{O}}(t)$ presents no fattening phenomenon,
according to the statement below. 

\begin{theorem}\label{DUPAT} Assume \eqref{NUCLEONORMA} with $n=2$. Then the set~$\Sigma_{\mathcal{O}}(t)$ has empty interior for all $t>0$.
\end{theorem}

The evolution of the double ball is sketched in
Figure~\ref{d1234778fGHJKAII}: roughly speaking, the set shrinks at its surroundings,
emanating some mass from the origin, but it does not possess ``gray regions''
at its boundary.\medskip

\begin{figure}
    \centering
    \includegraphics[width=8cm]{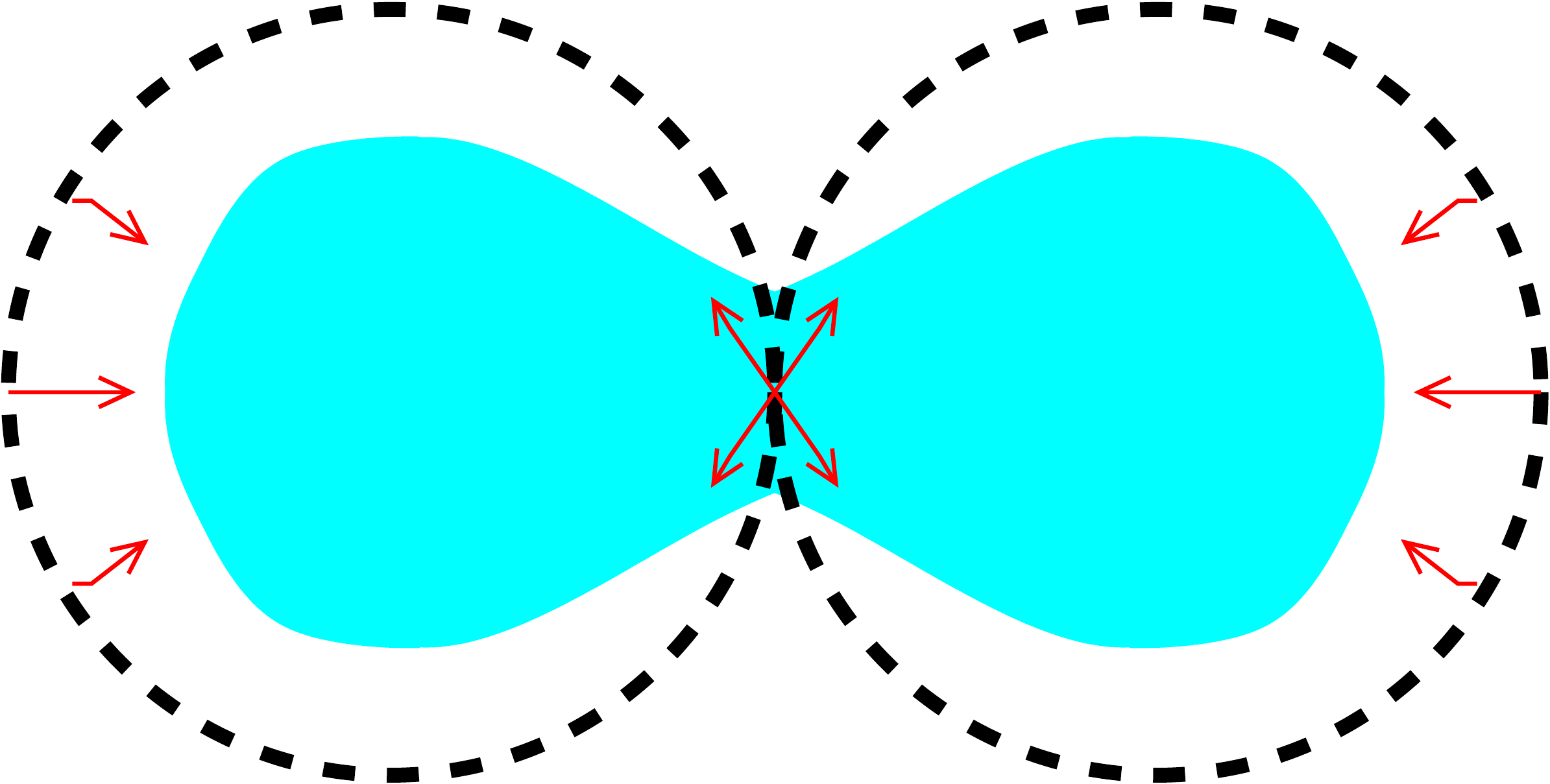}
    \caption{\em {{The evolution of two tangent balls described in Theorem~\ref{DUPAT}.}}}
    \label{d1234778fGHJKAII}
\end{figure}

\medskip
The rest of the paper is organized as follows. Section \ref{secpos}
deals with the fact
that the evolution starting from 
regular sets with positive $K$-curvature does not fatten and it contains the proof
of Theorem \ref{positive}.  In Section \ref{fatcross}  we prove Proposition \ref{PROP} and
the fattening of the evolution starting from the cross in $\R^2$,
under assumption \eqref{KERSI}, as stated in Theorem~\ref{FAT1}.

In Section \ref{secweak}, we show that under assumption \eqref{NOFHY}
the evolution starting from the cross in $\R^2$ does not fatten, but coincides
with the cross itself, that is we prove Theorem~\ref{NOFAT}.

Section \ref{secmin} contains the proof of 
the fact that the cross in $\R^2$ is never a $K$-minimal set for $\Per_K$, thus
establishing Proposition~\ref{mincroce}.

The last three sections present the evolution under the
fractional curvature flow, i.e., we assume that $K(x)=\frac{1}{|x|^{n+s}}$. In particular,
Section \ref{secstar} is devoted to the proof of the fact that the fractional
curvature evolution of strictly starshaped sets does not present fattening, which gives
Proposition~\ref{strictstar}.

In Section \ref{secdoppiagoccia},   we show an example in $\R^2$ of a compact
set with positive $K$-curvature, that is the double droplet, whose
fractional curvature evolution presents fattening,
thus proving Theorem \ref{FAT3}.

Then, in Section \ref{sectangenti} we show that the fractional 
curvature evolution starting from two tangent balls in $\R^2$ does not fatten,
which establishes Theorem~\ref{DUPAT}.

In Appendix \ref{viscosec} we review 
some basic  facts about level set flow, moreover we provide some
auxiliary results about  comparison with geometric barriers
and other basic properties of the evolution which are exploited in the 
proofs of the main results.

\medskip 

\subsubsection*{Notation} 
We denote by $B_r\subset\R^n$ the ball centered at $(0,0)$ of radius $r$
and by~$B_r(x_1,x_2,\dots, x_n)$ the ball of radius $r$ and center $x=(x_1,x_2,\dots, x_n)\in\R^n$.

Moreover $e_1=(1,0,\dots, 0)$, $e_2=(0,1, 0,\dots, 0)$ etc, and $\S^{n-1}= 
\{\omega\in \R^n\text{ s.t. }|\omega|=1\}$.

For a given closed set $E$, and for any~$x\in\R^n\setminus E$
we denote by~$\text{dist}(x, E)$ the distance from $x$ to $E$, that is 
$$ 
\text{dist}(x, E):= \inf_ {y\in E} |x-y|.
$$ 
Moreover,  we will denote with  $d_{E}(x)$  the signed distance function to $C=\partial E$,  
with the sign convention of being positive inside~$E$ and negative outside, that is
\begin{equation}\label{signed} d_E(x)=\begin{cases} 
\text{dist}(x, \R^n\setminus E) &{\mbox{if }} x\in \overline{E},\\
-\text{dist}(x, E) &{\mbox{if }} x\in \R^n\setminus E. \end{cases}\end{equation}
Finally, given two  sets $E, F\subset\R^n$, we denote by 
$d(E,F)$ the distance between the boundary of $E$ and the boundary of $F$, that is
\begin{equation}\label{defdef}
d(E,F):=\min_{{x\in\partial E}\atop{y\in \partial F}} |x-y|.\end{equation}  

\section{Regular sets of positive $K$-curvature and proof of Theorem \ref{positive}}\label{secpos} 

\begin{proof}[Proof of Theorem \ref{positive}]
We recall the continuity in $C^{1,1}$ of the $K$-curvature proved
in~\cite{MR3401008}. Namely, if $E^\eps$ is a family of
compact sets with boundaries in $C^{1,1}$ such that  $E^\eps\to E$ in $C^{1,1}$
(in the sense that the boundaries converges in $C^{1}$ and are of class $C^{1,1}$ uniformly in $\eps$) 
and $x^\e\in\partial E^\e\to x\in \partial E$, then $H^K_{E^\e}(x^\e)\to H^K_E(x)$, as~$
\e\searrow0$.

Now, let  $E$ be as in the statement of
Theorem \ref{positive}, and define, for $r>0$,
$$E^r:=\{x\in\R^n{\mbox{ s.t. }} d_E(x)\geq  -r\}.$$ 
Then, using also~\eqref{g7g7fugvasfegdu18iqey20}, we find that
there exists $\eps_0>0$ such that, 
for all $\eps\in (0, \eps_0) $,  there exists $0<\delta(\eps)\leq \delta$ such that 
 \[\min_{x\in \partial E^\eps} H^K_{E^\eps}(x)\geq \delta(\eps) >0.\] 
Fix $\eps<\eps_0$ and let $\bar{\delta}:=\inf_{\eta\in [0, \eps]} \delta(\eta)>0$. Fix $0<h<\bar \delta$. 
For all $t\in \left[0, \frac{\eps}{\bar \delta}\right]$ we define 
\[C(t):=E^{\eps-(\bar \delta-h)t}.\]  
We observe that  $C(t)$ is a supersolution to \eqref{kflow}, in the sense that it
satisfies \eqref{supergeo}. Indeed, 
\[ \partial_t x\cdot \nu=-\bar \delta+h\geq  -H^K_{C(t)}(x)+h.\]
Since $E\subseteq E^\eps=C(0)$, by Proposition \ref{subgeometrico},
we get that
\[ E^{+}(s)\subseteq  C(s)=E^{\eps-(\bar \delta-h)s} \ \ \  {\mbox{ for all }}
 s\in \left(0, \frac{\eps}{\bar \delta}\right]\quad 
{\mbox{with   }  }
d\left(E^+(s),E^{\eps-(\bar \delta-h)s}\right)\geq \eps.\]
This implies that  $E^{+}(s)\subseteq E$ for all $s\in \left[0, \frac{\eps}{\bar \delta}\right]$ and for all $h<\bar \delta$ 
and moreover that \[d\left(E^+(s), E\right)\geq d(E^{+}(s),E^{\eps-(\bar \delta-h)s})-d(E^{\eps-(\bar \delta-h)s}, E)\geq (\bar \delta-h)s.\]
Then, by the Comparison Principle in
Corollary \ref{cpgeometric}, we get  that \begin{equation}\label{diseq1}
E^{+}(t+s)\subseteq E^{-}(t), \quad  
{\mbox{with }} 
d\left(E^{+}(t+s),E^{-}(t)\right)\geq (\bar \delta-h)s\qquad  {\mbox{ for all }}
t>0, s\in \left(0, \frac{\eps}{\bar \delta}\right], h<\bar \delta.\end{equation}  Therefore, recalling Proposition \ref{scs}, we get 
\[| {\rm{int}}\, (E^+(t))\setminus \overline{E^-(t)}|\leq \limsup_{s\searrow 0}|
{\rm{int}}\, (E^+(t))|-  | E^+(t+s)|=
|{\rm{int}} (E^+(t))|-\liminf_{s\searrow 0} | E^+(t+s)| \leq 0.\] 
This gives the desired statement in Theorem \ref{positive}. 
\end{proof}

\section{$K$-curvature of the perturbed cross
and proofs of Proposition~\ref{PROP} and  of  Theorem~\ref{FAT1}}\label{fatcross}

In this section, our state space is $\R^2$. We consider the cross~${\mathcal{C}}\subseteq \R^2$ introduced in~\eqref{CROSS}
and the perturbed cross defined in~\eqref{PCROSS}. We also make
use of the notation in~\eqref{defpsi}. Then, we have:

\begin{lemma}\label{PC=2} Assume that~\eqref{ROTAZION} and
\eqref{ROTAZION2} hold true in $\R^2$.
Then, for any~$t\in[-r,r]$,
$$ H^K_{ {\mathcal{C}}_r }(t,r)\le-2\Psi(r).$$
\end{lemma}

\begin{figure}
    \centering
    \includegraphics[width=8cm]{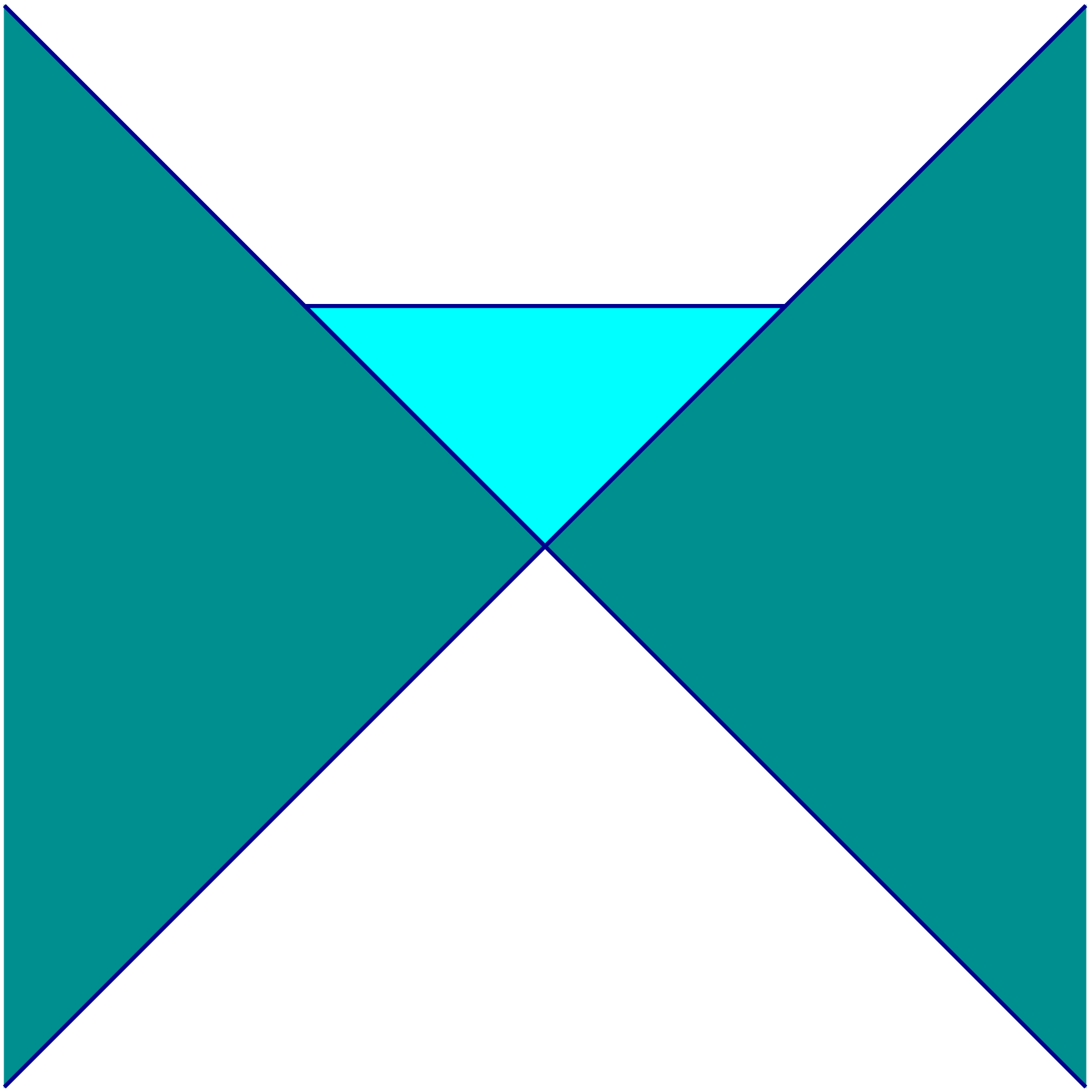}
    \caption{\em {{The set~${\mathcal{D}}_r$.}}}
    \label{D1}
\end{figure}

\begin{proof} Let
\begin{equation*}
{\mathcal{T}}_r:=\Big( (-r,r)^2\setminus{\mathcal{C}}\Big)\cap\{x_2<0\}
\end{equation*}
and
\begin{equation*}
{\mathcal{D}}_r:={\mathcal{C}}_r\setminus {\mathcal{T}}_r,\end{equation*}
see Figure~\ref{D1}.
Notice that~${\mathcal{C}}_r$ is the disjoint union of~${\mathcal{D}}_r$
and~${\mathcal{T}}_r$,
hence
$$ \chi_{ {\mathcal{C}}_r }=\chi_{ {\mathcal{D}}_r }+\chi_{ {\mathcal{T}}_r },$$
while~$\R^2\setminus{\mathcal{D}}_r$ is the disjoint union of~$\R^2\setminus{\mathcal{C}}_r$
and~${\mathcal{T}}_r$, which gives that
$$ \chi_{ \R^2\setminus{\mathcal{D}}_r }=\chi_{
\R^2\setminus{\mathcal{C}}_r}+\chi_{{\mathcal{T}}_r}.$$
Hence, we find that
\begin{equation}\label{819:2}
\chi_{\R^2\setminus{\mathcal{C}}_r}-\chi_{{\mathcal{C}}_r}=
\chi_{\R^2\setminus{\mathcal{D}}_r}-\chi_{{\mathcal{D}}_r}
-2\chi_{{\mathcal{T}}_r}.
\end{equation}
Now, we claim that, for any~$t\in[-r,r]$,
\begin{equation}\label{819:3}
H^K_{ {\mathcal{D}}_r }(t,r)\le0.\end{equation}
To this end, we partition~$\R^2$ into different regions, as depicted in Figure~\ref{X1},
and we use the notation, for each set~$Y\subseteq\R^2$,
\begin{equation}\label{INT-H}
{\mathcal{H}}(Y):=\lim_{\e\searrow0}\int_{Y\setminus B_\e(t,r)} 
K\big(x-(t,r)\big)\,dx.\end{equation}
In this way, we can write~\eqref{CURVA} as
\begin{equation}\label{0-1-P1}
H^K_{ {\mathcal{D}}_r }(t,r) = {\mathcal{H}}(C)+
{\mathcal{H}}(D)+
{\mathcal{H}}(U')+
{\mathcal{H}}(V')+
{\mathcal{H}}(W')-
{\mathcal{H}}(A)-
{\mathcal{H}}(B)-
{\mathcal{H}}(U)-{\mathcal{H}}(V)-{\mathcal{H}}(W).
\end{equation}
On the other hand, we can use symmetric reflections across the horizontal straight line
passing through the pole~$(t,r)$ to conclude that~${\mathcal{H}}(U)={\mathcal{H}}(U')$.
Similarly, we see that~${\mathcal{H}}(V)={\mathcal{H}}(V')$
and~${\mathcal{H}}(W)={\mathcal{H}}(W')$. As a consequence, the identity in~\eqref{0-1-P1}
becomes
\begin{equation}\label{0-1-P2}
H^K_{ {\mathcal{D}}_r }(t,r) = {\mathcal{H}}(C)+
{\mathcal{H}}(D)-
{\mathcal{H}}(A)-
{\mathcal{H}}(B).\end{equation}

\begin{figure}
    \centering
    \includegraphics[width=8cm]{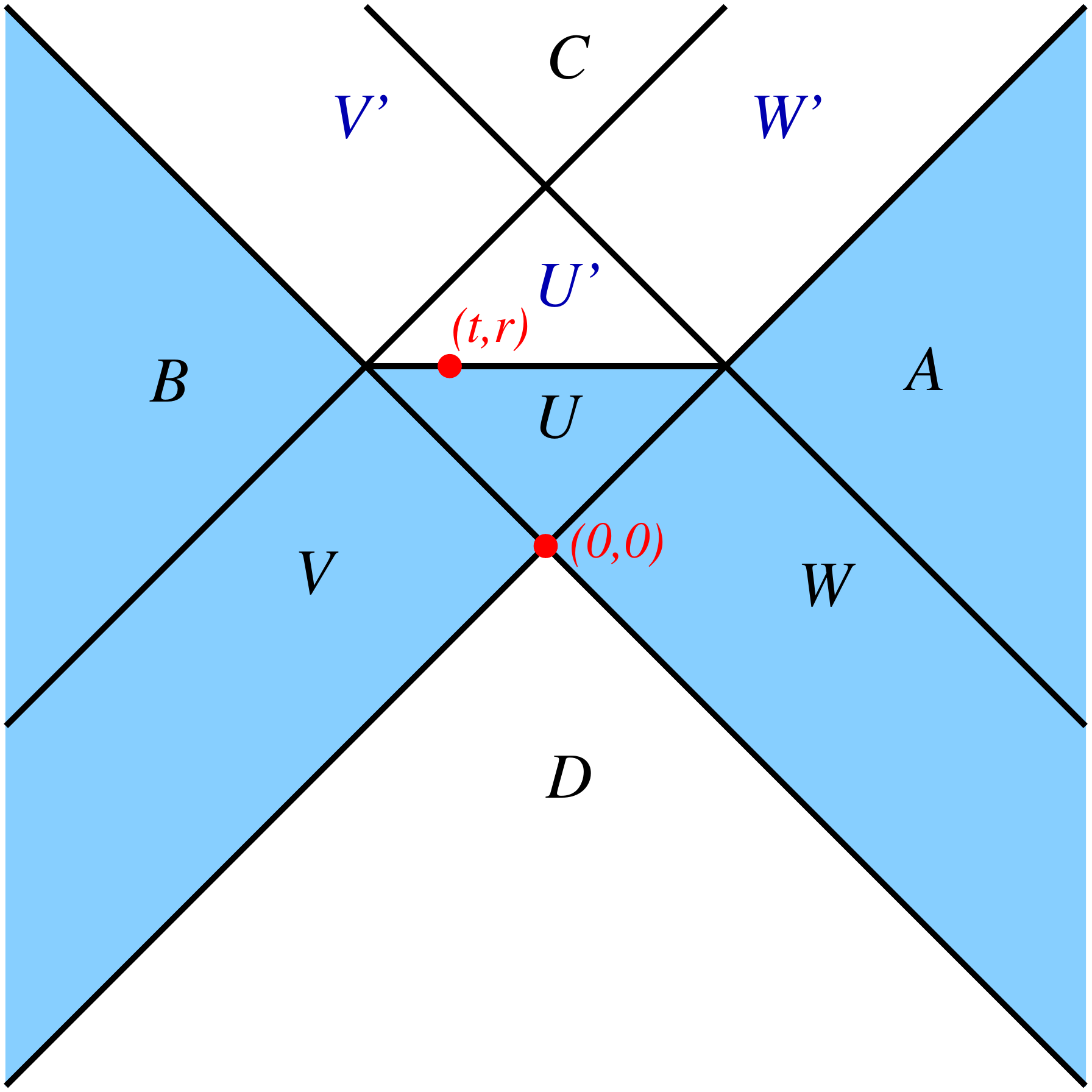}
    \caption{\em {{Splitting the set~${\mathcal{D}}_r$ and its complement
    into isometric regions.}}}
    \label{X1}
\end{figure}

Now we consider the straight line~$\ell:=\{x_2=x_1-t+r\}$. Notice that~$\ell$ passes
through the point~$(t,r)$ and it is parallel to two edges of the cross~${{\mathcal{C}}_r}$.
Considering the framework in Figure~\ref{X1},
reflecting the set~$D$ across~$\ell$ we obtain a set~$D'\subseteq B$,
and we write~$B=D'\cup E$, for a suitable slab~$E$.
Similarly, we reflect the set~$A$ across~$\ell$ to obtain a set~$ A'$
which is contained in~$C$, and we write~$C= A'\cup F$, for a suitable slab~$F$,
see Figure~\ref{X1ANNA}.

\begin{figure}
    \centering
    \includegraphics[width=8cm]{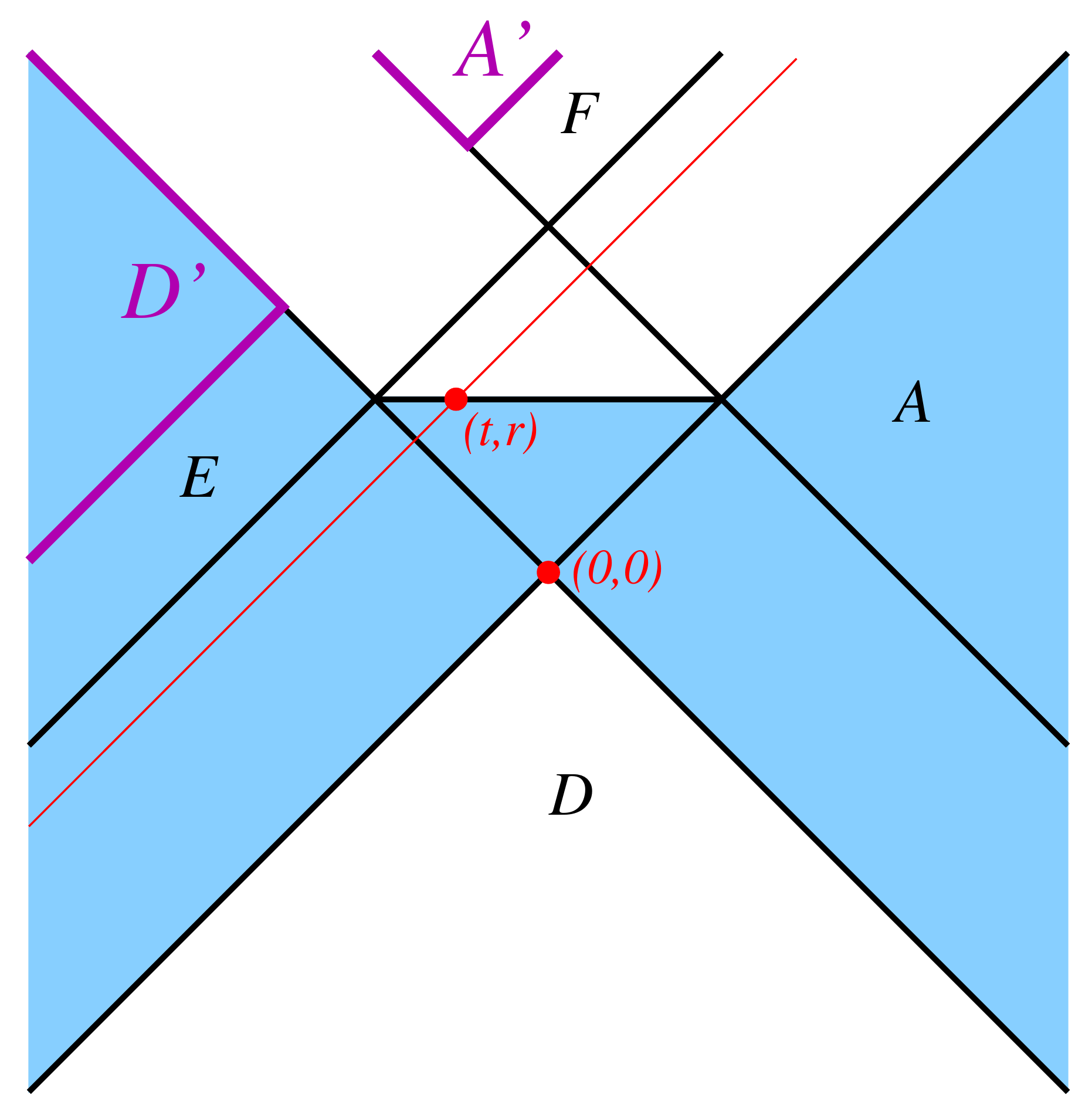}
    \caption{\em {{Reflecting~$D$ and~$A$ across~$\ell$, being~$E:=B\setminus D'$
    and~$F:=C\setminus A'$.}}}
    \label{X1ANNA}
\end{figure}

In further details, if~$T:\R^2\to\R^2$ is the reflection across~$\ell$,
we have that~$T(t,r)=(t,r)$ and~$|T(x-(t,r))|=|T(x)-(t,r)|=|x-(t,r)|$ for every~$x\in\R^2$,
and thus, by~\eqref{ROTAZION},
$$ K\big(x-(t,r)\big)=K_0(|x-(t,r)|)=K_0\big( \big|T(x-(t,r)) \big|\big)=K\big( T(x-(t,r)) \big).
$$
Accordingly, since~$D=T(D')$,
\begin{equation}\label{B82-1}
\begin{split}&
{\mathcal{H}}(B)-{\mathcal{H}}(D)=
\int_{B} K\big(x-(t,r)\big)\,dx - \int_{T(D')} K\big(x-(t,r)\big)\,dx\\&\qquad=
\int_{B} K\big(x-(t,r)\big)\,dx - \int_{D'} K\big(x-(t,r)\big)\,dx=
\int_{E} K\big(x-(t,r)\big)\,dx,
\end{split}\end{equation}
and similarly
\begin{equation}\label{B82-2} {\mathcal{H}}(C)-{\mathcal{H}}(A)=
\int_{F} K\big(x-(t,r)\big)\,dx.\end{equation}
Now we consider the straight line~$\ell':=\{x_2=-x_1+t+r\}$. Notice that~$\ell$ passes
through the point~$(t,r)$ and it is perpendicular to~$\ell$.
We let~$E'$ be the reflection across~$\ell'$ of the set~$E$ and we notice that~$E'\supseteq F$.
Therefore
$$ \int_{E} K\big(x-(t,r)\big)\,dx=\int_{E'} K\big(x-(t,r)\big)\,dx
\ge \int_{F} K\big(x-(t,r)\big)\,dx.$$
{F}rom this, \eqref{B82-1} and~\eqref{B82-2}, we obtain
\begin{eqnarray*}
&&{\mathcal{H}}(C)+
{\mathcal{H}}(D)-
{\mathcal{H}}(A)-
{\mathcal{H}}(B)= \int_{F} K\big(x-(t,r)\big)\,dx-\int_{E} K\big(x-(t,r)\big)\,dx\le0.
\end{eqnarray*}
This and~\eqref{0-1-P2} imply the desired result in~\eqref{819:3}.

Then, by~\eqref{819:2} and \eqref{819:3},
$$ H^K_{ {\mathcal{C}}_r }(t,r)=H^K_{ {\mathcal{D}}_r }(t,r)-2
\int_{ {\mathcal{T}}_r } K\big(y-(t,r)\big)\, dy\le0-2\Psi(r),$$
and this gives the desired result.
\end{proof}

With this, we are now in the position of completing the proof of Proposition~\ref{PROP}
via the following argument:

\begin{proof}[Proof of Proposition~\ref{PROP}] The claim in~\eqref{PRO2}
follows from Lemma~\ref{PC=2}. In addition, we have that~${\mathcal{C}}\subset {\mathcal{C}}_r$,
due to~\eqref{PCROSS}. We also observe that if~$p\in (\partial{\mathcal{C}}_r)\setminus
[-r,r]^2$, then~$p\in\partial{\mathcal{C}}$.
Consequently, by~\eqref{CURVA}, for
any~$p\in (\partial{\mathcal{C}}_r)\setminus
[-r,r]^2$, we have that
\begin{equation}\label{9230412}
H^K_{\mathcal{C}}(p)\ge H^K_{{\mathcal{C}}_r}(p). 
\end{equation}
Also, by symmetry, we see that~$H^K_{\mathcal{C}}(p)=0$
at any point ~$p\in \partial{\mathcal{C}}$, hence~\eqref{9230412}
gives that~$H^K_{{\mathcal{C}}_r}(p)\le0$ for any~$p\in (\partial{\mathcal{C}}_r)\setminus
[-r,r]^2$. 
 Since this inequality is also valid when~$
p\in (\partial{\mathcal{C}}_r)\cap[-r,r]^2$, due to~\eqref{PRO2},
the proof of~\eqref{PRO1} is complete.
\end{proof}

With Proposition~\ref{PROP}, we can now construct inner and outer barriers
as in Corollary~\ref{unboundedcoro}
to complete
the proof of Theorem~\ref{FAT1}. 
This auxiliary construction goes as follows.

\begin{lemma}\label{yg9uytfdiuye5t8wryegufgfiwefw}
Let~${\mathcal{C}}_r$ be as in \eqref{PCROSS}. Let~$R:=3\sqrt{2}\,r$
and define, for $\lambda\in \left[0,\frac{r}{2}\right)$,
\begin{equation}\label{PlIAK2dfgh93k} {\mathcal{C}}_r^\lambda:=\left\{x\in\R^2 {\mbox{ s.t. }}
d_{ {\mathcal{C}}_r }(x)\le-\lambda\right\}.\end{equation}
Then, for any~$p\in(\partial{\mathcal{C}}_r^\lambda)\setminus B_R$,
we have that~$H^K_{{\mathcal{C}}_r^\lambda}(p)\le0$.
\end{lemma}

\begin{figure}
    \centering
    \includegraphics[width=8cm]{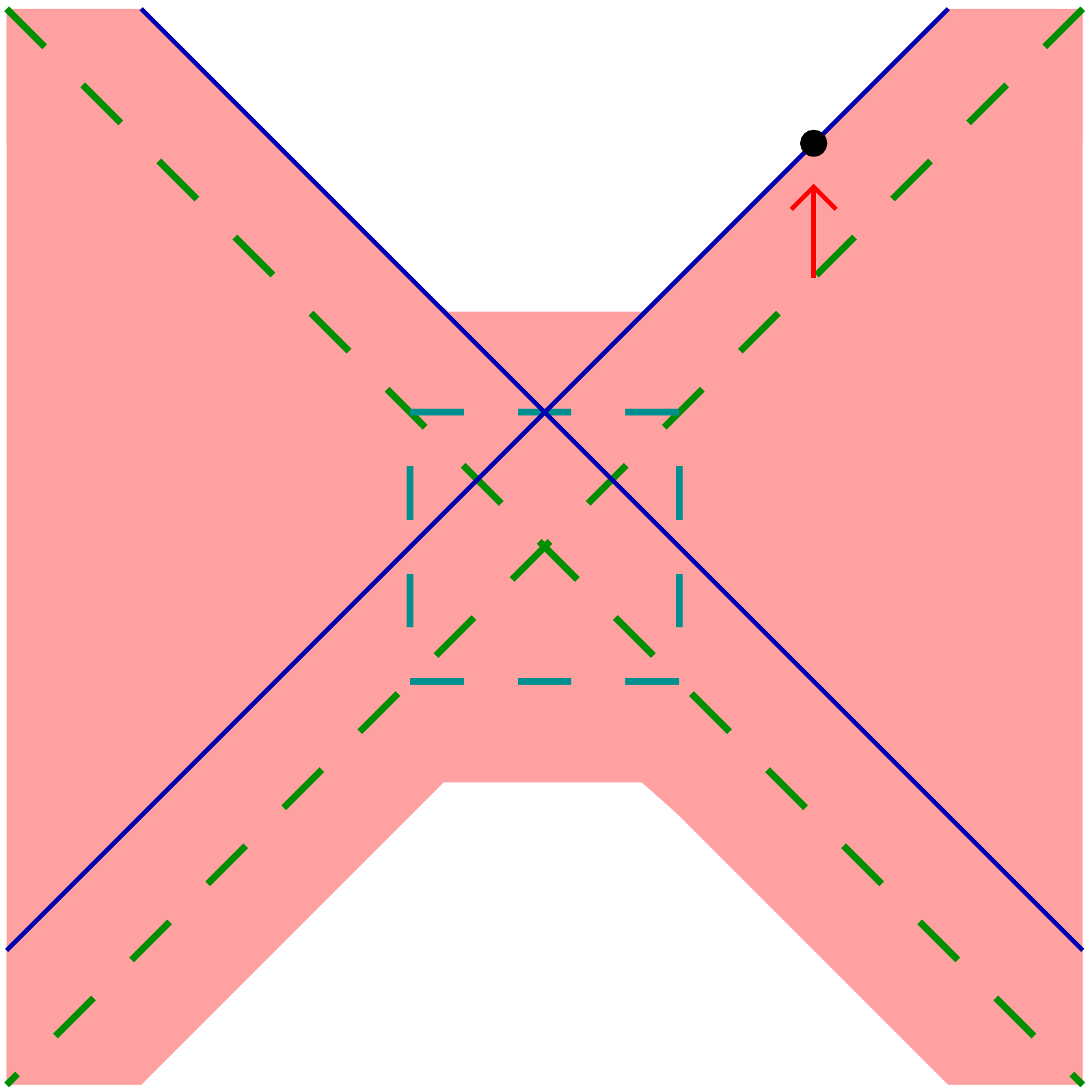}
    \caption{\em {{The set~${\mathcal{C}}_r^\lambda$, touched from inside at a boundary point by a translation of~${\mathcal{C}}$.}}}
    \label{POM4}
\end{figure}

\begin{proof} We observe that if~$p\in(\partial{\mathcal{C}}_r^\lambda)
\setminus B_R$,
then~$\partial{\mathcal{C}}_r^\lambda$ in the vicinity of~$p$
is a segment, and there exists a vertical translation of~${\mathcal{C}}$
by a vector~$v_0:=\pm\sqrt\lambda\,e_2$ such that~$p\in
{\mathcal{C}}+v_0$ and~${\mathcal{C}}+v_0\subset{\mathcal{C}}_r^\lambda$,
see Figure~\ref{POM4}. {F}rom this, we find that
$$ H^K_{{\mathcal{C}}_r^\lambda}(p)\le H^K_{{\mathcal{C}}+v_0}(p)=
H^K_{{\mathcal{C}}}(p-v_0)=0,$$
as desired.
\end{proof}

With this, we are ready to complete the proof of Theorem~\ref{FAT1},
by arguing as follows.

\begin{proof}[Proof of Theorem~\ref{FAT1}] 
The proof is based on the construction of suitable families
of geometric sub and supersolutions
starting from the perturbed cross $\mathcal{C}_r$, as defined
in \eqref{PCROSS}, to which apply Corollary \ref{unboundedcoro}. 

We
observe that $$
\mathcal{C}=\bigcap_{r>0} \mathcal{C}_r.$$ Moreover,
we see that $$
d_{\mathcal{C}}(x)\leq d_{\mathcal{C}_r}(x)\leq d_{\mathcal{C}}(x)+r.$$
These observations, together with
the Comparison Principle in Theorem \ref{comparison} and
Remark \ref{geosigma}, imply that 
\begin{equation}\label{86969741221122222}
\mathcal{C}^+(t)=\bigcap_{r>0} \mathcal{C}_r^+(t),\qquad {\mbox{ for all }} t>0. \end{equation}
Analogously, one can define \begin{equation}\label{Crgaue10}
\mathcal{C}^r
:= (\R^2\setminus \mathcal{C})\cup[-r,r]^2.\end{equation}
Let $\Psi$ as defined in \eqref{defpsi}.  
Fixed $r\in (0,r_0)$, where $r_0$ is as in \eqref{UNIMPIU},
we define $r_*(t)$ to be the solution to the ODE
\begin{equation}\label{PjPnqPLQ} \dot r_*(t)= \Psi(r_*(t))\end{equation}
 with initial datum $r_*(0)=r$. We fix $T>0$ such that $r_*(t)<r_0$ for all $t\in[0,T]$. 
Recalling the definition of $\Lambda$ in \eqref{LAMBDA DEF}, it is easy to check that \begin{equation}\label{8i6543i3}
\Lambda(r_*(t))=t+\Lambda(r),\qquad {\mbox{ for all }} t\in(0,T].
\end{equation}

Now, by~\eqref{rdit} and~\eqref{8i6543i3}, we see that
\begin{equation}\label{8ujn8uh6tg6tguh}
\Lambda(r_*(t))=\Lambda(r(t))+\Lambda(r)\ge\Lambda(r(t)).\end{equation}
Now, recalling the setting in~\eqref{PlIAK2dfgh93k},
we take into account the sets~${\mathcal{C}}_{r_*(t)}$
and~${\mathcal{C}}_{r_*(t)}^\lambda$,
with~$\lambda\in\left[0,\frac{r}{2}\right)$ and~$t\in [0,T]$,
and we claim that these sets
satisfy the assumptions in Corollary~\ref{unboundedcoro}, item ii).
To this end, 
we observe that, in the vicinity of the angular points of~${\mathcal{C}}_r$,
the complement of~${\mathcal{C}}_r$ is a convex set,
and therefore condition~\eqref{subgeo2} is satisfied by~${\mathcal{C}}_{r_*(t)}$.
Also, we take 
$$ 
\delta_1: =\inf_{t\in[0,T]} \Psi(r_*(t)),\qquad
\delta_2:=\inf_{t\in [0,T]} \inf_{p\in B_{3\sqrt{2}r_*(t)}}\int_{B_{r_*(t)/4}(3r_*(t)/4,0)-p} K(y)\,dy\qquad
{\mbox{and}}\qquad\delta:=\min\{\delta_1,\delta_2\}.$$
Notice that~$\delta>0$ thanks to \eqref{KERSI} and~\eqref{UNIMPIU}.
Then, by Proposition~\ref{PROP}
and~\eqref{PjPnqPLQ}, we get that at 
any point~$x=(x_1,x_2)$ of~$\partial{\mathcal{C}}_{r_*(t)}$
with~$x_2=\pm r_*(t)$, 
we have that
\begin{equation}\label{9UJA98765sdfgjnids}
-H^K_{ {\mathcal{C}}_{r_*(t)} }(x) \geq  2\Psi(r_*(t)) =
\dot r_*(t)+\Psi(r_*(t))
\geq \dot r_*(t) +\delta_1\geq \partial_t x\cdot \nu(x)+\delta.\end{equation}
In addition, if~$x=(x_1,x_2)\in(\partial{\mathcal{C}}_{r_*(t)})\cap B_{4R}$
and~$|x_2|> r_*(t)$, we
have that
\[-H^K_{ {\mathcal{C}}_{r_*(t)} }(x) \geq
-H^K_{ {\mathcal{C}} }(x)+\int_{B_{ r_*(t)/4}(3 r_*(t)/4,0)} K(y-x)\,dy=
\int_{B_{ r_*(t)/4}(3 r_*(t)/4,0)-x} K(y)\,dy\ge\delta_2\ge\delta=
\partial_t x\cdot \nu(x)+\delta.
\]
This and~\eqref{9UJA98765sdfgjnids} give that condition~\eqref{subgeo}
is fulfilled by~${\mathcal{C}}_{r_*(t)} $.

Furthermore, in light of Lemma~\ref{yg9uytfdiuye5t8wryegufgfiwefw},
we know that, for any~$x\in(\partial{\mathcal{C}}_{r_*(t)}^\lambda)
\setminus B_R$,
$$H^K_{{\mathcal{C}}_{r_*(t)}^\lambda}(p)\le0=\partial_t x\cdot \nu(x),$$
which says that condition~\eqref{zero1} is fulfilled
by~${{\mathcal{C}}_{r_*(t)}^\lambda}$.

Therefore, we are in the position of using
Corollary \ref{unboundedcoro}, item ii). In this way, we find that
\[
\mathcal{C}_{r_*(t)}\subseteq \mathcal{C}_r^+(t),\qquad {\mbox{ for all }} t\in [0,T].
\]
Hence, recalling~\eqref{8ujn8uh6tg6tguh},
\[
\mathcal{C}_{r(t)}\subseteq \mathcal{C}_r^+(t),\qquad {\mbox{ for all }} t\in [0,T].
\]
Taking intersections, in view of~\eqref{86969741221122222},
we obtain that
\begin{equation}\label{09}
\mathcal{C}_{r(t)}\subseteq \mathcal{C}^+(t),\qquad {\mbox{ for all }} t\in [0,T].
\end{equation}
Analogously, one can use the setting in~\eqref{Crgaue10}, combined
with Corollary \ref{unboundedcoro}, item i), and deduce that
\begin{equation}\label{099}
 \mathcal{C}^{r(t)}\subseteq (\R^2\setminus \mathcal{C})^+(t)\qquad {\mbox{ for all }} t\in[0, T]. 
\end{equation}
By \eqref{09} and \eqref{099} we get
\[  \left[-r(t), r(t)\right]^2=
\mathcal{C}_{r(t)}\cap \mathcal{C}^{r(t)}\subseteq
\mathcal{C}^+(t)\cap(\R^2\setminus \mathcal{C})^+(t)
=\Sigma_{\mathcal{C}}(t),\]
which implies~\eqref{9238475y6uedjidftgeri}, as desired.
\end{proof}

\section{Moving boxes, weak interaction kernels and proof of Theorem~\ref{NOFAT}}\label{secweak} 

To simplify some computation,
in this section we operate a rotation of coordinates so that
\begin{equation}\label{conobis}
\mathcal{C}=\{x\in\R^2\text{ s.t. } x_1x_2\geq 0\}.\end{equation} 
To prove Theorem~\ref{NOFAT}, it is convenient to consider ``expanding boxes''
built by the following sets. For any~$r\in(0,1)$, we define
\begin{equation}\label{ENNERER} {\mathcal{N}}_r:=\Big( [r,+\infty)\times[r,+\infty)\Big)\cup
\Big( (-\infty,-r]\times(-\infty,-r]\Big),\end{equation}
see Figure~\ref{D36475869073494994949za1}. 
 
\begin{figure}
    \centering
    \includegraphics[width=8cm]{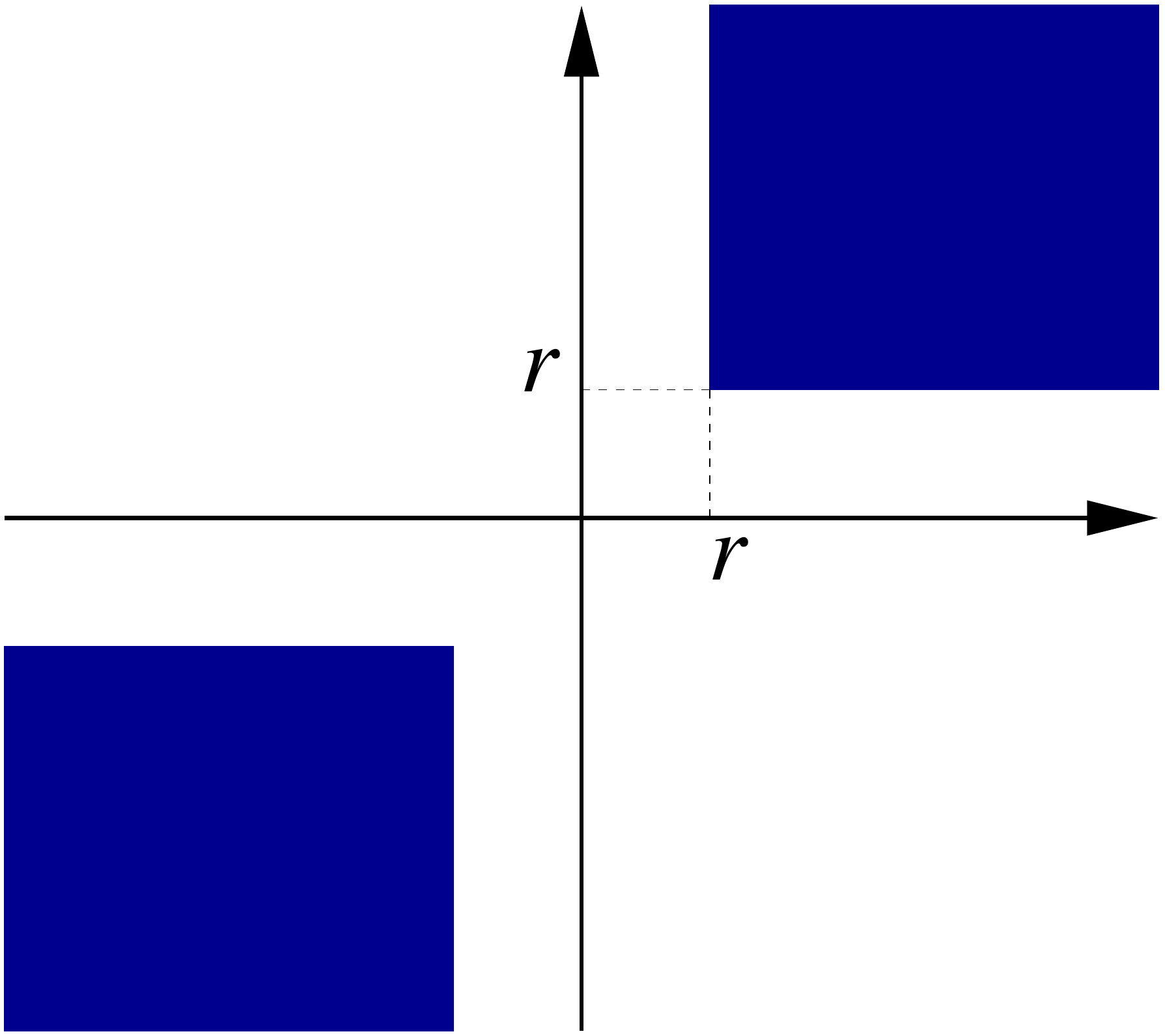}
    \caption{\em {{The set~${\mathcal{N}}_r$.}}}
    \label{D36475869073494994949za1}
\end{figure}

Then, recalling the notation in~\eqref{NOFHY}, we have:

\begin{lemma}\label{D36475869073494994949zaL}
Assume that~$K$ satisfies~\eqref{ROTAZION},
\eqref{ROTAZION2} and~\eqref{NOFHY} in $\R^2$. Then, for any~$p\in\partial{\mathcal{N}}_r$,
$$ H^K_{ {\mathcal{N}}_r }(p) \leq 2\,\Phi(2r).$$
\end{lemma}

\begin{figure}
    \centering
    \includegraphics[width=8cm]{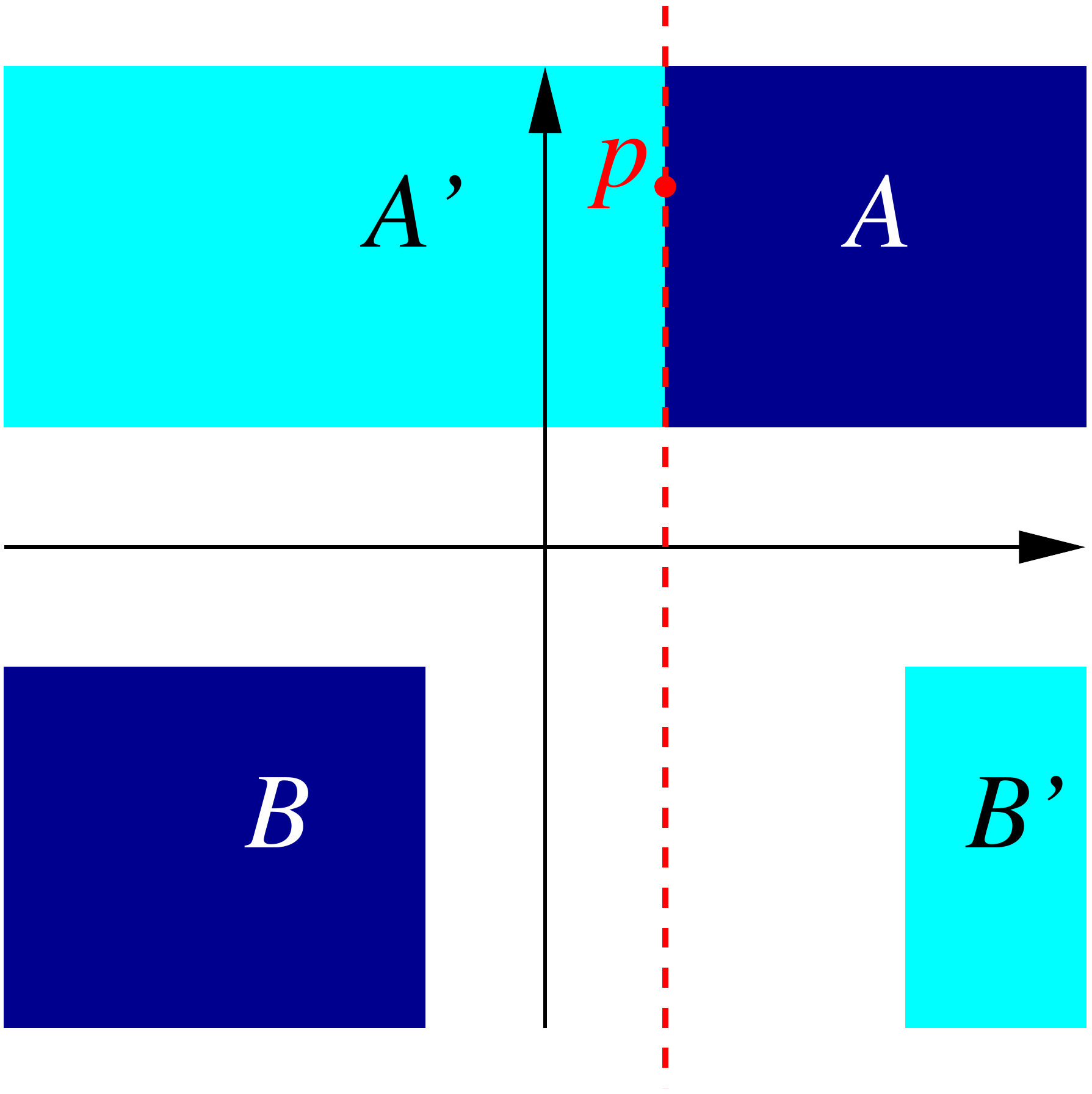}
    \caption{\em {{Simplifications in the computations of Lemma~\ref{D36475869073494994949zaL}.}}}
    \label{D36475869073494994949za22}
\end{figure}

\begin{proof} We denote by $A$ and~$B$ the two connected components of~${ {\mathcal{N}}_r }$
and consider the straight line~$\ell$ passing through~$p$ and tangent to~${ {\mathcal{N}}_r }$
at~$p$: see Figure~\ref{D36475869073494994949za22}. By reflection across~$\ell$,
we can consider the regions~$A'$ and~$B'$ which are symmetric to~$A$ and $B$, respectively.
In particular, if~$p=(p_1,p_2)$ and~$M(x_1,x_2):=(2p_1-x_1,x_2)$,
we have that~$M(A\cup B)=A'\cup B'$ and~$M(B_\e(p))=B_\e(p)$, and therefore
\begin{eqnarray*} &&\int_{(A'\cup B')\setminus B_\e(p)} K(p-y)\,dy=
\int_{M((A\cup B)\setminus B_\e(p))} K(p-y)\,dy=
\int_{M((A\cup B)\setminus B_\e(p))} K(p-Mx)\,dx\\&&\qquad=
\int_{M((A\cup B)\setminus B_\e(p))} K(-p_1+x_1,p_2-x_2)\,dx
=\int_{(A'\cup B')\setminus B_\e(p)} K(p-x)\,dx,\end{eqnarray*}
thanks to~\eqref{ROTAZION}. Then, denoting by
$$T:=\big( \R^2\setminus{ {\mathcal{N}}_r }\big)\setminus \big(A'\cup B'\big),$$
which is the ``white region'' in Figure~\ref{D36475869073494994949za22}, we see that
\begin{equation}\label{WHI}
\begin{split}
& H^K_{ {\mathcal{N}}_r }(p) = \lim_{\e\searrow0}
\int_{(A'\cup B')\setminus B_\e(p)} K(p-x)\,dx
-\int_{(A\cup B)\setminus B_\e(p)} K(p-x)\,dx+
\int_{T} K(p-x)\,dx\\&\qquad\qquad\qquad= \int_{T} K(p-x)\,dx.\end{split}
\end{equation}
Up to rotations, we may assume that
\begin{equation}\label{05yhbnmcvjdf0}
T= \big( \R\times[-r,r]\big)\cup \big( [-r,3r]\times (-\infty,-r]\big).\end{equation}
Recalling ~\eqref{NOFHY}, and that $p_1=r$, we get 
\begin{equation}\label{05yhbnmcvjdf}
\int_{[-r,3r]\times (-\infty,-r]} K(x-p)\,dx\leq
\int_{[-r,3r]\times (-\infty,-r]} K_1(|x-p|)\,dx\leq
\int_{[-r,3r]\times \R} K_1(|x-p|)\,dx =
\Phi(2r)\end{equation}
where $\Phi$ is defined in \eqref{NOFHY}.
Moreover, since $p_1=r$ and $p_2\geq r$, and $K_1$ is nonincreasing,
we get that $K_1(|x-p|)\leq K_1(|x-(r,r)|)$,
for every $x\in \R\times[-r,r]$. As consequence,
\begin{eqnarray*}
&& \int_{\R\times [-r,r]} K(x-p)\,dx\leq  \int_{\R\times [-r,r]}  K_1(|x-p|)\,dx\\&&\qquad
\leq \int_{\R\times [-r,r]}  K_1(|x-(r,r)|)\,dx \leq \int_{\R\times [-r,3r]} K_1(|x-(r,r)|)\,dx= 
\Phi(2r).\end{eqnarray*}
{F}rom this and~\eqref{05yhbnmcvjdf}, and recalling~\eqref{05yhbnmcvjdf0},
we obtain that
$$\int_{T} K(p-x)\,dx\leq 2\Phi(2r).$$
This and~\eqref{WHI} give the desired result.
\end{proof}


\medskip 

For $\lambda\in (0,r)$ we define the sets 
\begin{equation}\label{defNrla}
{\mathcal{N}}_{r}^\lambda:=\{
x\in\R^2 {\mbox{ s.t. }} d_{ {\mathcal{N}}_{r} }(x)\ge-\lambda
\}.
\end{equation}
We observe that for any~$x\in \partial{\mathcal{N}}_{r}^\lambda$ there exists
a unique point $x'\in\partial{\mathcal{N}}_{r}$ such that 
$|x-x'|=d(\mathcal{N}_{r}^\lambda,\mathcal{N}_{r})=\lambda$.
Letting $v_x:=x-x'$, it follows that~${\mathcal{N}}_{r}+v_x\subset{\mathcal{N}}_{r}^\lambda$,
see Figure~\ref{4567PO52367485M4}.
This and Lemma~\ref{D36475869073494994949zaL} give that
\begin{equation}\label{eqHKla}
H^K_{ {\mathcal{N}}_{r}^\lambda }(x)
\le H^K_{ {\mathcal{N}}_{r} }(x+v_x)\le 2\Phi(2r)\qquad
\text{for any $x\in  \partial{\mathcal{N}}_{r}^\lambda$.}
\end{equation}

\begin{figure}
    \centering
    \includegraphics[width=8cm]{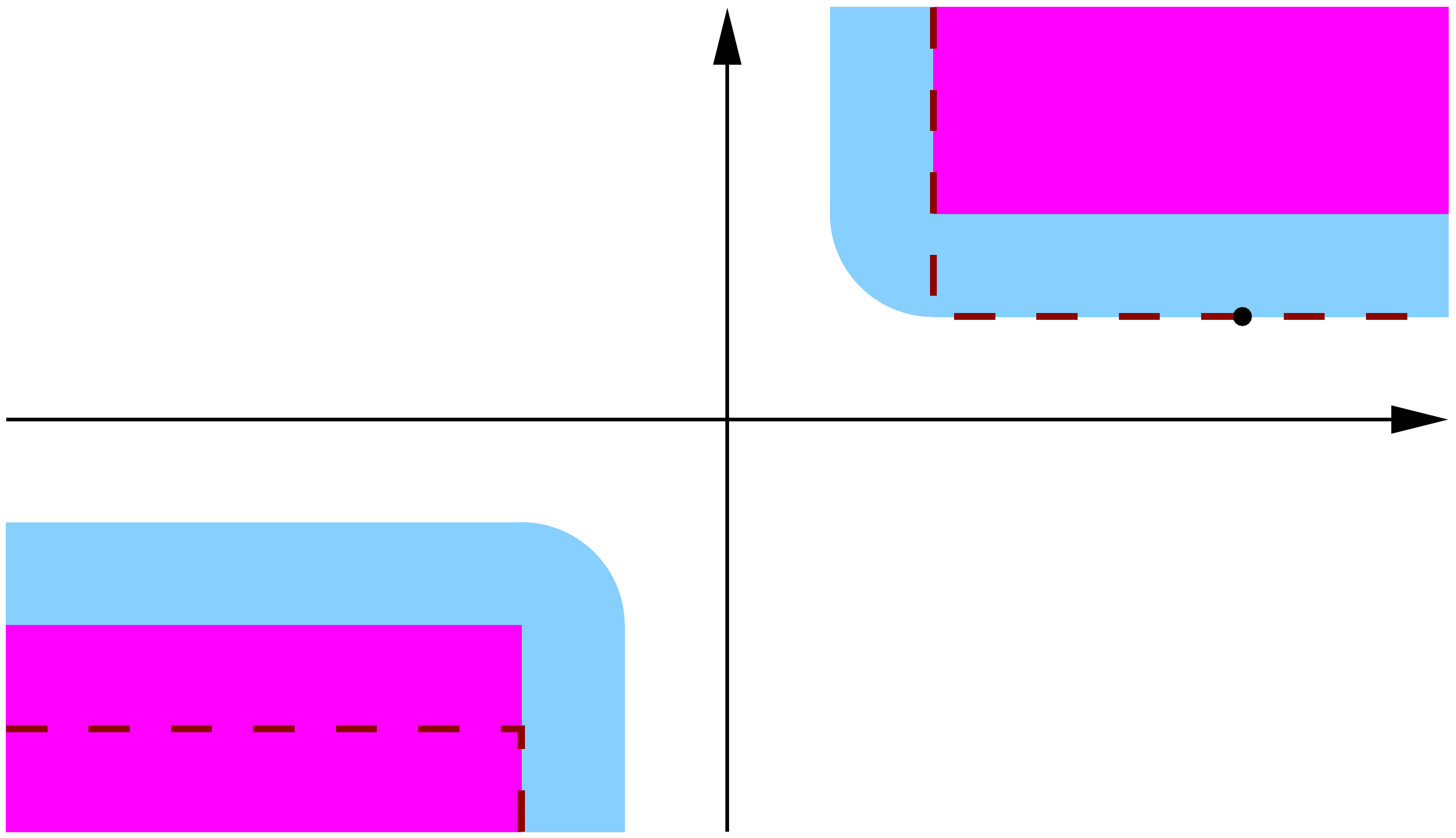}
    \caption{\em {{The set~${\mathcal{N}}_r^\lambda$, touched from inside at a boundary point by a translation of~${\mathcal{N}}_r$.}}}
    \label{4567PO52367485M4}
\end{figure}

\medskip

With this preliminary work, we can prove Theorem~\ref{NOFAT}.

\begin{proof}[Proof of Theorem~\ref{NOFAT}] 
We note that ${\mathcal{M}}_{r}:={\mathcal{N}}_{r}^{r/2}\subseteq \mathcal{C}$,
being~$\mathcal{C}$ defined in~\eqref{conobis}
and~${\mathcal{N}}_{r}^{r/2}$ defined in~\eqref{defNrla}, with $\lambda=r/2$.
Moreover, we have that~$d(\mathcal{C},{\mathcal{M}}_r)=r/2>0$.
Hence, by Corollary~\ref{cpgeometric} 
we get that~$ {\mathcal{M}}_{r}^+(t)\subseteq \mathcal{C}^-(t)$ for all $t>0$.
In particular, since
$$\bigcup_{r>0} \mathcal{M}_r={\rm{int}}(\mathcal{C}),$$ 
we see that
\begin{equation}\label{unonr} 
\bigcup_{r>0}  {\mathcal{M}}_{r}^+(t)= \mathcal{C}^-(t).
\end{equation}
Our aim is to construct starting from ${ {\mathcal{M}}_{r} }$ a continuous family of geometric subsolutions and then apply Proposition \ref{subgeometrico}.
Fixed~$\varrho\in(0,1)$, we define
$$ F_\varrho(r):=\int_\varrho^r \frac{d\vartheta}{6\Phi(2\vartheta)}.$$ 
Notice that~$F_\varrho$ is strictly increasing, so we can consider its inverse~$G_\varrho$
in such a way that~$F_\varrho(G_\varrho(t))=t$. Then, for~$t\in [0, T]$, we set~$
r_\varrho (t):= G_\varrho (t)$  and we consider the evolving sets ~$
{ {\mathcal{M}}_{r_\varrho(t)} }$.  We remark that
$$ F_\varrho(\varrho)=0=F_\varrho(G_\varrho(0))=F_\varrho(r_\varrho(0)),$$
and so~$r_\varrho(0)=\varrho$. In addition,
the outer  normal velocity of~${ {\mathcal{M}}_{r_\varrho(t)} }$
is
\begin{equation}\label{58998u7y669dcu}
-\dot r_\varrho(t)+ \frac{1}{2}\dot r_\varrho(t)=-\frac{1}{2}G_\varrho '(t)=-\frac1{2F_\varrho'(G_\varrho(t))}=-3
{\,\Phi(2G_\varrho(t))}=-3{\,\Phi(2r_\varrho(t))}.\end{equation}
So, if $$\delta:=\Phi(2\varrho)=\min_{r\in [\varrho, r_\varrho(T)]} \Phi(2r),$$ we have that 
\begin{equation}\label{uuu} \partial_t x\cdot \nu(x)=-\frac{1}{2}\dot r_\varrho(t)=-2\Phi(2r_\varrho(t)) - \Phi(2r_\varrho(t))\leq  - H^K_{\mathcal{M}_{r_\varrho(t)}}(x) -\delta\end{equation} 
for all $x\in \partial \mathcal{M}_{r_\varrho(t)}$,
thanks to \eqref{eqHKla}.

We observe that~\eqref{uuu} says that~\eqref{subgeo}
is satisfied by~$\mathcal{M}_{r_\varrho(t)}$.
So, to exploit Corollary~\ref{unboundedcoro},
we now want to check that condition~\eqref{zero1}
is satisfied by the set
$$ {\mathcal{M}}_{r_\varrho(t)}^\lambda:=\{
x\in\R^2 {\mbox{ s.t. }} d_{ {\mathcal{M}}_{r_\varrho(t)} }(x)\ge-\lambda
\}\qquad\mbox{ for $\lambda\in (0, \rho)$}.$$  We exploit again the estimate \eqref{eqHKla} which gives that 
\[H^K_{ {\mathcal{M}}_{r_\varrho(t)}^\lambda}(x)
\le 2\Phi(2r_{\varrho}(t))\qquad
\text{for any $x\in  \partial{\mathcal{M}}_{r_\varrho(t)}^\lambda$.} \]
Thus, in view of~\eqref{58998u7y669dcu},
$$ \partial_t x\cdot \nu(x)=-\frac{1}{2}\dot r_\varrho(t)=
-3{\,\Phi(2r_\varrho(t))}\le-H^K_{ {\mathcal{M}}_{r_\varrho(t)}^\lambda }(x).$$
This gives that~${\mathcal{M}}_{r_\varrho(t)}^\lambda $ satisfies
condition~\eqref{zero1} and therefore we can  apply   Corollary~\ref{unboundedcoro}, item ii). 

Then, it follows that, for all $\varrho\in (0,1)$, 
\begin{equation}\label{trenr} {\mathcal{M}}_{r_\varrho(t)}\subseteq {\mathcal{M}}_{\varrho}^+(t). \end{equation}
Also, for any~$t>0$, we claim that
\begin{equation}\label{923rtrsDEDuu}
\lim_{\varrho\searrow0} r_\varrho(t)=0.
\end{equation}
To prove this, we argue by contradiction and suppose that~$r_{\varrho_k}(t)\ge a_0$,
for some~$a_0>0$ and some infinitesimal sequence~$\varrho_k$. Then,
$$ t=F_{\varrho_k}(G_{\varrho_k}(t))=
F_{\varrho_k}(r_{\varrho_k}(t))\ge F_{\varrho_k}(a_0)
=\int_{\varrho_k}^{a_0} \frac{d\vartheta}{6\,\Phi(2\vartheta)}=\frac{1}{12}
\int_{2\varrho_k}^{2a_0} \frac{d\tau}{\Phi(\tau)}.
$$
This is in contradiction with~\eqref{HIG} and so it proves~\eqref{923rtrsDEDuu}.

In view of \eqref{923rtrsDEDuu}, we find that
$$\bigcup_{\varrho>0} {\mathcal{M}}_{r_\varrho(t)}=\text{int } \mathcal{C}.$$ 
So, recalling \eqref{unonr} and \eqref{trenr},
we conclude that
\begin{equation}\label{08678943899w8e8hx8w:19231}
\text{int } \mathcal{C} =\bigcup_{\varrho>0} {\mathcal{M}}_{r_\varrho(t)}\subseteq
\bigcup_{\varrho>0}  {\mathcal{M}}_{\varrho}^+(t)= \mathcal{C}^-(t)\qquad
{\mbox{for all }}t\in [0,T]. \end{equation}
Analogously, 
one can define
$$ {\mathcal{N}}^{r}:=\Big( (-\infty, -r]\times[r,+\infty)\Big)\cup
\Big( [r+\infty)\times(-\infty,-r]\Big), \qquad \mathcal{M}^r=({\mathcal{N}}^{r})^{r/2}:=\{
x\in\R^2 {\mbox{ s.t. }} d_{ {\mathcal{N}}^{r} }(x)\ge-\lambda
\}.
 $$
and see that
\begin{equation}\label{08678943899w8e8hx8w:19232}\text{int } (\R^2\setminus \mathcal{C})\subseteq \R^2\setminus
\mathcal{C}^-(t)\qquad {\mbox{for all }}t\in [0,T]. \end{equation}
Putting together \eqref{08678943899w8e8hx8w:19231}
and~\eqref{08678943899w8e8hx8w:19232}, we conclude that 
\[ \text{int } \mathcal{C} \subseteq  \mathcal{C}^-(t)\subseteq \mathcal{C}^+(t)
\subseteq \mathcal{C},\] and 
so $\Sigma_{\mathcal{C}}(t)=\partial\mathcal{C}$, thus establishing~\eqref{8G}. 
\end{proof}

\section{$K$-minimal cones and proof of Proposition \ref{mincroce}} \label{secmin} 

In this section we show  that~$\mathcal{C}\subseteq \R^2$, as defined in~\eqref{CROSS}, is never
a $K$-minimal set, under the assumptions~\eqref{ROTAZION}
and~\eqref{ROTAZION2}, namely we prove Proposition~\ref{mincroce}. 
This will be proved using the family of perturbed crosses~$\mathcal{C}_r$
introduced in~\eqref{PCROSS} and  
the fact that $H^K_E$ is the first variation of the
nonlocal perimeter $\Per_K$ defined in~\eqref{per},
as shown in \cite{MR3401008}. 
%

\begin{proof}[Proof of Proposition \ref{mincroce}]
With the notation in~\eqref{CROSS} and~\eqref{PCROSS},
we claim that that there exists~$r>0$ such that,
for all~$R>\sqrt{2}\,r$, 
\begin{equation}\label{claim}\Per_K(\mathcal{C}_{r}, B_R)
< \Per_K(\mathcal{C}, B_R). \end{equation}
Let $r>0$ and $R>\sqrt{2}r$, so that $ \mathcal{C}_{r}\setminus B_R= \mathcal{C}\setminus B_R$. 
Let \[W_r:= \mathcal{C}_{r}\setminus \mathcal{C} \subseteq B_R.  \]
Let $\delta\in(0, r)$ and $K_\delta(y):=K(y)(1-\chi_{B_\delta}(y))$.
We define
$\Per_\delta(E)$ as in \eqref{per},
$\Per_\delta(E,B_R)$ as in \eqref{locper},
and~$H^\delta_E$ as in~\eqref{CURVA}, with $K_\delta$ in place of $K$. 
In this setting, we get that 
\begin{equation}\label{perw}   \Per_{\delta} (W_r)=\Per_\delta(W_r, B_R)=
\Per_\delta(\mathcal{C}_{r}, B_R) -\Per_\delta
(\mathcal{C}, B_R)+2\int_{W_r}\int_{\mathcal{C}}K_\delta(x-y)\,dx\,dy.\end{equation} 
We also observe that
\begin{equation*} 
\Per_{\delta} (W_r)= \int_{W_r}\int_{\R^2\setminus W_r} K_\delta(x-y)\,dx\,dy
=  \int_{W_r}\int_{\R^2\setminus \mathcal{C}_{r}} K_\delta(x-y)\,dx\,dy
+ \int_{W_r}\int_{ \mathcal{C}} K_\delta(x-y)\,dx\,dy.
\end{equation*}
Substituting this identity into~\eqref{perw}, we find that
\begin{equation}\label{PLPLPLPAKJerUJS}\begin{split}
\Per_\delta(\mathcal{C}_{r}, B_R) -\Per_\delta
(\mathcal{C}, B_R)\,&=\Per_{\delta} (W_r)-
2\int_{W_r}\int_{\mathcal{C}}K_\delta(x-y)\,dx\,dy\\
&=\int_{W_r}\int_{\R^2\setminus \mathcal{C}_{r}} K_\delta(x-y)\,dx\,dy
-\int_{W_r}\int_{ \mathcal{C}} K_\delta(x-y)\,dx\,dy.
\end{split}\end{equation}
Now, given $x=(x_1,x_2)\in W_r$, we have that~$x\in \partial 
\mathcal{C}_{r(x)}$, with~$r(x):=|x_2|\in(0,r]$,
where the notation of~\eqref{PCROSS} has been used.
Then, by Lemma \ref{PC=2}, we have that 
\begin{equation}\label{curvw}
H^\delta_{\mathcal{C}_{r(x)}}(x)\leq -2\Psi_\delta(r(x)),\end{equation} 
where $\Psi_\delta$ is as in \eqref{defpsi} with $K_\delta$ in place of $K$, that is 
$$\Psi_\delta(s):=\int_{B_{s/4}(7s/4,0)} K_\delta(x)\,dx\geq 0.$$
We write~\eqref{curvw} as
\begin{equation*}\begin{split}
-2\Psi_\delta(r(x))\,&\ge
\int_{\R^2\setminus \mathcal{C}_{r(x)}} K_\delta(x-y)\,dy-
\int_{\mathcal{C}_{r(x)}} K_\delta(x-y)\,dy\\
&=\int_{\R^2\setminus \mathcal{C}_{r(x)}} K_\delta(x-y)\,dy-
\int_{\mathcal{C}} K_\delta(x-y)\,dy-
\int_{W_{r(x)}} K_\delta(x-y)\,dy\\
&=\int_{\R^2\setminus \mathcal{C}_{r}} K_\delta(x-y)\,dy
+\int_{W_r\setminus W_{r(x)}}K_\delta(x-y)\,dy
-
\int_{\mathcal{C}} K_\delta(x-y)\,dy-
\int_{W_{r(x)}} K_\delta(x-y)\,dy.\end{split}
\end{equation*}
Therefore, integrating over~$x\in W_r$,
\begin{equation}\label{g5GA7hAU9jsdHA}\begin{split}&
\int_{W_r}\int_{\R^2\setminus \mathcal{C}_{r}} K_\delta(x-y)\,dx\,dy-\int_{W_r}
\int_{\mathcal{C}} K_\delta(x-y)\,dx\,dy\\ \le\;&
\int_{W_r}\int_{W_{r(x)}} K_\delta(x-y)\,dx\,dy-
\int_{W_r}\int_{W_r\setminus W_{r(x)}}K_\delta(x-y)\,dx\,dy
-2\int_{W_r}\Psi_\delta(r(x))\,dx\\
=\;&
2\int_{W_r}\int_{W_{r(x)}} K_\delta(x-y)\,dx\,dy-
\int_{W_r}\int_{W_r}K_\delta(x-y)\,dx\,dy
-2\int_{W_r}\Psi_\delta(r(x))\,dx
.\end{split}\end{equation}
We now observe that
$$W_r=\{x\in\R^2 {\mbox{ s.t. }}|x_2|>|x_1|
{\mbox{ and }} |x_2|<r\},$$
and thus
\begin{eqnarray*}&&
2\int_{W_r}\int_{W_{r(x)}} K_\delta(x-y)\,dx\,dy\\ &=&
\int_{x\in W_r}\left(\int_{y\in W_{r(x)}} K_\delta(x-y)\,dy\right)\,dx
+\int_{y\in W_r}\left(\int_{x\in W_{r(y)}} K_\delta(x-y)\,dx\right)\,dy\\
&=& \int_{\{ |x_1|<|x_2|<r\}}\left(\int_{
\{ |y_1|<|y_2|<r(x)\}
} K_\delta(x-y)\,dy\right)\,dx+
\int_{\{ |y_1|<|y_2|<r\}}\left(\int_{
\{ |x_1|<|x_2|<r(y)\}
} K_\delta(x-y)\,dx\right)\,dy\\
&=& \int_{\{ |x_1|<|x_2|<r\}}\left(\int_{
\{ |y_1|<|y_2|<|x_2|\}
} K_\delta(x-y)\,dy\right)\,dx+
\int_{\{ |y_1|<|y_2|<r\}}\left(\int_{
\{ |x_1|<|x_2|<|y_2|\}
} K_\delta(x-y)\,dx\right)\,dy\\
&=& \int_{\{ |x_1|<|x_2|<r\}}\left(\int_{
\{ |y_1|<|y_2|<|x_2|\}
} K_\delta(x-y)\,dy\right)\,dx+
\int_{\{ |x_1|<|x_2|<r\}}\left(\int_{
\{\max\{|y_1|,|x_2|\}<|y_2|<r\}
} K_\delta(x-y)\,dy\right)\,dx\\
&=&\int_{\{ |x_1|<|x_2|<r\}}\left(\int_{
|y_1|<|y_2|<r\}
} K_\delta(x-y)\,dx\right)\,dy.
\end{eqnarray*}
Hence, plugging this information into~\eqref{g5GA7hAU9jsdHA},
we conclude that
$$ \int_{W_r}\int_{\R^2\setminus \mathcal{C}_{r}} K_\delta(x-y)\,dx\,dy-\int_{W_r}
\int_{\mathcal{C}} K_\delta(x-y)\,dx\,dy\le
-2\int_{W_r}\Psi_\delta(r(x))\,dx.$$
This and~\eqref{PLPLPLPAKJerUJS} give that
\begin{equation}\label{WrU7u7yehasddw}
\Per_\delta(\mathcal{C}_{r}, B_R) -\Per_\delta
(\mathcal{C}, B_R)\le-2\int_{W_r}\Psi_\delta(r(x))\,dx.
\end{equation}
Now, as $\delta\searrow 0$, we have that
$\Per_\delta(\mathcal{C}_{r}, B_R)\to \Per_K(\mathcal{C}_{r}, B_R)$ and
$\Per_\delta(\mathcal{C}, B_R)\to \Per_K(\mathcal{C}, B_R)$,
by Dominated Convergence Theorem,  see \cite{MR3401008}.   
Moreover, $\Psi_\delta(s)\to \Psi(s)=\int_{B_{s/4}(7s/4,0)} K(x)\,dx$
a.e. and in $L^1(0,1)$ by Dominated Convergence Theorem
(observe that $\Psi\in L^1(0,1)$ by assumption~\eqref{ROTAZION2}). 

So, letting $\delta\searrow 0$ in \eqref{WrU7u7yehasddw}, we end up with
\begin{equation} \label{9ik8ij9ik9iw9sidjhwihfhd}\Per_K(\mathcal{C}_{r}, B_R) -\Per_K(\mathcal{C}, B_R)\le
-2 \int_{W_r}\Psi (|x_2|) \,dx.\end{equation}
Recalling that $K$ is not identically zero, we take a Lebesgue point~$\tau_0\in(0,+\infty)$
such that~$K_0(\tau_0)>0$.
Then,
$$ \lim_{\e\searrow0}\frac1{2\e}\int_{\tau_0-\e}^{\tau_0+\e} K_0(\tau)\,d\tau=K_0(\tau_0)>0.$$
Consequently, we take~$\epsilon_0>0$ such that for all~$\epsilon\in(0,\epsilon_0]$ we have that
\begin{equation}\label{PALJA82uwshGAUw98di}
\int_{\tau_0-\e}^{\tau_0+\e} K_0(\tau)\,d\tau\ge \e K_0(\tau_0).\end{equation}
Then, if~$\bar\epsilon:=\min\left\{\epsilon_0,\frac{\tau_0}{100}\right\}$
and~$
r\in \left[ \frac{4\tau_0}{7}-\frac{\bar\epsilon}{14},
\frac{4\tau_0}{7}+\frac{\bar\epsilon}{14}\right]$,
we have that
\begin{equation}\label{9ikhHHSisdjjdjdj}
\begin{split}& \frac{7r}{4}+\frac{r}{8}=
\frac{15r}{8}\ge \frac{15\tau_0}{14}-\frac{15\bar\epsilon}{112}
\ge\tau_0+\bar\epsilon
\\{\mbox{and }}\;&
\frac{7r}{4}-\frac{r}{8}=\frac{13 r}{8}\le
\frac{13\tau_0}{14}+\frac{13\bar\epsilon}{112}
\le\tau_0-\bar\epsilon.
\end{split}\end{equation}
Now we cover the ring~$A_r:=B_{({7r}/{4})+({r}/{8})}\setminus
B_{({7r}/{4})-({r}/{8})}$ by $N_0$ balls of radius~$r/4$ centered at~$\partial B_{7r/4}$,
with~$N_0$ independent of~$r$. Then
\begin{eqnarray*} &&
\frac{13\pi r}{4}\int_{({7r}/{4})-({r}/{8})}^{({7r}/{4})+({r}/{8})} 
K_0(\tau)\,d\tau\le
2\pi\int_{({7r}/{4})-({r}/{8})}^{({7r}/{4})+({r}/{8})} \tau\,K_0(\tau)\,d\tau\\&&\qquad=
\int_{ A_r } K_0(|x|)\,dx\le N_0\,
\int_{B_{r/4}(7r/4,0)} K_0(|x|)\,dx=N_0\,\Psi(r),\end{eqnarray*}
thanks to~\eqref{defpsi}.

Using this, \eqref{PALJA82uwshGAUw98di}
and~\eqref{9ikhHHSisdjjdjdj}, we obtain that,
for any~$r\in \left[ \frac{4\tau_0}{7}-\frac{\bar\epsilon}{14},
\frac{4\tau_0}{7}+\frac{\bar\epsilon}{14}\right]$,
\begin{equation}\label{7hkajsdhfgH}\begin{split}
\Psi(r) \,&\ge \frac{13\pi r}{4N_0}\int_{({7r}/{4})-({r}/{8})}^{({7r}/{4})+({r}/{8})} 
K_0(\tau)\,d\tau\\
&\ge \frac{\tau_0}{4N_0}\int_{\tau_0-\bar\epsilon}^{\tau_0+\bar\epsilon} 
K_0(\tau)\,d\tau\\
&\ge \frac{\bar\e\tau_0\,K_0(\tau_0)}{4N_0}\\&=:\bar c.
\end{split}\end{equation}
Then, if~$r_0:=\frac{4\tau_0}{7}+\frac{\bar\epsilon}{14}$,
we have that
$$ W_{r_0}\supset \left(0,\frac{4\tau_0}{7}-\frac{\bar\epsilon}{14}\right)\times
\left(\frac{4\tau_0}{7}-\frac{\bar\epsilon}{14},
\frac{4\tau_0}{7}+\frac{\bar\epsilon}{14}\right)$$
and therefore
\begin{eqnarray*}
&& \int_{W_{r_0}}\Psi (|x_2|) \,dx\ge
\left(\frac{4\tau_0}{7}-\frac{\bar\epsilon}{14}\right)\,\int_{
\frac{4\tau_0}{7}-\frac{\bar\epsilon}{14}}^{
\frac{4\tau_0}{7}+\frac{\bar\epsilon}{14}} \Psi(x_2)\,dx_2\ge
\frac{\bar c\,\bar\e}{7}\;
\left(\frac{4\tau_0}{7}-\frac{\bar\epsilon}{14}\right),
\end{eqnarray*}
where~\eqref{7hkajsdhfgH} has been used in the last inequality.
In particular,
$$ \int_{W_{r_0}}\Psi (|x_2|) \,dx>0,$$
which combined with~\eqref{9ik8ij9ik9iw9sidjhwihfhd}
implies that
claim in~\eqref{claim} with~$r:=r_0$.

Then, in light of~\eqref{claim}, we get that 
$\mathcal{C}$ is not a $K$-minimal set, thus completing the proof
of Proposition \ref{mincroce}.
 \end{proof} 

\section{Strictly starshaped domains and proof of Proposition \ref{strictstar}} \label{secstar} 
\begin{proof}[Proof of Proposition \ref{strictstar}]
We observe that, due to assumption in \eqref{strictstar1}, for every $\lambda>0$, we have that
there exists $\delta_\lambda>0$ such that the distance
between~$\partial E$ and~$\partial(\lambda E)$
is at least~$\delta_\lambda$. Therefore, for any $\lambda>1$, 
from Corollary \ref{cpgeometric} and Lemma \ref{A-102-B}, we deduce that 
\[E^+(\lambda^{1+s}t)\subseteq E^-_\lambda(\lambda^{1+s}t)= \lambda E^-\left(t\right). \]
Then for $\lambda>1$, 
\[|{\rm{int}} (E^+(t))\setminus \overline{E^-(t)}|\leq |{\rm{int}} (E^+(t))\setminus \lambda^{-1} E^+(\lambda^{1+s}t)|=|{\rm{int}} (E^+(t))|-\lambda^{-1} | E^+(\lambda^{1+s}t)|. \]
Also, by Proposition \ref{scs}, $$
\liminf_{\lambda\searrow 1}  | E^+(\lambda^{1+s}t)|\geq |{\rm{int}} (E^+(t))|.$$ 
Therefore we get 
\[|{\rm{int}} (E^+(t))\setminus \overline{E^-(t)}|\leq \limsup_{\lambda\searrow 1}|{\rm{int}} (E^+(t))|-\lambda^{-1} | E^+(\lambda^{1+s}t)|=
|{\rm{int}} (E^+(t))|-\liminf_{\lambda\searrow 1} \lambda^{-1} | E^+(\lambda^{1+s}t)| \leq 0.\] 
This gives the desired statement. \end{proof}

\section{Perturbed double droplet
and proof of Theorem~\ref{FAT3}}\label{secdoppiagoccia} 
In this section, the state space is $\R^2$. 
Recalling the notation in~\eqref{DDR}, given~$r\in\left(0,\frac12\right)$ we set
\begin{equation}\label{DUPA} {\mathcal{G}}_r:= [-r,r]^2\cup {\mathcal{G}}_0\subseteq \R^2,\end{equation}
where ~${\mathcal{G}}_0$
is the union in $\R^2$ of~${\mathcal{B}}^{+}$, which is the convex envelope between~$B_1(\sqrt2,0)$
and the origin, and~${\mathcal{B}}^{-}$, which is the convex envelope between~$B_1(-\sqrt2,0)$
and the origin, see  
Figure~\ref{P8234}. 

\begin{figure}
    \centering
    \includegraphics[width=8cm]{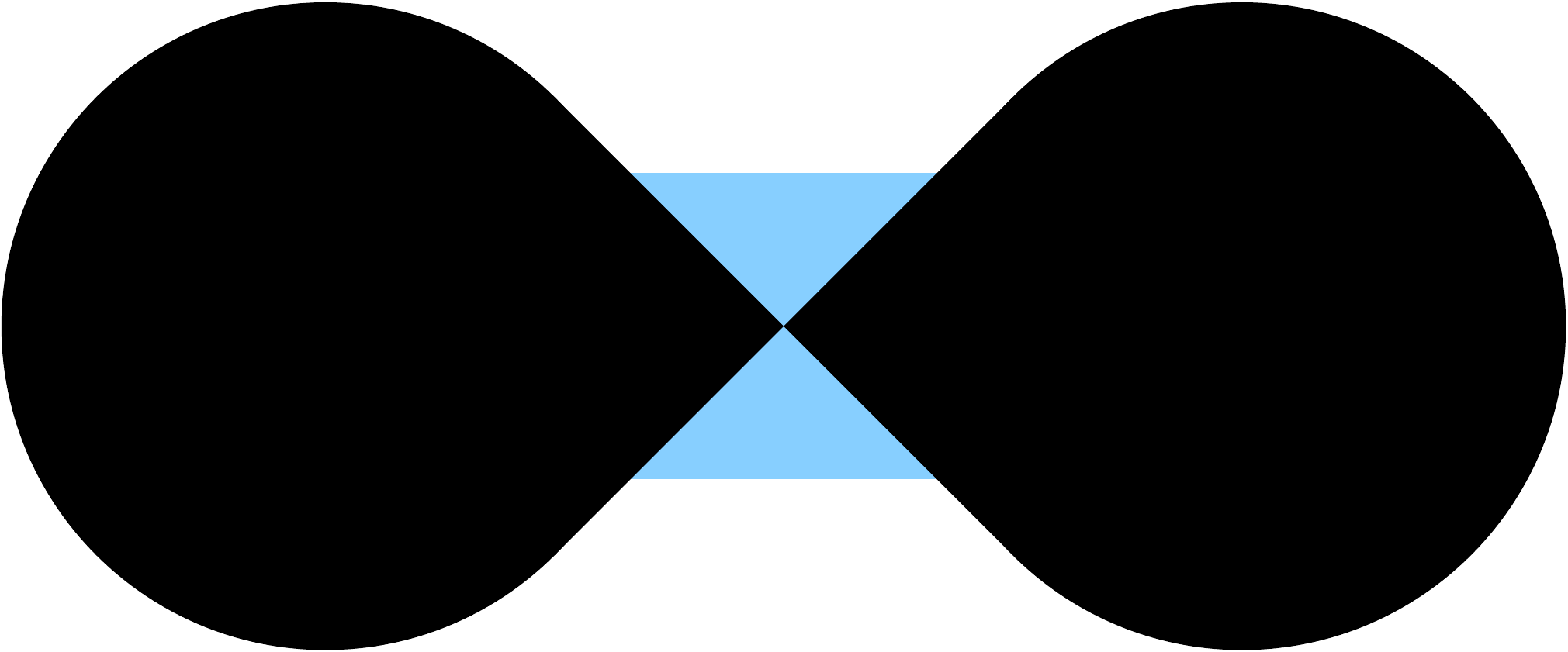}
    \caption{\em {{The set~${\mathcal{G}}_r$.}}}
    \label{P8234}
\end{figure}

Now, fixed~$\delta\in(0,r)$, we denote
by~${\mathcal{B}}_{\delta}^+$
the convex envelope between~$B_{1-\delta}(\sqrt2,0)$
and the origin, and~${\mathcal{B}}^{-}_\delta$
the convex envelope between~$B_{1-\delta}(-\sqrt2,0)$
and the origin. We let
$$ {\mathcal{G}}_{\delta,r}:=\big([-2r,2r]\times[-r,r]\big)\cup{\mathcal{B}}_{\delta}^+\cup{\mathcal{B}}_{\delta}^-.$$
Then we can estimate the $K$-curvature of~${\mathcal{G}}_{\delta,r}$ as follows:

\begin{lemma}\label{9rierhgj4848dfg85}
Assume that~\eqref{ROTAZION},
\eqref{ROTAZION2} and~\eqref{NUCLEONORMA} hold true in $\R^2$. Then,
there exists~$c_\sharp\in(0,1)$ such that the following statement holds true.
If~$r\in(0,c_\sharp)$ and~$\delta\in(0,c_\sharp^4 r)$, then
\begin{equation}\label{G021} H^s_{ {\mathcal{G}}_{\delta,r} }(p)\le\frac{1}{c_\sharp}\end{equation}
for any~$p\in\partial{ {\mathcal{G}}_{\delta,r} }$. In addition,
for any~$p\in(\partial{ {\mathcal{G}}_{\delta,r} })\cap([-2r,2r]\times[-r,r])$,
\begin{equation}\label{G022}
H^s_{ {\mathcal{G}}_{\delta,r} }(p)\le-\frac{c_\sharp}{r^s}.\end{equation}
\end{lemma}

\begin{proof}
Let $\alpha(\delta) $ the angle at $x=0$ in $\mathcal{B}_\delta^+$. Observe that  when $\delta=0$, this angle is $\pi/2$ and moreover there exist $\delta_0$ and $C_0>0$ such that
$|\alpha(\delta)-\frac{\pi}{2}| \leq C_0\delta$, for all $0<\delta<\delta_0$. In particular we may assume that $\alpha(\delta)\geq \pi/3$. 
We fix then $\delta\leq r<\delta_0$.   

First of all note that for all $p=(p_1,p_2)\in\partial{ {\mathcal{G}}_{\delta,r} }$, with $p_1\geq \sqrt{2}-\frac{(1-\delta)^2}{\sqrt{2}}$ (resp. $p_1\leq -\sqrt{2}+\frac{(1-\delta)^2}{\sqrt{2}}$), then $p\in \partial B_{1-\delta}(\sqrt{2},0)$ (resp. $p\in \partial B_{1-\delta}(-\sqrt{2},0)$), and then 
\[ H^s_{ {\mathcal{G}}_{\delta,r} }(p)\leq H^s_{B_{1-\delta}(\sqrt{2},0)}(p)=c(1)(1-\delta)^{-s}\qquad \left(\text{resp. }
H^s_{ {\mathcal{G}}_{\delta,r} }(p)\leq H^s_{B_{1-\delta}(-\sqrt{2},0)}(p)=c(1) (1-\delta)^{-s}\right)\]
where $c(1)= H^s_{B_1}$. 

We take~$c_\sharp\in(0,1)$ to be taken conveniently small in what follows.
We notice that~${\mathcal{S}}:=(\partial{\mathcal{G}}_{\delta,r})\cap\{|x_2|=r\}$ consists of four points. We take~$p=(p_1,p_2)\in\partial{ {\mathcal{G}}_{\delta,r} }$
such that there exists~$q\in{\mathcal{S}}$ such that $|p-q|<c_\sharp r$ (see e.g. Figure~\ref{PCAS1P} for a possible configuration).

Then,
\begin{equation}\label{INCOC}\begin{split}
&\lim_{\e\searrow0}
\int_{B_{\sqrt{c_\sharp}\,r}(p)\setminus B_\e(p)}
\Big(\chi_{\R^2\setminus { {\mathcal{G}}_{\delta,r} }}(y)-
\chi_{ {\mathcal{G}}_{\delta,r} }(y)\Big) \frac{1}{|p-y|^{2+s}}\,dy\\
&\qquad \le - \iint_{(0,\pi/6)\times(
{c_\sharp}\,r,\,\sqrt{c_\sharp}\,r)}
\frac{1}{\varrho^{1+s}}\,d\vartheta\,d\rho=-
\frac{\pi}{6s}\frac{1}{c_\sharp^{s/2}r^s}\left(  \frac{1}{c_\sharp^{s/2}} -1 \right),
\end{split}\end{equation}
while
$$ \int_{\R^2\setminus B{\sqrt{c_\sharp}\,r}(p)}
\Big(\chi_{\R^2\setminus { {\mathcal{G}}_{\delta,r} }}(y)-
\chi_{ {\mathcal{G}}_{\delta,r} }(y)\Big) \frac{1}{|p-y|^{2+s}}
\le 2\pi\int_{ \sqrt{c_\sharp} r }^{+\infty}\frac{1}{\rho^{1+s}}\,d\rho= \frac{2\pi}{s}\frac{1}{c_\sharp^{s/2}r^s}.
$$
As a consequence,
$$ H^K_{ {\mathcal{G}}_{\delta,r} }(p)\le-
\frac{\pi}{6s}\frac{1}{c_\sharp^{s/2}r^s}\left(  \frac{1}{c_\sharp^{s/2}}-1\right)
+
\frac{2\pi}{s}\frac{1}{c_\sharp^{s/2}r^s}\leq -c_\sharp \frac{1}{r^s}$$
as long as~$c_\sharp$
is sufficiently small,
which implies~\eqref{G022} (and also~\eqref{G021}) in this case.

Now consider $p \in\partial{ {\mathcal{G}}_{\delta,r} }$ such that $p_2\neq r$ and $d(p,\mathcal{S})\geq c_\sharp r$. If $p\in\partial B_{1-\delta}(\pm\sqrt{2},0)$ we are ok,
and in the other case, note that we can define a set ${ { {\mathcal{G}} } }'$
with $C^{1,1}$-boundary
(uniformly in~$\delta$ and~$r$) such that~${ { {\mathcal{G}}_{\delta,r} } }\subset
{ { {\mathcal{G}}  } }'$ and~${ { {\mathcal{G}} } }'\setminus B_{1/8}=
{ { {\mathcal{G}}_{\delta,r} } }\setminus B_{1/8}$.
Then, we obtain that
$$ C'\ge H^K_{ { {\mathcal{G}} }' }(p)\ge H^K_{ { {\mathcal{G}}_{\delta,r} } }(p)-C'',$$
for some~$C'$, $C''>0$, depending only on the local $C^{1,1}$-norms
of the boundary of~${ { {\mathcal{G}}  } }'$, and this gives~\eqref{G021}
in this case.

Finally note
that $ {\mathcal{G}}_{\delta,r}\subseteq \mathcal{C}_r$, where $\mathcal{C}_r$ is the perturbed cross is defined in \eqref{PCROSS}. So, if  $p\in\partial{ {\mathcal{G}}_{\delta,r} } \cap([-r,r]\times[-r,r])$, then $p\in \partial\mathcal{C}_r$. Moreover by  Lemma \ref{PC=2}  and the definition of $\Psi$ in \eqref{defpsi}
\[ H^s_{ {\mathcal{C}}_{r} }(p)\leq -2\Psi(r)=-C \frac{1}{r^s}\]
where $C>0$ is a universal constant. In this case, we notice that~${ {\mathcal{G}}_{\delta,r} }$
and~${\mathcal{C}}_r$ coincide in $B_r$,
and, outside such a neighborhood of the origin, they differ by four
portions of cones (passing in the vicinity of~${\mathcal{S}}$)
with opening bounded by~$C_0\delta$.
That is, if we set
$$ { {\mathcal{D}}_{\delta,r} }:=\big( { {\mathcal{G}}_{\delta,r} }\setminus
{\mathcal{C}}_r\big)\cup\big({\mathcal{C}}_r\setminus
{ {\mathcal{G}}_{\delta,r} }\big),$$
we have that
\begin{eqnarray*}&&
\int_{{ {\mathcal{D}}_{\delta,r} }}
\frac{dy}{|p-y|^{2+s}}\le
C_2\,\left[
\iint_{(0,C_1\delta)\times(c_\sharp r/2, 10r]}\frac{\rho\,d\vartheta\,d\rho}{
(c_\sharp \,r/2)^{2+s}}
+\iint_{(0,C_1\delta)\times( 10r,+\infty)}\frac{\rho\,d\vartheta\,d\rho}{
\rho^{2+s}}
\right]
\le \frac{C_3\,\delta}{c_\sharp^{2+s}\,r^s}
\le \frac{c_\sharp}{r^s},
\end{eqnarray*}
thanks to our assumption on~$\delta$. Consequently
$$ \big| H^K_{ { {\mathcal{G}}_{\delta,r} } }(p)-H^K_{ {\mathcal{C}}_r }(p)\big|\le
\frac{c_\sharp}{r^s}$$
and so, making use of~\eqref{PRO2} and~\eqref{NOME},
$$ H^K_{ { {\mathcal{G}}_{\delta,r} } }(p) \le H^K_{ {\mathcal{C}}_r }(p)+
\frac{c_\sharp}{r^s}\le-\frac{c_*}{r^s}+
\frac{c_\sharp}{r^s}\le-
\frac{c_*}{2\,r^s},$$
for a suitable~$c_*>0$,
as long as~$c_\sharp>0$ is sufficiently small.
This establishes~\eqref{G022} (and also~\eqref{G021}) in this case. \end{proof}

\begin{figure}
    \centering
    \includegraphics[width=8cm]{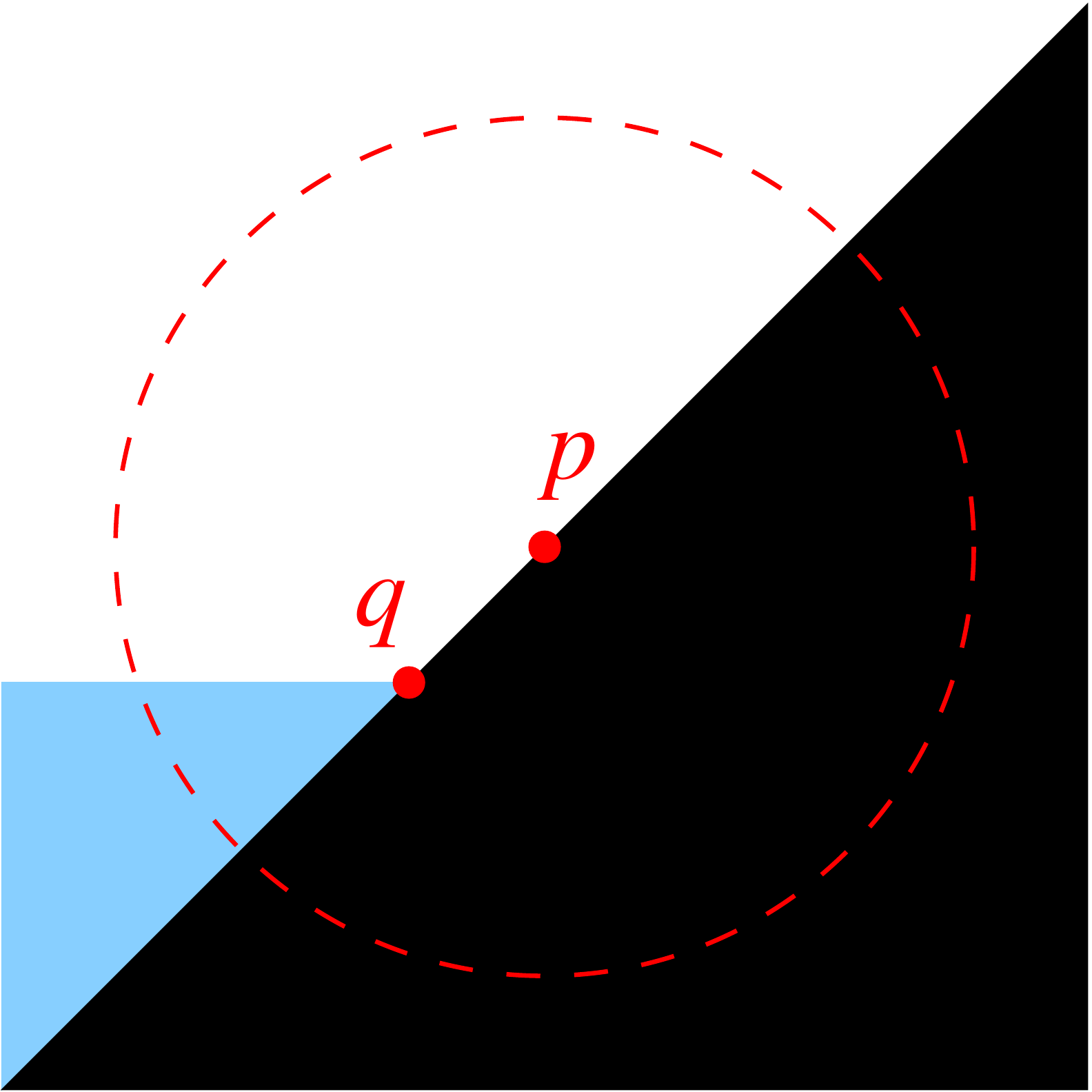}
    \caption{\em {{ .}}}
    \label{PCAS1P}
\end{figure}

With these auxiliary computations, we can now complete the proof
of Theorem~\ref{FAT3}, by arguing as follows.

\begin{proof}[Proof of Theorem~\ref{FAT3}]
Let~$c_\sharp>0$ be as in Lemma~\ref{9rierhgj4848dfg85}, $0<\eps<c_\sharp/2$
and $c_\star:=((c_\sharp-\eps)\,(1+s))^{1/(1+s)}$. We define $r(t)$ such that $\dot r(t)= (c_\sharp-\eps)  r(t)^{-s}$,
with $r(0)=0$. So, we have that $r(t)=c_\star t^{1/(1+s)}$.
Let also
$$ \delta(t):=\frac1{c_\sharp\eps}\,\int_0^t \frac{d\tau}{r(\tau)}=
\frac{1+s}{(c_\sharp-\eps)\,c_\star\,s}\; t^{s/(1+s)}
.$$
We now estimate the outer normal velocity of~${ { { {\mathcal{G}}_{\delta(t),r(t)} } } }$
via Lemma~\ref{9rierhgj4848dfg85}.
First of all, from~\eqref{G022} at $p\in (\partial{ { { {\mathcal{G}}_{\delta(t),r(t)} } } })\cap \{|x_2|=r(t), |x_1|<\sqrt{2}\}$
we get 
$$ \dot r(t)=\frac{c_\sharp-\eps}{(r(t))^s}\leq -H^s_{ {\mathcal{G}}_{\delta(t),r(t)} }(p)-\frac{\eps}{c_\star^s t^{s/(1+s)}}.$$

Moreover, the shrinking velocity at~$x\in(\partial{ { { {\mathcal{G}}_{\delta(t),r(t)} } } })
\setminus\{|x_2|=r(t)\}$
is at least~$r(t)\dot\delta(t)=1/(c_\sharp-\eps)$. This implies that 
at every $x\in(\partial{ { { {\mathcal{G}}_{\delta(t),r(t)} } } })
\setminus\{|x_2|=r(t)\}$
we get
\[\partial_t x\cdot \nu(x)\leq -\frac{1}{c_\sharp-\eps}\leq  -H^s_{ {\mathcal{G}}_{\delta(t),r(t)} }(x)-\frac{\eps}{c_\sharp(c_\sharp-\eps)}.\]
by \eqref{G021}.
Therefore, by Proposition \ref{subgeometrico}, we get that  
\begin{equation}\label{0923e0234xcde-BIS}B_{ c_\star\,t^{1/(1+s)} }\subseteq
{ { { {\mathcal{G}}_{\delta(t),r(t)} } } }\subseteq
 \mathcal{G}^+(t).\end{equation}
Conversely, since~${\mathcal{G}}$ is contained in the cross~${\mathcal{C}}$,
it follows from Corollary \ref{cpgeometric} and 
Theorem~\ref{FAT1} that
$$ B_{c_o t^{1/(1+s)}}\subseteq (\R^2\setminus\mathcal{C})^+(t)\subseteq (\R^2\setminus{\mathcal{G}})^+(t).  
$$
{F}rom this and~\eqref{0923e0234xcde-BIS} it follows that
$$ B_{\hat{c}\, t^{1/(1+s)}}\subseteq \mathcal{G}^+(t)\cap (\R^2\setminus{\mathcal{G}})^+(t)=\Sigma_{\mathcal{G}}(t)$$
with~$\hat{c}:=\min\{ c_\star,\,c_o\}$,
which proves~\eqref{9238475y6uedjidftgeri:T3}.
\end{proof}

\section{Perturbation of tangent balls and proof of Theorem~\ref{DUPAT}}\label{sectangenti}

Also in this Section, the state space is $\R^2$. The idea to prove Theorem~\ref{DUPAT} is to construct inner barriers
using ``almost tangent'' balls and take advantage of the scale invariance given
by the homogeneous kernels in~\eqref{NUCLEONORMA}. For this, given~$\delta\in\left[0,\frac18\right]$,
we consider the set
$$ {\mathcal{Z}}_{\delta,r}:= B_r\big( (1+\delta)r,0\big)\cup B_r\big( (-1-\delta)r,0\big)\subseteq \R^2.$$
Then, we have that the nonlocal curvature of~${\mathcal{Z}}_{\delta,r}$
is always controlled from above by that of the ball, and it becomes negative
in the vicinity of the origin. More precisely:

\begin{lemma}\label{NUK}
Assume~\eqref{NUCLEONORMA} with $n=2$.
Then, for any~$p\in\partial{\mathcal{Z}}_{\delta,r}$ we have that
\begin{equation}\label{ZT-1}
H^K_{ {\mathcal{Z}}_{\delta,r} }(p)\le \frac{C}{r^s},
\end{equation}
for some~$C>0$.
In addition, there exists~$c\in(0,1)$ such that if~$\delta\in(0,c^2)$ and~$
p\in(\partial{\mathcal{Z}}_{\delta,r})\cap B_{cr}$
then
\begin{equation}\label{ZT-2}
H^K_{ {\mathcal{Z}}_{\delta,r} }(p)\le-c. 
\end{equation}
\end{lemma}

\begin{proof} Notice that~$\partial{\mathcal{Z}}_{\delta,r}\subseteq\big(
\partial B_r\big( (1+\delta)r,0\big)\big)\cup\big(\partial B_r\big( (-1-\delta)r,0\big)
\big)$. Moreover, ${\mathcal{Z}}_{\delta,r}\supseteq B_r\big( (1+\delta)r,0\big)$,
as well as~${\mathcal{Z}}_{\delta,r}\supseteq B_r\big( (-1-\delta)r,0\big)$,
hence, in view of~\eqref{CURVA}, the nonlocal curvature of~${\mathcal{Z}}_{\delta,r}$
is less than or equal to that of~$B_r$, which proves~\eqref{ZT-1}.

Now we prove~\eqref{ZT-2}. For this, up to scaling, we assume that~$r:=1$
and we take~$p\in(\partial{\mathcal{Z}}_{\delta,1})\cap B_{c}$.
Without loss of generality,
we also suppose that~$p_1$, $p_2>0$ and we observe that 
\begin{equation}\label{QRL}
B_c(-2c,0)\subseteq
B_1 \big( (-1-\delta),0\big),
\end{equation}
as long as~$c$ is small enough. 
Indeed if~$x\in B_c(2c,0)$ then we can write~$x=-2ce_1+ce$, for some~$e\in \S^1$, and so
$$ |x-(-1-\delta)e_1|=|(1+\delta-2c)e_1+ce|\le |1+\delta-2c|+c=(1+\delta-2c)+c=1+\delta-c\le
1+c^2-c<1.$$
This proves~\eqref{QRL}.

Hence, from~\eqref{CURVA} and~\eqref{QRL},
the nonlocal curvature of~${\mathcal{Z}}_{\delta,1}$ at~$p$
is less than or equal to the nonlocal
curvature of~$B_r\big( (1+\delta),0\big)$, which is bounded by some~$C>0$,
minus the contribution coming from~$B_c(-2c,0)$.
That is,
\begin{equation} \label{9340545jdwie8838383}-H^K_{ {\mathcal{Z}}_{\delta,r} }(p)\ge -C+
\int_{ B_c(-2c,0) }\frac{dx}{|x-p|^{2+s}}=
-C+
\int_{ B_c(2c+p_1,p_2) }\frac{dx}{|y|^{2+s}}.\end{equation}
Also, if~$y\in B_c(2ce_1+p_1, p_2)$, we have that~$|y|\le |y-2ce_1-p|+|2ce_1+p|\le
c+2c+|p|\le 4c$, and so
$$ \int_{ B_c(2c+p_1,p_2) }\frac{dx}{|y|^{2+s}}\ge \frac{c_0\, c^2}{c^{2+s}}=\frac{c_0}{c^s},$$
for some~$c_0>0$.
So we insert this information into~\eqref{9340545jdwie8838383}
and we obtain
$$ -H^K_{ {\mathcal{Z}}_{\delta,r} }(p)\geq -C+\frac{c_0}{c^s}\ge\frac{c_0}{2c^s}$$
as long as~$c$ is sufficiently small.
This completes the proof of~\eqref{ZT-2}, as desired.
\end{proof}

{F}rom Lemma~\ref{NUK}, we can control the geometric flow of the double tangent balls
from inside with barriers that shrink the sides of the picture and make the origin emanate some mass:

\begin{lemma} \label{A-102-A}
There exist~$\delta_0\in(0,1)$,
and~$\bar C>0$ such that if~$\delta\in(0,\delta_0)$, then
\[{\mathcal{O}}^-(\delta)\supset \bigcup_{\sigma\in(-\delta^2,\delta^2)}\big(
B_{1-\bar C\delta}
\big(1-\bar C\delta, 0\big)
\cup B_{1-\bar C\delta}\big(-1+\bar C\delta, 0\big)+\sigma e_2
\Big).\end{equation*}
\end{lemma}

\begin{proof} Fix~$\e\in(0,1)$, to be taken arbitrarily small in what follows.
Let~$\mu\in [0,\sqrt\e]$ and let, for any~$t\in [0,(1-\e)/C_0)$,
$$ \e_\mu(t):= \e-\mu t \qquad{\mbox{and}}\qquad r(t):=1-\e-C_0\,t,$$
with~$C_0>0$ to be chosen conveniently large.
We consider an inner barrier consisting in two balls
of radius~$r(t)$ which, for any~$t\in[0,(1-\e)/C_0)$, remain at distance~$2\e(t)$.
Namely, we set
\begin{equation}\label{DEF F EM}
{\mathcal{F}}_{\e,\mu}(t):= B_{r(t)}\big(r(t)+\e_\mu(t),0\big)
\cup B_{r(t)}\big(-r(t)-\e_\mu(t),0\big).\end{equation}
Notice that
\begin{equation}\label{COMP9345702}
{\mathcal{O}}\supseteq {\mathcal{F}}_{\e,\mu}(0)+\sigma e_2
\qquad{\mbox{ for any }}\sigma\in(-\e,\e)\qquad d(\mathcal{O},  {\mathcal{F}}_{\e,\mu}(0)+\sigma e_2)>0.\end{equation}
We also observe that the vectorial velocity of this set is
the superposition of a normal velocity~$-\dot r \nu$, being~$\nu$ the interior normal,
and a translation velocity~$\pm(\dot r+\dot\e_\mu)e_1$, with the plus sign
for the ball on the right and the minus sign for the ball on the left. The normal
velocity of this set is therefore equal to
\begin{equation}\label{901-VEL} \Big(-\dot r \nu\pm(\dot r+\dot\e_\mu)e_1\Big)\cdot\nu
=-\dot r\pm (\dot r+\dot\e_\mu)\nu_1
= C_0\,(1\mp\nu_1)\mp\mu.
\end{equation}
Now, taken a point~$p$ on~$\partial{\mathcal{F}}_{\e,\mu}(t)$,
we distinguish two cases. Either~$p\in B_c$, where~$c$ is the one given
in Lemma~\ref{NUK}, or~$p\in\R^2\setminus B_c$. In the first case,
we have that
$$ C_0\,(1\mp\nu_1)\mp\mu\ge C_0\,(1-|\nu_1|)-\mu\ge 0-\mu\ge -\sqrt\e >-c,$$
This and~\eqref{901-VEL}
give that the normal
velocity of~${\mathcal{F}}_{\e,\mu}(t)$ at~$p$ is larger than~$-c$, and therefore greater
than~$H^K_{ {\mathcal{Z}}_{\delta,r} }(p)$, thanks to~\eqref{ZT-2}.

If instead~$p\in\R^2\setminus B_c$,
we have that~$|\nu_1(p)|\le 1-c_0$, for a suitable~$c_0\in (0,1)$, depending on~$c$,
and therefore
\begin{equation*}
C_0\,(1\mp\nu_1)\mp\mu\ge C_0\,(1-|\nu_1|)-\mu\ge C_0\,c_0-\mu\ge C_0\,c_0-1 \ge\frac{C_0\,c_0}{2}
\ge
\frac{C_0\,c_0}{2^{s+1} \,(r(t))^s},\end{equation*}
as long as~$C_0$ is sufficiently large.
This and~\eqref{901-VEL}
give that the inner  normal
velocity of~${\mathcal{F}}_{\e,\mu}(t)$ at~$p$ is strictly larger than~$\frac{C_0\,c_0}{2^{s+1} \,(r(t))^s}$,
which, if~$C_0$ is chosen conveniently big, is in turn strictly larger
than~$H^K_{ {\mathcal{Z}}_{\delta,r} }(p)$, thanks to~\eqref{ZT-1}.

In any case, we have shown that the inner  normal
velocity of~${\mathcal{F}}_{\e,\mu}(t)$ at~$p$ is strictly larger
than~$H^K_{ {\mathcal{Z}}_{\delta,r} }(p)$. This implies that ${\mathcal{F}}_{\e,\mu}(t)$ is a strict subsolution according to Proposition \ref{subgeometrico}. 

Then, by \eqref{COMP9345702} and Proposition \ref{subgeometrico}
\begin{equation}\label{9w:OSOSX10}
{\mathcal{O}}^-(t) \supseteq \bigcup_{\sigma\in(-\e,\e)}
\Big({\mathcal{F}}_{\e,\mu}(t)+\sigma e_2
\Big),\end{equation}
for any~$t\in[0,(1-\e)/C_0)$.

Now, taking~$\mu:=\sqrt\e$ in~\eqref{DEF F EM}, we see that
\begin{equation*}
{\mathcal{F}}_{\e,\sqrt\e}(t)
=
B_{1-\e-C_0t}\big(1-\sqrt\e t-C_0 t,0\big)
\cup B_{1-\e-C_0t}\big(-(1-\sqrt\e t-C_0 t),0\big)
\end{equation*}
for all~$t\in [0,\,(1-\e)/C_0]$. 
In particular, taking~$t:=\sqrt\e$,
\begin{equation*}
{\mathcal{F}}_{\e,\sqrt\e}(\sqrt\e)
=
B_{1-\e-C_0\sqrt\e}\big(1-\e-C_0 \sqrt\e,0\big)
\cup B_{1-\e-C_0\sqrt\e}\big(-(1-\e-C_0 \sqrt\e),0\big),
\end{equation*}
and the latter are two tangent balls at the origin. {F}rom this and~\eqref{9w:OSOSX10},
we deduce that
\begin{equation*}
{\mathcal{O}}^-(\sqrt\e) \supseteq \bigcup_{\sigma\in(-\e,\e)}
\Big(B_{1-\e-C_0\sqrt\e}\big(1-\e-C_0 \sqrt\e,0\big)
\cup B_{1-\e-C_0\sqrt\e}\big(-(1-\e-C_0 \sqrt\e),0\big)+\sigma e_2
\Big),\end{equation*}
and this implies the desired result by choosing~$\delta:=\sqrt\e$
and~$\bar C:= 2(C_0+1)$.
\end{proof}

We can now complete the proof of Theorem~\ref{DUPAT} in the following way:

\begin{proof}[Proof of Theorem~\ref{DUPAT}]
We observe that, in the setting of Lemma~\ref{A-102-B},
the result in Lemma~\ref{A-102-A} can be written as
$$ {\mathcal{O}}^-(\delta)\supseteq (1-\bar C\delta){\mathcal{O}}^+(0)\qquad\text{with }\quad d( {\mathcal{O}}^-(\delta), (1-\bar C\delta){\mathcal{O}}^+(0))\geq \delta^2$$
for all $\delta\in (0,\delta_0)$.

Fix now $C\ge \bar C$ and let 
${\mathcal{U}}:=(1-C\delta){\mathcal{O}}^+(0)$. Then,
by Corollary \ref{cpgeometric}, we have
\begin{equation}\label{9wo02e00} {\mathcal{O}}^-(t+\delta)\supseteq {\mathcal{U}}(t)\end{equation}
for all~$t\ge0$.

Now, in view of Lemma~\ref{A-102-B},
$$ {\mathcal{U}}(t)=(1-C\delta)\;{\mathcal{O}}^+\left( \frac{t}{(1-C\delta)^{1+s}}\right)$$
and so, combining with~\eqref{9wo02e00},
$$ {\mathcal{O}}^-(t+\delta)\supseteq(1-
C\delta)\;{\mathcal{O}}^+\left( \frac{t}{(1-C\delta)^{1+s}}\right).$$
Consequently, for any~$t\ge\delta$,
we can estimate the measure of the fattening set as
\begin{equation}\label{93-LWS0220}
\begin{split}
\big| {\rm{int}}\left({\mathcal{O}}^+(t)\right) \setminus \overline{{\mathcal{O}}^-(t)}\big|\,&\le
\left| {\rm{int}}\left({\mathcal{O}}^+(t)\right)\setminus
(1-C\delta)\;{\mathcal{O}}^+\left( \frac{t-\delta}{(1-C\delta)^{1+s}}\right)
\right|\\
&= \left| {\rm{int}}\left({\mathcal{O}}^+(t)\right)\right|-\left|
(1-C\delta)\;{\mathcal{O}}^+\left( \frac{t-\delta}{(1-C\delta)^{1+s}}\right)
\right|\\
&= \left| {\rm{int}}\left({\mathcal{O}}^+(t)\right)\right|-(1-C\delta)^2\,\left|
{\mathcal{O}}^+\left( \frac{t-\delta}{(1-C\delta)^{1+s}}\right)
\right|.\end{split}
\end{equation}
We now fix $t_0\ge \delta$ and choose $C=C(t_0)\ge \bar C$ such that 
\[
t\le \frac{t-\delta}{(1-C\delta)^{1+s}}\qquad \text{for all }t\ge t_0\,.
\]
So, by Proposition \ref{scs}, we get that 
\[
\liminf_{\delta\searrow0}
\left| {\mathcal{O}}^+\left( \frac{t-\delta}{(1-C\delta)^{1+s}}\right) \right| 
\geq \big| {\rm{int}}\left({\mathcal{O}}^+(t)\right)\big|.
\]
This and~\eqref{93-LWS0220} yield that, for $t\ge t_0$,
\begin{eqnarray*}
&& \big| {\rm{int}}\left({\mathcal{O}}^+(t)\right) - \overline{{\mathcal{O}}^-(t)}\big|
\le \limsup_{\delta\searrow0}
\left| {\rm{int}}\left({\mathcal{O}}^+(t)\right)\right|-(1-C\delta)^2\,\left|
{\mathcal{O}}^+\left( \frac{t-\delta}{(1-C\delta)^{1+s}}\right)
\right|
\\ &&\qquad =
\left| {\rm{int}}\left({\mathcal{O}}^+(t)\right)\right|-\liminf_{\delta\searrow0}(1-C\delta)^2\,\left|
{\mathcal{O}}^+\left( \frac{t-\delta}{(1-C\delta)^{1+s}}\right)
\right|\le0.
\end{eqnarray*}
Since $t_0$ was chosen arbitrarily, this completes the proof of Theorem~\ref{DUPAT}.
\end{proof}

\begin{appendix}

\section{Viscosity solutions and geometric barriers}\label{viscosec}
In this appendix, we recall the existence and uniqueness results about the
level set flow associated to the nonlocal evolution
\eqref{kflow}, and we provide some auxiliary results which will be useful in the proof of the main theorems. All the results hold in $\R^n$ for $n\geq 2$. 

Before introducing the level set equation and the notion
of viscosity solutions, we briefly 
discuss the evolution of balls according to the setting in~\eqref{kflow}--\eqref{CURVA}. 

\begin{lemma}\label{lemmacerchi}
Assume that~\eqref{ROTAZION} and
\eqref{ROTAZION2} hold true.  Then for every $R>0$ there
exists $c(R)> 0$ such that 
\[H_{B_R}^K(x)=c(R)\qquad {\mbox{for all }}x\in \partial B_R.  \]
Moreover  the function \[R\in (0, +\infty) \to c(R)\in (0, +\infty)\] is continuous, 
nonincreasing  and such that  $$\lim_{R\to +\infty} c(R)=0.$$
Furthermore,
if $K$ is a fractional kernel, that is $K(x)=\frac{1}{|x|^{n+s}}$, then
$c(R)= c(1) R^{-s}$. 
\end{lemma} 
\begin{proof} We observe that, in virtue of~\eqref{ROTAZION}
and~\eqref{CURVA}, 
and the fact that~$K(x)\not\equiv 0$, we have that~$H^K_{B_R}(x)>0$ for $x\in\partial B_R$,
it does not depend on $x$, and finally $H^K_{B_R}(x)\geq H^K_{B_R'}(x)$ if $R'>R$. Condition \eqref{ROTAZION2} assures that 
$H^K_{B_R}(x)$ is finite for every $R>0$ and also
that $$\lim_{R\to +\infty} c(R)=0,$$
see ~\cites{MR2487027, MR3401008}.
For the computation in the case of fractional kernels, see \cite{SAEZ}.   \end{proof} 

\begin{remark}\label{palla} \upshape
Using Lemma~\ref{lemmacerchi}, we study the evolution of a ball $B_R$
according to the flow in~\eqref{kflow}.
Such evolution is given by a ball $B_{R(t)}$,
where \begin{equation}\label{gvhdiuytbf78uewyr}
\dot R(t)= -c(R(t))\end{equation}
with initial datum $R(0)=R$. 
We define $$C(R):=\int_1^R \frac{1}{c(s)} ds.$$
Then  $C(R)$ is a monotone increasing function, and the
solution $R(t)$ to \eqref{gvhdiuytbf78uewyr} is given implicitly by the formula
\begin{equation}\label{evball} C(R(t))= C(R)-t\qquad {\mbox{ for all }} t>0,
\ \text{ s.t. } \ R(t)>0.\end{equation} 
Let also
$$T_R:=\sup\{t>0\ |\ R(t)>0\}.$$
By~\eqref{gvhdiuytbf78uewyr} and the monotonicity of~$c(\cdot)$,
it is easy to check that $$T_R\leq \frac{R}{c(R)}.$$
Moreover, from~\eqref{gvhdiuytbf78uewyr} we have that
$$T_R=\lim_{\eps\searrow 0} \int_\eps^R \frac{1}{c(s)} ds=C(R)
-\lim_{\eps\searrow0}C(\eps).$$
If $K$ is a fractional kernel, that is $K(x)=\frac{1}{|x|^{n+s}}$, then $C(R)=\frac{1}{c(1)(s+1)}(R^{s+1}-1)$ and $T_R=\frac{R^{s+1}}{c(1)(s+1)}$. \end{remark} 
 
We introduce now the notion of viscosity solutions for the level set equation 
\begin{equation}\label{levelset1}
\begin{cases}
\partial_t u(x,t)+|Du(x,t)|  H^K_{\{y| u(y,t)\geq u(x,t)\}}(x)=0 & {\mbox{for all }}x\in\R^n, \quad t>0,\\
u(x,0)= u_0(x)&{\mbox{for all }}x\in\R^n.
\end{cases} 
\end{equation}
For more details, we refer to \cites{MR2487027, MR3401008}.
The viscosity theory for the classical  mean curvature flow is
contained in~\cite{MR1100211}, see also~\cite{MR2238463}
for a comprehensive level set approach for classical geometric
flows.
 
 \begin{definition}[Viscosity solutions] \label{viscosoldefi} $\,$
\begin{itemize} 
\item[i)] An upper semicontinuous function $u:\R^n\times (0, T)\to \R$ 
is a viscosity subsolution of \eqref{levelset1} if,
for every smooth test function $\phi$ such that $u-\phi$  admits a
global  maximum at $(x,t)$, we have that
either $\partial_t \phi(x,t)\leq 0$ if $D\phi(x,t)=0$, or 
\[ 
\partial_t \phi(x,t)+|D\phi(x,t)| H^K_{\{y| \phi(y,t)\geq \phi(x,t)\}}(x)\leq 0  \]
if $D\phi(x,t)\neq 0$.
\item[ii)]
A lower semicontinuous function $u:\R^n\times (0, T)\to \R$ 
is a viscosity supersolution of \eqref{levelset1} if, for every smooth test
function $\phi$ such that $u-\phi$  admits a global minimum at $(x,t)$, we have
that
either $\partial_t \phi(x,t)\geq 0$ if $D\phi(x,t)=0$, or 
\[ 
\partial_t \phi(x,t)+|D\phi(x,t)| H^K_{\{y| \phi(y,t)> \phi(x,t)\}}(x)\geq 0  \]
if $D\phi(x,t)\neq 0$.
\item[iii)] A  continuous  function $u:\R^n\times (0, T)\to \R$  is a
solution to \eqref{levelset1} if it is both a subsolution and a supersolution. 
\end{itemize} 
 \end{definition}
  
\begin{remark}\upshape 
It is easy to verify that any smooth subsolution (respectively supersolution) is in particular a viscosity subsolution (respectively supersolution).
\end{remark} 

Now,
we recall the Comparison Principle and the existence and uniqueness results for
viscosity solutions to~\eqref{levelset1}. 

\begin{theorem}\label{comparison} 
Suppose that $u_0$ is a bounded and uniformly  continuous function.  
Let $u$ (respectively $v$) be a bounded viscosity subsolution (respectively supersolution) of \eqref{levelset1}.  
If $u(x,0)\leq  u_0(x) \leq v(x,0)$ for any~$x\in\R^n$, then $u\leq v$ on $\R^n\times [0, +\infty)$.

In particular, there exists a unique continuous viscosity solution $u$
to \eqref{levelset1} such that $u(x,0)=u_0(x)$ for any~$x\in\R^n$. 

Moreover if $u_0$ is Lipschitz continuous then 
$u(\cdot, t)$ is Lipschitz continuous, uniformly with respect to $t$,
and
$$ |u(x,t)-u(y,t)| \leq \|Du_0\|_{\infty}|x-y|,$$
for all $x,y\in\R^n$ and $t>0$. \end{theorem} 
\begin{proof} 
For the  proof of the existence and uniqueness result,
and for the Comparison Principle,
we refer to  \cite[Theorems  2 and 3]{MR2487027}, see also \cite{MR3401008}. 

Finally  the Lipschitz continuity
is a consequence of the Comparison Principle. Indeed,
for any~$h\in\R^n$, we define $$
v_{\pm}(x, t): =u(x+h, t) \pm \|Du_0\|_{\infty} |h|.$$ 
Then, if $u$ is a viscosity solution to 
\eqref{levelset1}, we have that also $v_+$ and~$v_-$
are  viscosity solutions to the same equation.
Moreover,
$$v_-(x,0)=u_0(x+h)-\|Du_0\|_{\infty} |h|\leq u_0(x)=u(x,0)\leq 
u_0(x+h)+\|Du_0\|_{\infty} |h|= v_+(x,0) ,$$
which implies the desired Lipschitz bound. 
\end{proof} 

\begin{remark}\upshape \label{geosigma}
 Let $E\subset\R^n$ be a closed set in $\R^n$ and let $u_E(x)$ be
 a bounded Lipschitz continuous function such that 
\begin{equation}\label{ic} \partial E= \{x\in\R^n\text{ s.t. } u_E(x)=0\}
=\partial \{x\in\R^n\text{ s.t. }  u_E(x)>0\}\qquad{\mbox{and}}\qquad
E= \{x\in\R^n\text{ s.t. }  u_E(x)\geq 0\}. \end{equation} 
Let $u_E$ be the unique viscosity solution to \eqref{levelset1}
with initial datum $u_E$ and define 
\[E^+(t):=\{x\in\R^n\text{ s.t. } u_E(x,t)\geq 0\}\qquad{\mbox{and}}
\qquad E^-(t):=\{x\in\R^n\text{ s.t. }   u_E(x,t)> 0\}. \]
The level set flow is defined as $$
\Sigma_E(t)=\{x\in\R^n\text{ s.t. } |\ u_E(x,t)=0\}.$$ 
Due to  the fact that the operator in \eqref{levelset1} is geometric, which
means that if $u$ is a subsolution (resp. a supersolution) then also $f(u)$ is a subsolution (resp. a supersolution) for all monotone increasing functions $f$,   the following result holds:
if~$v_0$ is
a Lipschitz continuous function which satisfies \eqref{ic} and $v$ 
is the viscosity solution to \eqref{levelset1} with
initial datum $v_0$,
then  $E^+(t)=\{x\in\R^n\text{ s.t. }  v(x,t)\geq 0\}$ and
$E^-(t)=\{x\in\R^n\text{ s.t. } v(x,t)> 0\}$. 

In particular, the inner flow, the outer flow and the level set flow do not depend on the choice of
the initial datum $u_E$ but only on the set $E$.
\end{remark} 

\begin{remark}\label{palla2}  In the setting of Remark \ref{palla},
one can show that $u(x,t)=R(t)-|x|$, for $t\in [0, T_R)$,
is a viscosity solution to \eqref{levelset1}.  
Therefore, in this case
we have that $E= B_R$ and~$E^+(t)=\overline{E^-(t)}=B_{R(t)}$.
\end{remark} 

An important consequence of the Comparison Principle stated
in Theorem~\ref{comparison}, is the following result (in which we also use
the notation for the distance function introduced in~\eqref{defdef}
and~\eqref{signed}).

\begin{corollary} \label{cpgeometric}$\,$ \begin{itemize}\item[i)] 
Let $F\subset E$ two closed sets  in $\R^n$ such that  
$d(F,E)=\delta>0$. Then $F^{+}(t)\subset E^{-}(t)$ for all $t>0$,
and the map  $t\to d(F^{+}(t), E^{-}(t))$ is
nondecreasing. 
\item[ii)] Let $v:\R^n\times[0, T)\to \R$ be a bounded uniformly continuous
viscosity supersolution to \eqref{levelset1}, and assume that $$
F\subseteq \{x\in\R^n\text{ s.t. } v(x,0)\geq 0\}.$$
Then $$F^+(t)\subseteq \{x\in\R^n\text{ s.t. } v(x,t)\geq 0\},$$ for all $t\in (0, T)$. 

Moreover, if $$d(F, \;\{x\in\R^n\text{ s.t. } v(x,0)> 0\})=\delta>0,$$ then $$
F^+(t)\subseteq \{x\in\R^n\text{ s.t. } v(x,t)>0\},$$ for all $t\in (0, T)$, and 
$$d\Big(F^{+}(t),   \{x\in\R^n\text{ s.t. } v(x,t)>0\}\Big)\geq \delta .$$ 
\item[iii)]  Let $w:\R^n\times[0, T)\to \R$
be a bounded  uniformly continuous viscosity subsolution to \eqref{levelset1},
and assume that $$E\supseteq \{x\in\R^n\text{ s.t. } w(x,0)\geq 0\}).$$
Then $$E^+(t)\supseteq \{x\in\R^n\text{ s.t. } w(x,t)\geq 0\},$$ for all $t\in (0, T)$. 

Moreover, if $$d(E,\{x\in\R^n\text{ s.t. } w(x,0)\geq 0\})=\delta>0,$$
then $$E^{-}(t)\supseteq \{x\in\R^n\text{ s.t. } w(x,t)\geq 0\},$$ for all~$t
\in(0,T)$, and 
$$d\Big(E^{-}(t), \;\{x\in\R^n\text{ s.t. } w(x,t)\geq 0\}\Big)\geq \delta.$$ 
\end{itemize} 
\end{corollary} 
\begin{proof}
First, we prove i). Since $F\subseteq E$ and $d(E, F)=\delta$,
then it is easy to check that $d_E(x)\geq d_F(x)+\delta$. 

Let now~$C>2 \delta$ and define $$u_F(x):
= \max \big\{ -C-\delta, \; \min\{d_F(x), C-\delta\}\big\}
\qquad{\mbox{and}}\qquad u_E(x):= \max \big\{
-C, \;\min\{d_E(x), C\}\big\}.$$ So, again we obtain that~$u_F(x)+\delta\leq u_E(x)$. 
Therefore,
by the Comparison Principle in
Theorem \ref{comparison}, we get that $u_E(x,t)\geq u_F(x,t)+\delta$
for every $t>0$. 

This in turn implies that
$F^+(t)\subset E^{-}(t)$ and moreover that 
$d(F^+(t), E^-(t))\geq \delta$, due to the fact that  $u_E(x,t)$ and $u_F(x,t)$
are $1$-Lipschitz in $x$, by Theorem \ref{comparison}.  

If we repeat the same argument with initial data $E^{-}(t)$ and $F^{+}(t)$,
we obtain the desired statement in i). 

We prove now ii).
For this, we distinguish two cases:
if
\begin{equation}\label{9iwdkc83837yryff}
d(F,  \{x\in\R^n\text{ s.t. } v(x,0)>0\})=\delta>0,\end{equation}
we let $$E:=\overline{ \{x\in\R^n\text{ s.t. } v(x,0)>0\}}.$$
Then, by item i), we get that $F^+(t)\subseteq E^{-}(t)$ and
$d(F^{+}(t), E^{-}(t))\geq \delta$. Let $u_E$ be the unique viscosity
solution to \eqref{levelset1} with $u_E(x,0)=v(x,0)$. Then, by the
Comparison Principle
in Theorem \ref{comparison}, we get that $u_{E}(x,t)\leq v(x,t)$ for all $t\in (0, T)$. 
In turn, this implies that $E^{-}(t)\subseteq   \{x\in\R^n\text{ s.t. } v(x,t)>0\}$,
and this permits to conclude that ii) holds true, under the assumption in~\eqref{9iwdkc83837yryff}.

If, on the other hand, we have that~\eqref{9iwdkc83837yryff} does not hold,
we write
$$d(F,  \{x\in\R^n\text{ s.t. } v(x,0)>0\})\geq 0.$$ Then,
by the uniform continuity of $v(\cdot,0)$,
we have that for every $\eps>0$ there exists $\delta_\eps>0$ such that $$
d(F,\{x\in\R^n\text{ s.t. } v(x,0)> -\eps\})\geq \delta_\eps>0.$$
So we repeat the argument above (based on~\eqref{9iwdkc83837yryff})
substituting $v(x,t)$ with the function $v(x,t)+\eps$ and $E$ with $\overline{ \{x\in\R^n\text{ s.t. } v(x,0)>-\eps\}}$. 
This gives that  $F^+(t)\subseteq E^{-}(t)$,  and  $E^{-}(t)
\subseteq   \{x\in\R^n\text{ s.t. } v(x,t)>-\eps\}$ for all $\eps>0$.
Therefore $F^+(t)\subseteq  \{x\in\R^n\text{ s.t. } v(x,t)\geq 0\}$. 

This completes the proof of ii).
The proof of iii) is completely analogous, and we omit it. 
\end{proof} 

\begin{remark}\upshape\label{compact} 
Observe that if $E$ is a compact set and in particular $E\subseteq B_R$
for some $R>0$, then by Remark \ref{palla2} and Corollary \ref{cpgeometric} we have that 
$E^{+}(t)\subseteq B_{R(t)}$ where $R(t)<R$ has been defined in Remark \ref{palla}. In particular, there exists $T_E\leq T_R$ such that 
$T_E=\sup\{t>0  ${\mbox{ s.t. }}int$ \,E^+(t)\neq \varnothing\}$. 
\end{remark}

Now, we define the lower and upper semicontinuous envelopes of a family of sets $C(t)\subseteq \R^n$ as follows:
\[C_\star(t):=\bigcup_{\eps>0}\;\;\bigcap_{0\leq t-\eps<s<t+\eps} C(s)\qquad{\mbox{and}}\qquad
C^\star(t):=\bigcap_{\eps>0}\;\;\bigcup_{0\leq t-\eps<s<t+\eps} C(s).\]
We have that $C_\star(t)\subseteq C(t)\subseteq C^\star(t)$. Moreover for any sequence $(x_n, t_n)\to (x,t)$, if $x_n\in \overline{C^\star(t_n)}$ then $x\in \overline{C^\star(t)}$, 
whereas, if $x_n\not \in\,$int$\,(C_\star(t_n))$, then $x\not \in\,$int$\,(C_\star(t))$.

If $C^\star(t)=C(t)=C_\star(t)$ for every $t$, we say that the family is continuous. 

We also
need a result to compare geometric sub and supersolutions to \eqref{kflow} with the level set flow, see \cite{MR3401008}. 

\begin{proposition}\label{subgeometrico}
Let $C(t)\subseteq \R^n$ for $t\in [0,T]$, be a continuous family of
sets with compact Lipschitz boundaries, which are
piecewise of class $C^{1,1}$
outside a finite number of angular\footnote{As customary,
a point of a piecewise~$C^{1,1}$ curve is called ``angular''
if the tangent directions from  different sides are different.} points. 

Fix $E\subset\R^2$ and $u_E$ a bounded Lipschitz
continuous function such that $$E=\{x\in\R^n\text{ s.t. }u_E(x)\geq 0\}
\qquad{\mbox{and}}\qquad\partial E=\{x\in\R^n\text{ s.t. }u_E(x)=0\}.$$
Consider the inner and outer flows associated to $E$, according to \eqref{outin}. 

\begin{itemize}\item[i)]
Assume  that there exists $\delta>0$ such that at every~$
x\in \partial C(t)$ where $\partial C(t)$
is~$C^{1,1}$ there holds 
\begin{equation}\label{supergeo} 
\partial_t x\cdot \nu(x)\geq -H^K_{C(t)}(x)+\delta.\end{equation}
Moreover, assume that
\begin{equation}\label{supergeo2} \begin{split}
&{\mbox{at every angular point $x\in \partial C(t)$
there exists 
$r_0>0$ such that}}\\
&{\mbox{the set $B(x,r)\cap C(t)$  is convex for all $r<r_0$. }}\end{split}\end{equation}
Then, if~$E  \subseteq C(0)$, with $d(E, C(0))=k\geq 0$, it holds that  $E^+(t)\subset C(t)$ for all $ t\in [0,T)$, with  $d(E^+(t), C(t))\geq k\geq 0$.
\item[ii)]  Assume  that there exists $\delta>0$ such that
at every $x\in \partial C(t)$ where $\partial C(t)$ is~$C^{1,1}$
it holds 
\begin{equation}\label{subgeo} \partial_t x\cdot \nu(x)\leq -H^K_{C(t)}(x)-\delta.\end{equation}
Moreover, assume that
\begin{equation}\label{subgeo2} \begin{split}&
{\mbox{at every angular point $x\in \partial C(t)$
there exists 
$r_0>0$ such that}}\\
&{\mbox{
the set $B(x,r)\cap (\R^n\setminus C(t))$  is convex for all $r<r_0$. }}\end{split}\end{equation}
Then, if  $E \supseteq C(0)$,  it holds that  $E^+(t)\supseteq C(t)$ for all $ t\in[0, T)$. 

Moreover, if
$d(C(0), \{x\in\R^n\text{ s.t. }u_E(x)> 0\})=k>0$, it holds that  $E^-(t)\supset C(t)$ for all $ t\in[0,T)$, with  $d(E^-(t), C(t))\geq k$.
\end {itemize} 
\end{proposition}

\begin{proof} We give just a sketch of the proof of i), since it relies on classical arguments in viscosity solution theory and level set methods (the proof  of ii) is analogous), 
see \cite{MR3401008}.

For~$\eps>0$ sufficiently small,
we define the function $$u_\eps(x,t):= \max\big\{ 0, \;\min\{\eps, d_{C(t)}(x)\}\big\}.$$  
We claim that for $\eps>0$ sufficiently small (depending on $\delta$ in \eqref{supergeo}) the function $u_\eps$ is a viscosity supersolution to \eqref{levelset1}. If the claim is true, then
the statement in~i) is a direct consequence of the Comparison Principle in
Corollary \ref{cpgeometric}. 

To prove the claim, for  every $\lambda\in [0, \eps]$, we define
$$C_\lambda(t):=\{x\in C(t)\text{ s.t. }
d_{C(t)}(x)\geq \lambda\}.$$ 
Note that $u_\eps=0$ on $\overline{\R^n\setminus C(t)}$,
$u_\eps=\lambda$ on  $\partial C_\lambda(t)$ and~$u_\eps=\eps$ on $C_\eps(t)$. 

Due to the regularity assumption on $C(t)$, we have that for every $
\lambda\in [0, \eps]$, the sets $C_\lambda(t)$ are Lipschitz
continuous,
piecewise $C^{1,1}$ outside a finite number of angular points
and satisfy the following property:
at every angular point $x\in \partial C_\lambda(t)$
there exists 
$r_0>0$ such that  
the set $B(x,r)\cap C_\lambda(t)$  is convex for all $r<r_0$. Therefore assumption \eqref{supergeo2} is satisfied for every $C_\lambda$, with $ \lambda\in [0, \eps]$. 

Now we observe that, due to the regularity assumptions
and to \eqref{supergeo2}, we have that
for every  $x_\eps\in  \partial C_\eps(t)$ there exists   $x_0\in \partial C (t)$ 
such that  $|x_0-x_\eps|=\eps$ ($x_0$ is unique if  $\partial C_\eps(t)$ is $C^{1,1}$ at $x_\eps$, and it is eventually non unique if $x_\eps$ is an angular point).  Moreover $\partial C(t)$ is $C^{1,1}$ around $x_0$. 

Assume first that~$x_\eps$ is an angular point of~$
\partial C_\eps(t)$.
We fix $\zeta(\eps, x_\eps, t)=\zeta_\eps>0$
such that~$\partial C_\eps(t)$ is $C^{1,1}$ 
at every $x\in   B(x_\eps, \zeta_\eps )\cap \partial C_\eps(t)$, $x\neq x_\eps$,   $\partial C(t)$ is  $C^{1,1}$ 
at every $x\in   B(x_0, \zeta_\eps )\cap \partial C  (t)$  (so that the $K$-curvature is well defined) and moreover there holds 
\[
H^K_{\partial C_\eps(t)}(x)>\sup_{y\in B(x_0, \zeta_\eps) \cap
\partial C(t)} H^K_{\partial C(t)}(y)\qquad {\mbox{for all }}
x\in B(x_\eps, \zeta_\eps )\cap \partial C_\eps(t),\;
x\neq x_\eps, {\mbox{ and for all }} t\in [0, T]. \]
Since the  angular points $x_\eps$ of $\partial C_\eps(t)$  are finite for every $t\in [0,T]$, and the interval $[0, T]$ is compact, 
we can choose $\zeta_\eps$ independent of $x_\eps$ and $t$. 
Now consider the case in which $\partial C_\eps(t)\cap B(x_\eps, \zeta_\eps)$ is $C^{1,1}$. 
Then we use the  continuity of the $K$-curvature as
$\eps\to 0$ (see \cite{MR3401008}) to see that
there exists $\eta_\eps=\eta(\eps, x_\eps, \zeta_\eps, t)>0$ such that 
\[ |H^K_{C(t)}(x_0)-  H^K_{C_\eps(t)}(x_\eps)|\leq 
 \eta_\eps.\]
Finally, due to the compactness of
$$\partial C_\eps(t)\setminus \bigcup_{i\in I} B(x_i, \zeta_\eps),$$
where $x_i$ are the angular points of $\partial C_\eps(t)$, and due to compactness of 
the time interval $[0,T]$, we observe that we may
choose~$\eta_\eps=\eta(\eps, \zeta_\eps)$ independent of $x_\eps$ and~$ t$. 
In conclusion we get that there exists $\eta_\eps>0$ depending on $\eps$ such that for all $x_\eps\in\partial C_\eps(t)$ which are not angular points there holds 
\begin{equation}\label{reg} H^K_{C_\eps(t)}(x_\eps)\geq H^K_{C(t)}(x_0)-\eta_\eps\qquad \text{ where $|x_0-x_\eps|=\eps$.}\end{equation}
  
The same argument can be repeated for all $\lambda\in (0, \eps)$, and so for every $\lambda$ there exists $\eta_\lambda>0$ such that
\eqref{reg} holds.  We define 
 \begin{equation}\label{eta}\eta=\eta(\eps)=\sup_{\lambda\in(0, \eps]}\eta_\lambda.\end{equation}

Now we distinguish different cases according to the position of
the point~$x$, in order to prove that  $u_\eps$ is a viscosity supersolution to \eqref{levelset1}.  

If  $x\in \,$int$\,(\R^n\setminus C(t))$, or  $x\in$ int$(C_\eps(t))$,
then actually the equation in~\eqref{levelset1}
is trivially satisfied since $|D u_\eps(x,t)|=0$
and $\partial_t u_\eps(x,t)= 0$ by the continuity properties of
the families $C(t)$ and $C_\eps(t)$. 

Now we suppose that~$x\in \partial C_\eps(t)$. 
Then it is easy to show that the set of test functions is empty, so again
the equation in~\eqref{levelset1} is trivially satisfied. 

We finally assume that $x \in \partial C_\lambda(t)$ for
some $\lambda\in [0,\eps)$. Observe that 
at every angular point $x  \in \partial C_\lambda(t)$, by the   assumption \eqref{supergeo2}
(which holds also for $C_\lambda(t)$ as proved above), the set of test functions is empty 
so the equation in~\eqref{levelset1} is trivially satisfied.
So assume that $C_\lambda(t)$ is locally of
class $C^{1,1}$ around $x$. We fix $x_0\in \partial C(t)$
such that $|x-x_0|=\lambda$. So, if  $\nu(x)$ is the outer normal to $\partial C_\lambda(t)$ at $x$, then $\nu(x)=\frac{x_0-x}{|x-x_0|}$ and $\nu(x)=\nu(x_0)$, so it coincides with the outer normal  to $\partial C(t)$ at $x_0$ and $ \partial_t x_0\cdot \nu(x_0)= \partial_t x\cdot \nu(x)$.
Moreover, due to \eqref{reg}, and the definition of $\eta$ in \eqref{eta},we get 
\begin{equation}\label{test1}H^K_{C_\lambda(t)}(x) \geq H^K_{C (t)} (x_0) -\eta.\end{equation} 
Let $\phi$ be a test function for $u_\eps$ at $(x,t)$,
then~$D\phi(x,t)=-\rho\nu(x) $ for
some~$\rho\in [0,1]$ for $\lambda=0$
and $D\phi(x,t)=-\nu(x)$ for $\lambda>0$,
whereas $\phi_t(x,t)=\rho  \partial_t x\cdot \nu(x)$
(with $\rho=1$ as $\lambda>0$). 
%
Moreover 
\begin{equation}\label{test}
H^K_{C_\lambda(t)}(x)= H^K_{\{y|u_\eps(y, t)\geq \lambda\}} (x)
\leq H^K_{\{y|\phi(y, t)>\lambda\}} (x).
\end{equation}
Therefore, computing the equation at $(x,t)$, we get, using \eqref{test1}, \eqref{test} and  \eqref{supergeo},  
\begin{eqnarray*}
&&\partial_t \phi(x,t)+|D\phi(x,t)|  H^K_{\{y|\phi(y, t)
>\phi(x,t)\}} (x)\geq \rho \partial_t x\cdot \nu(x)
+\rho H^K_{C_\lambda(t)} (x)\\&&\qquad\quad \geq   
 \rho \partial_t x_0\cdot \nu(x_0)
 +\rho H^K_{C(t)} (x_0)-\rho \eta\geq \rho(\delta-\eta).  \end{eqnarray*}
So, if we choose $\eps>0$ sufficiently small, according to $\delta$,  so that $\eta=\eta(\eps)\leq \delta$, then the previous inequality gives that 
$u_\eps$ is a supersolution to \eqref{levelset1}, as we claimed. 
 \end{proof} 

Now we present
the following extension to the noncompact case of
Proposition \ref{subgeometrico}.

\begin{corollary} \label{unboundedcoro} 
Let $C(t)\subseteq \R^n$ for $t\in [0, T)$, be a continuous family of
sets with Lipschitz boundaries, which are piecewise of class $C^{1,1}$
outside a finite number of angular points, and
such that there exists 
$R>0$ such that~$C(t)\cap (\R^n\setminus B_R)$ is
of class $C^{1,1}$ for all $t$. 

Fix $E\subset\R^n$ and $u_E$ a bounded Lipschitz continuous function such
that $$E=\{x\in\R^n\text{ s.t. }u_E(x)\geq 0\}\qquad{\mbox{and}}\qquad
\partial E=\{x\in\R^n\text{ s.t. }u_E(x)=0\},$$
and consider the inner and outer flows associated to $E$, according to \eqref{outin}. 

\begin{itemize}\item[i)]
Assume that there exists $\delta>0$ such that \eqref{supergeo} 
holds for every $x\in \partial C(t)\cap B_{4R}$. Suppose also
that~\eqref{supergeo2} holds true.

Moreover, assume that 
there exists $\lambda_0$  such that, for
all $\lambda\in [0, \lambda_0]$,
it holds that 
\begin{equation}\label{zero} \partial_t x\cdot \nu(x)\geq  -H^K_{C_\lambda(t)}(x) \end{equation}  
for all $x\in \partial C_\lambda(t)\cap (\R^n\setminus B_{2R})$,
where  $$C_\lambda(t):=\{x\in C(t)\text{ s.t. }
d_{C(t)}(x)\geq \lambda\}.$$ 
Then, if  $E  \subset C(0)$, with $d(E, C(0))=k\geq 0$, it holds that  $E^+(t)\subset C(t)$ for all $ t>0$, with  $d(E^+(t), C(t))\geq k\geq 0$.

\item[ii)] Assume that there exists  $\delta>0$ such that~\eqref{subgeo}
holds for every $x\in \partial C(t)\cap B_{4R}$.
Suppose also
that~\eqref{subgeo2} holds true.

Moreover, assume that 
there exists $\lambda_0$  such that for    all $\lambda\in [0, \lambda_0]$,   it holds that 
\begin{equation}\label{zero1} \partial_t x\cdot \nu(x)\leq  -H^K_{C^\lambda(t)}(x) \end{equation}  
for all $x\in \partial C^\lambda(t)\cap (\R^n\setminus B_{2R})$ where 
$$C^\lambda(t):=\{x\in \R^n\text{ s.t. }
d_{C(t)}\geq- \lambda\}.$$ 
Then, if  $E \supseteq C(0)$,  it holds that  $E^+(t)\supseteq C(t)$ for all $ t>0$.

In addition, if
$d(C(0), \{x\in\R^n\text{ s.t. }u_E(x)> 0\})=k>0$, it holds that  $E^-(t)\supset C(t)$ for all $ t>0$, with  $d(E^-(t), C(t))\geq k$.
\end {itemize} 
\end{corollary} 

The proof of Corollary~\ref{unboundedcoro}
is similar to that of Proposition \ref{subgeometrico}, and we omit the details.

We also have the following semicontinuity
type result for the outer evolutions. 

\begin{proposition}\label{scs}   There holds \begin{equation}\label{PUNT2}
\liminf_{\eta\searrow 0} 
\left| E^+(t+\eta)\right|\geq | {\rm{int}} \,E^+(t)|. 
\end{equation}
\end{proposition} 
\begin{proof}
We claim that for any fixed~$t>0$ and a.e. in~$\R^n$,
\begin{equation}\label{PUNT}
\liminf_{\eta\searrow 0}\chi_{ \left\{ u_E\left(\cdot, t+\eta\right)
\ge0\right\} }\ge\chi_{ 
\left\{ {\rm{int}}(\{u_E\left(\cdot, t\right)
\ge0\})\right\} }.
\end{equation}
To show \eqref{PUNT}, it
is enough to consider a point $x\in {\rm{int}}(\{u_E\left(\cdot, t\right)\ge0\})$, so that 
$\{u_E\left(\cdot, t\right)\ge0\}\supset B_r(x)$ for some $r>0$. 
Then, recalling formula~\eqref{evball} in Remark \ref{palla}, we have that 
$C(r(\eta))=C(r)-\eta$, for 
$\eta\in \left(0, T_r\right)$, where $T_r>0$ is the extinction time of the ball $B_r$ under the flow \eqref{kflow}. 
 
Hence, by  Remark \ref{palla2} and Corollary \ref{cpgeometric} we get 
\[
\{u_E\left(\cdot, t+\eta\right)\ge0\}\supset B_{r(\eta)}(x),
\qquad \text{for all }\eta\in \left(0,T_r\right).
\]
In particular, it follows that 
\[
\liminf_{\eta \searrow 0} u_E(x,t+\eta)\ge 0,\qquad \text{for all }x\in  {\rm{int}}
\{u_E
\left(\cdot, t\right)\ge0\},
\] 
which implies \eqref{PUNT}.

Then, by~\eqref{PUNT} and the Fatou Lemma, for all $t>0$ we obtain
\begin{eqnarray*} &&
\liminf_{\eta\searrow 0}
\left|
E^+(t+\eta)\right|=\liminf_{\eta\searrow 0} \int_{\R^n} \chi_{ \left\{ u_E\left(\cdot, t+\eta\right)
\ge0\right\} }(x)\,dx\\ &&\qquad\ge \int_{\R^n} \liminf_{\eta\searrow 0}\chi_{ \left\{ u_E\left(\cdot, t+\eta\right)
\ge0\right\} }(x)\,dx\ge \big|
{\rm{int}} \left(\left\{ u_E\left(\cdot, t\right)
\ge0\right\} \right)
\big|= \big| \text{int } E^+(t)\big|,
\end{eqnarray*}
establishing~\eqref{PUNT2}.
\end{proof} 

In the case of homogeneous kernels, i.e. under the assumption in~\eqref{NUCLEONORMA},
the geometric flow possesses a useful time scaling property as follows. 
\begin{lemma}\label{A-102-B} Assume that $K(x)=\frac{1}{|x|^{n+s}}$
for some $s\in (0,1)$. 
Let~$\lambda>0$, $M>0$, 
$E\subseteq\R^n$
and~$u_{E,\lambda}(x,t)$ be the viscosity solution
to \eqref{levelset1} 
with initial condition given by $$u_{E, \lambda}(x):=
\max\big\{-\lambda M, \;\min\{d_{\lambda E}(x),\lambda M\}\big\}.$$ 
Let also~$E_\lambda^+(t):=\{x\in\R^n {\mbox{ s.t. }} u_{E,\lambda}(x,t)\ge0\}$
and~$E_\lambda^-(t):=\{x\in\R^n {\mbox{ s.t. }} u_{E,\lambda}(x,t)>0\}$. Then
$$ E_\lambda^\pm(t)=\lambda E_1^\pm \left(\frac{t}{\lambda^{1+s}}\right).$$
\end{lemma}

\begin{proof}   For every $x\in \R^n$ such that $-M\leq d_E(x)\leq M$,
we have that 
$$\lambda u_{E,1}(x)=\lambda d_E(x)= d_{\lambda E} (\lambda x)= 
u_{E,\lambda }(\lambda x).$$  
Moreover if $d_E(x)\geq  M$, then $\lambda u_{E,1}(x)=\lambda M= u_{E,\lambda }(\lambda x)$, and analogously for $d_E(x)\leq M$. 
Therefore we get that $\lambda u_{E,1}\left(\frac{x}{\lambda}\right)=u_{E,\lambda}(x)$.
Moreover, by the scaling properties of $K$, we have that
$H^s_E(x)=\lambda^{-s} H^s_{\lambda E}(\lambda x)$.  Therefore the function $ \lambda u_{E,1}\left(\frac{x}{\lambda}, \frac{t}{\lambda^{1+s}}\right)$
is a viscosity solution to \eqref{levelset1}, with initial datum $u_{E,\lambda}(x)$. 
By the uniqueness of viscosity solutions, given in
Theorem \ref{comparison}, we get that $\lambda u_{E,1}\left(\frac{x}{\lambda},\frac{t}{\lambda^{1+s}}\right)=u_{E,\lambda}(x,t)$. 
{F}rom this we deduce the desired statement. 
\end{proof} 
\end{appendix}

\bigskip

\paragraph{\bf Acknowledgements} This work has been
supported by the Australian Research Council 
Discovery Project DP170104880 ``NEW -- Nonlocal Equations at Work'',
and by the University of Pisa Project PRA 2017 ``Problemi di ottimizzazione e di evoluzione in ambito variazionale". The authors are members of the INdAM-GNAMPA.


\bigskip 

\begin{bibdiv}
\begin{biblist}
 
 \bib{bss}{article}{
   author={Barles, G.},
   author={Soner, H. M.},
   author={Souganidis, P. E.},
   title={Front propagation and phase field theory},
   journal={SIAM J. Control Optim.},
   volume={31},
   date={1993},
   number={2},
   pages={439--469},
   issn={0363-0129},
   review={\MR{1205984}},
   doi={10.1137/0331021},
}

\bib{MR2208291}{article}{
   author={Bellettini, Giovanni},
   author={Caselles, Vicent},
   author={Chambolle, Antonin},
   author={Novaga, Matteo},
   title={Crystalline mean curvature flow of convex sets},
   journal={Arch. Ration. Mech. Anal.},
   volume={179},
   date={2006},
   number={1},
   pages={109--152},
   issn={0003-9527},
   review={\MR{2208291}},
   doi={10.1007/s00205-005-0387-0},
}
	
\bib{bell}{article}{
   author={Bellettini, Giovanni},
   author={Paolini, Maurizio},
   title={Two examples of fattening for the curvature flow with a driving
   force},
   journal={Atti Accad. Naz. Lincei Cl. Sci. Fis. Mat. Natur. Rend. Lincei
   (9) Mat. Appl.},
   volume={5},
   date={1994},
   number={3},
   pages={229--236},
   issn={1120-6330},
   review={\MR{1298266}},
}
	
\bib{bcl}{article}{
   author={Biton, Samuel},
   author={Cardaliaguet, Pierre},
   author={Ley, Olivier},
   title={Nonfattening condition for the generalized evolution by mean
   curvature and applications},
   journal={Interfaces Free Bound.},
   volume={10},
   date={2008},
   number={1},
   pages={1--14},
   issn={1463-9963},
   review={\MR{2383534}},
   doi={10.4171/IFB/177},
}

\bib{MR3469920}{book}{
   author={Bucur, Claudia},
   author={Valdinoci, Enrico},
   title={Nonlocal diffusion and applications},
   series={Lecture Notes of the Unione Matematica Italiana},
   volume={20},
   publisher={Springer; Unione Matematica Italiana, Bologna},
   date={2016},
   pages={xii+155},
   isbn={978-3-319-28738-6},
   isbn={978-3-319-28739-3},
   review={\MR{3469920}},
   doi={10.1007/978-3-319-28739-3},
}
	 
\bib{MR2675483}{article}{
   author={Caffarelli, Luis},
   author={Roquejoffre, Jean-Michel},
   author={Savin, Ovidiu},
   title={Nonlocal minimal surfaces},
   journal={Comm. Pure Appl. Math.},
   volume={63},
   date={2010},
   number={9},
   pages={1111--1144},
   issn={0010-3640},
   review={\MR{2675483}},
   doi={10.1002/cpa.20331},
}

\bib{cn}{article}{
   author={Cesaroni, Annalisa},
   author={Novaga, Matteo},
   title={The isoperimetric problem for nonlocal perimeters},
   journal={Discrete Contin. Dyn. Syst. Ser. S},
   volume={11},
   date={2018},
   number={3},
   pages={425--440},
   issn={1937-1632},
   review={\MR{3732175}},
   doi={10.3934/dcdss.2018023},
}

\bib{MR3401008}{article}{
   author={Chambolle, Antonin},
   author={Morini, Massimiliano},
   author={Ponsiglione, Marcello},
   title={Nonlocal curvature flows},
   journal={Arch. Ration. Mech. Anal.},
   volume={218},
   date={2015},
   number={3},
   pages={1263--1329},
   issn={0003-9527},
   review={\MR{3401008}},
   doi={10.1007/s00205-015-0880-z},
}

\bib{MR3713894}{article}{
   author={Chambolle, Antonin},
   author={Novaga, Matteo},
   author={Ruffini, Berardo},
   title={Some results on anisotropic fractional mean curvature flows},
   journal={Interfaces Free Bound.},
   volume={19},
   date={2017},
   number={3},
   pages={393--415},
   issn={1463-9963},
   review={\MR{3713894}},
   doi={10.4171/IFB/387},
}

\bib{MR1100211}{article}{
   author={Chen, Yun Gang},
   author={Giga, Yoshikazu},
   author={Goto, Shun'ichi},
   title={Uniqueness and existence of viscosity solutions of generalized
   mean curvature flow equations},
   journal={J. Differential Geom.},
   volume={33},
   date={1991},
   number={3},
   pages={749--786},
   issn={0022-040X},
   review={\MR{1100211}},
}

\bib{csv}{article}{
   author={Cinti, Eleonora},
   author={Serra, Joaquim},
   author={Valdinoci, Enrico},
   title={Quantitative flatness results and BV-estimates
   for stable nonlocal minimal surfaces},
   journal={J. Differential Geom.},
   date={2018},
}

\bib{MR3778164}{article}{
   author={Cinti, Eleonora},
   author={Sinestrari, Carlo},
   author={Valdinoci, Enrico},
   title={Neckpinch singularities in fractional mean curvature flows},
   journal={Proc. Amer. Math. Soc.},
   volume={146},
   date={2018},
   number={6},
   pages={2637--2646},
   issn={0002-9939},
   review={\MR{3778164}},
   doi={10.1090/proc/14002},
}		

\bib{MR1100206}{article}{
   author={Evans, L. C.},
   author={Spruck, J.},
   title={Motion of level sets by mean curvature. I},
   journal={J. Differential Geom.},
   volume={33},
   date={1991},
   number={3},
   pages={635--681},
   issn={0022-040X},
   review={\MR{1100206}},
}

\bib{MR2238463}{book}{
   author={Giga, Yoshikazu},
   title={Surface evolution equations. A level set approach},
   series={Monographs in Mathematics},
   volume={99},
   publisher={Birkh\"auser Verlag, Basel},
   date={2006},
   pages={xii+264},
   isbn={978-3-7643-2430-8},
   isbn={3-7643-2430-9},
   review={\MR{2238463}},
}

\bib{MR1189906}{article}{
   author={Ilmanen, Tom},
   title={Generalized flow of sets by mean curvature on a manifold},
   journal={Indiana Univ. Math. J.},
   volume={41},
   date={1992},
   number={3},
   pages={671--705},
   issn={0022-2518},
   review={\MR{1189906}},
   doi={10.1512/iumj.1992.41.41036},
}
	
\bib{MR2487027}{article}{
   author={Imbert, Cyril},
   title={Level set approach for fractional mean curvature flows},
   journal={Interfaces Free Bound.},
   volume={11},
   date={2009},
   number={1},
   pages={153--176},
   issn={1463-9963},
   review={\MR{2487027}},
   doi={10.4171/IFB/207},
}

\bib{m}{article}{
   author={Maz\cprime ya, Vladimir},
   title={Lectures on isoperimetric and isocapacitary inequalities in the
   theory of Sobolev spaces},
   conference={
      title={Heat kernels and analysis on manifolds, graphs, and metric
      spaces},
      address={Paris},
      date={2002},
   },
   book={
      series={Contemp. Math.},
      volume={338},
      publisher={Amer. Math. Soc., Providence, RI},
   },
   date={2003},
   pages={307--340},
   review={\MR{2039959}},
   doi={10.1090/conm/338/06078},
} 

\bib{SAEZ}{article}{
   author = {{S{\'a}ez}, Mariel},
   author = {Valdinoci, Enrico},
    title = {On the evolution by fractional mean curvature},
  journal = {Comm. Anal. Geom.},
   volume={27},
   date = {2019},
    number={1},
}

\bib{MR3090533}{article}{
   author={Savin, Ovidiu},
   author={Valdinoci, Enrico},
   title={Regularity of nonlocal minimal cones in dimension 2},
   journal={Calc. Var. Partial Differential Equations},
   volume={48},
   date={2013},
   number={1-2},
   pages={33--39},
   issn={0944-2669},
   review={\MR{3090533}},
   doi={10.1007/s00526-012-0539-7},
}
	
\bib{soner}{article}{
  author={Soner, Halil Mete},
   title={Motion of a set by the curvature of its boundary},
   journal={J. Differential Equations},
   volume={101},
   date={1993},
   number={2},
   pages={313--372},
   issn={0022-0396},
   review={\MR{1204331}},
   doi={10.1006/jdeq.1993.1015},
}

\end{biblist}\end{bibdiv}
\end{document}